\documentclass[11pt,a4paper]{article}
\usepackage{amsmath}
\usepackage{amsthm}
\usepackage{amssymb}
\usepackage{amsfonts}
\usepackage{mathrsfs}
\usepackage{parskip}
\usepackage{dsfont}
\usepackage{geometry}
\usepackage{bm}
\usepackage{fancyhdr}
\usepackage{hyperref}

\makeatletter
\def\thm@space@setup{%
  \thm@preskip=\parskip \thm@postskip=0pt
}
\makeatother
\theoremstyle{theorem}
\newtheorem{theorem}{THEOREM}
\newtheorem*{theorem*}{THEOREM}
\newtheorem{lemma}[theorem]{Lemma}  
\newtheorem*{lemma*}{Lemma}
\newtheorem*{prop*}{Proposition}
\newtheorem{prop}[theorem]{Proposition}

\newtheorem*{corollary*}{Corollary}

\theoremstyle{definition}
\newtheorem{definition}{Definition}

\newtheorem*{note}{Note}

\newtheorem*{remark}{Remark}

\newcommand{\E}[1]{\ensuremath{\mathbb{E} \left[#1 \right]}}
\newcommand{\Prob}[1]{\ensuremath{\mathbb{P} \left(#1 \right)}}
\newcommand{\var}[1]{\ensuremath{\mathrm{var} \left(#1 \right)}}

\newcommand{\supp}[1]{\ensuremath{\mathrm{supp} \left(#1 \right)}}
\newcommand{\Bin}[2]{\ensuremath{\mathrm{Bin}\left(#1,#2 \right)}}

\newcommand{\R}{\ensuremath{\mathbb{R}}}
\newcommand{\Z}{\ensuremath{\mathbb{Z}}}
\newcommand{\N}{\ensuremath{\mathbb{N}}}

\newcommand{\F}{\ensuremath{\mathcal{F}}}

\newcommand{\fl}[1]{\ensuremath{\lfloor #1 \rfloor}}

\renewcommand{\subset}{\subseteq}
\newcommand{\equidist}{\ensuremath{\stackrel{d}{=}}}

\begin{document}
\title{\textbf{Frozen percolation on inhomogeneous random graphs}}
\author{Dominic Yeo\thanks{
Faculty of Industrial Engineering and Management, Technion. 
\texttt{yeo@technion.ac.il}}}
\date{}
\maketitle

\begin{abstract}
Mean-field frozen percolation is a random graph-valued process, which adjusts the dynamics of the classical Erd\H{o}s--R\'enyi process with an additional mechanism to `freeze' potential giant components before they can form. It is known to exhibit self-organised criticality from a wide class of initial graphs. We show that a family of inhomogeneous random graphs with finitely-many types form a stable class under these dynamics. We study how the survival of a vertex depends on its initial type, and establish a hydrodynamic limit for the process recording surviving vertices of each type, in terms of multitype branching processes which approximate the graphs. The parameters of these branching processes are eventually critical, and their evolution in time is described by solutions to an unusual class of differential equations driven by Perron--Frobenius eigenvectors.
\end{abstract}

\renewcommand{\thefootnote}{}

\section{Introduction}

\emph{Frozen percolation} is a graph-valued Markov process. Given some base graph $G$, the process starts with all edges of $G$ initially declared closed; subsequently each edge of $G$ becomes open at constant rate. Open connected components which attain a certain threshold size get \emph{frozen}, meaning that they are removed from the graph. This process was introduced by Aldous \cite{Aldous99FP} in the setting where $G$ is the infinite binary tree, and has been studied by several authors in various lattice settings \cite{Kiss15,vandenberg12square,vandenberg12binary}.

A variant due to R\'ath \cite{RathFP} in the \emph{mean-field} setting of the complete graph on $N$ vertices instead adds edges at constant rate, and freezes components at a rate proportional to their size, and this is the version we study here. More precisely, we consider some (possibly random) initial graph $\mathcal{G}^N(0)$ on vertex set $[N]:=\{1,\ldots,N\}$ and declare its vertices to be \emph{alive}. The process $(\mathcal{G}^N(t),\, t\ge 0)$ then evolves as follows:
\begin{itemize}
\item edges between pairs of alive vertices arrive independently at rate $1/N$ (though we do not allow the same edge to arrive twice);
\item independently, any connected component $\mathcal{C}$ is removed from the graph at rate $\lambda(N)|\mathcal{C}|$, where $|\mathcal{C}|$ counts the number of vertices in $\mathcal{C}$. Once removed, a vertex is no longer alive, and never becomes alive again.
\end{itemize}

Throughout, we will consider a sequence of such processes, where the so-called \emph{lightning rate} $\lambda(N)$ satisfies the \emph{critical scaling}
\begin{equation}\label{eq:lambdascaling}
1/N\ll \lambda(N) \ll 1,
\end{equation}
as $N\rightarrow\infty$, so that, heuristically, a small component with size $\Theta(1)$ is very unlikely to be frozen, while \emph{giant components} with size $\Theta(N)$ would be `immediately' frozen.

R\'ath \cite{RathFP} shows that subject to \eqref{eq:lambdascaling}, the model exhibits \emph{self-organised criticality}, whereby from a broad class of initial configurations, the dynamics of the process drive it into a critical state and then maintain it there. This concept was introduced by Bak, Tang, and Wiesenfeld \cite{BakTangW88} in the setting of the sandpile model. In a graph-valued context, criticality is often characterised by a power-law decay of component sizes, which is in many random graph models the point of phase transition for the emergence of a giant component. Part of the motivation for the setting we study here is the opportunity to consider a different characterisation of criticality.

The focus of this paper is a setting where vertices carry an extra piece of information, a \emph{type}, which takes a value in $[k]:=\{1,\ldots,k\}$ for fixed $k\ge 1$. The structure of the initial graph $\mathcal{G}^N(0)$ depends on the types as follows. We fix a \emph{kernel} $\kappa$, a $k\times k$ non-negative symmetric matrix. Then, if vertices $v,w\in[N]$ have types $i$ and $j$, they are connected by an edge in $\mathcal{G}^N(0)$
\begin{equation}\label{eq:textGNkedges}
\text{with probability }1-\exp\left(-\frac{\kappa_{i,j}}{N}\right),\text{ independently for different pairs }\{v,w\}.
\end{equation}
We write $(\mathcal{G}^N(t),\,t\ge 0)$ for the frozen percolation process started from $\mathcal{G}^N(0)$, augmented with the types of the vertices in $\mathcal{G}^N(0)$. (Note that the type of a vertex does not change with time.)

Our main results describe how a vertex's survival depends on its type. We study
\begin{equation}\label{eq:defnpiN}\pi^N_i(t):= \frac{1}{N}\#\left\{ \text{alive vertices of type }i\text{ at time }t  \right\},\quad i\in[k],\,t\ge 0,
\end{equation}
corresponding to $\mathcal{G}^N(t)$. The following theorems give a complete description of limits of $\pi^N(\cdot)$ for a sequence of processes $(\mathcal{G}^N(\cdot))$ where the types in the initial graphs $\mathcal{G}^N(0)$ converge appropriately.

Throughout, we write $\rho(A)$ for the principal eigenvalue of a positive matrix $A\in\R_{>0}^{k\times k}$, and $\mu(A)$ for its principal left-eigenvector, normalised so that each component is positive and $||\mu(A)||_1=1$.

\begin{definition}\label{defntypeflow}
Given a matrix $A\in\R^{k\times k}$ and a vector $v\in\R^k$, we write $A\circ v$ for the matrix $(A_{i,j}v_j)_{i,j\in[k]}$.

We then say $\pi:[0,\infty)\rightarrow \R_{\ge 0}^k\backslash\{\mathbf{0}\}$  is a \emph{frozen percolation type flow} with initial kernel $\kappa$ and positive initial measure $\pi(0)$ if $\pi$ is continuous, and there exists some \emph{critical time} $t_c\ge 0$, and a continuous function $\phi:(t_c,\infty)\rightarrow\R_+$ such that:
\begin{equation}\label{eq:pregelation} \pi(t)=\pi(0),\quad t\le t_c,\end{equation}
\begin{equation}\label{eq:critcondition} \rho(\kappa(t)\circ \pi(t))=1,\quad t\ge t_c,\end{equation}
\begin{equation}\label{eq:flowDE} \frac{\mathrm{d}}{\mathrm{d}t}\pi(t) = -\mu(\kappa(t)\circ\pi(t))\phi(t),\quad t>t_c.\end{equation}
with $\kappa(t):=\kappa + t\mathbf{1}$, where $\mathbf{1}$ is the $k\times k$ matrix with all entries equal to $1$.
\end{definition}

\begin{theorem}\label{DEtheorem}We consider a kernel $\kappa$ and $\pi(0)\in(0,\infty)^k$ satisfying $||\pi(0)||_1=1$. We assume that at least one of the following holds:
\begin{itemize}
\item $\kappa$ is a \emph{strictly positive} kernel (ie with strictly positive entries), and $\rho(\kappa\circ \pi(0))\le 1$;
\item $\rho(\kappa\circ \pi(0)<1$.
\end{itemize}
Then there exists a unique frozen percolation type flow with initial kernel $\kappa$ started from distribution $\pi(0)$.
\end{theorem}



\begin{theorem}\label{flowtimelimitthm}
For any frozen percolation type flow with positive $\pi(0)$, $\lim_{t\rightarrow\infty} \frac{\pi(t)}{||\pi(t)||_1}$ exists, and is positive.
\end{theorem}

The main theorem, and the motivation for considering frozen percolation type flows, is the following.

\begin{theorem}\label{weaklimitthm}Fix $\kappa$ and $\pi(0)$ satisfying the conditions of Theorem \ref{DEtheorem}, and $\lambda:\N\rightarrow\R_+$ satisfying \eqref{eq:lambdascaling}. Let $(\mathcal{G}^N(\cdot))_{N\in\N}$ be a family of frozen percolation processes with lightning rates $\lambda(N)$, for which the vertices of the initial graphs $\mathcal{G}^N(0)$ are endowed with $k$ types, and edges given randomly by kernel $\kappa$. If the (possibly random) initial type distributions $\pi^N(0)$ satisfy $\pi^N(0)\stackrel{d}\rightarrow \pi(0)$, then the \emph{process} convergence
$$\pi^N(\cdot)\stackrel{d}\rightarrow \pi(\cdot)$$
holds in distribution as $N\rightarrow\infty$ with respect to the uniform topology on $\mathbb{D}^k([0,T])$ for each $T<\infty$, where $\pi$ is the unique frozen percolation type flow with initial kernel $\kappa$ started from distribution $\pi(0)$.
\end{theorem}

\subsection{Background}

We now place our results into context, and introduce the two main probabilistic objects in more detail.

\subsubsection{Mean-field frozen percolation}

In \cite{RathFP}, R\'ath considers the proportions of vertices which lie in components of different sizes in a family of mean-field frozen percolation processes. For $\ell\ge 1$, we let
\begin{equation}\label{eq:defnvN}v^N_\ell(t):= \frac{1}{N}\#\{\text{vertices in size $\ell$ components in }\mathcal{G}^N(t)\},\quad \ell\ge 1,\, t\ge 0,\end{equation}
and $\Phi^N(t):= \sum_{\ell\ge 1}v^N_\ell(t)$. We note that $(v^N(t),t\ge 0)$ is itself an ($\ell_1$-valued) Markov process, since the transition rates depend only on component sizes, and not on graph structure within components. Because of this, one can view mean-field frozen percolation as a graph-valued coalescent-fragmentation process with multiplicative coalescence rates and linear deletion rates \cite{MartinRath17}. It is also helpful to reinterpret the dynamics for freezing by assigning to the vertices independent exponential clocks with rate $\lambda(N)$: when a vertex's clock rings, all the vertices in its current component are frozen and removed.

The main result of \cite{RathFP} is a hydrodynamic limit for these quantities, under the assumption that the (possibly random) initial conditions satisfy $v^N_\ell(0) \stackrel{d}\rightarrow v_\ell(0)$ in $\ell_1$. Before formally stating the theorem, let us first describe the limit, which is a solution to the following version of Smoluchowki's coagulation equations \cite{Smol16} with multiplicative kernel:

\begin{equation}\label{eq:smol}\frac{\mathrm{d}}{\mathrm{d}t} v_\ell(t) = \frac{\ell}{2}\sum_{m=1}^{\ell-1} v_m(t)v_{\ell-m}(t) - \ell v_\ell(t) \sum_{m=1}^\infty v_m(t),\quad \ell\ge 1.\end{equation}

Solutions to these equations are characterised by a \emph{gelation time} $T_{\mathrm{gel}}$, before which the total mass $\Phi(t):= \sum_{\ell=1}^\infty v_\ell(t)$ is constant, and thereafter is strictly decreasing. Mathematical treatment of \eqref{eq:smol} began with McLeod \cite{McLeod62}, who demonstrated existence and uniqueness of solutions on $t\in[0,1)$ under \emph{monodisperse initial conditions}, where $v(0)=(1,0,0,\ldots)$. (This corresponds to starting from an empty graph.) McLeod's results have been improved by several authors \cite{Kokholm88,Norris99,Laurencot00}. We will use  Normand and Zambotti's recent results \cite{NormandZambotti} on global existence and uniqueness of solutions to \eqref{eq:smol}.

\begin{theorem}[\cite{NormandZambotti}, Theorem 2.2]\label{smolunique}
Whenever $\sum_{\ell\ge 1}v_\ell(0)<\infty$, there exists a unique solution to Smoluchowski's equations \eqref{eq:smol} starting from $v(0)$. For this solution, $\Phi(t)$ is uniformly continuous on $[0,\infty)$. Indeed, $\Phi(t)$ is constant on $[0,T_{\mathrm{gel}}]$ and strictly decreasing on $[T_{\mathrm{gel}},\infty)$, where
\begin{equation}\label{eq:Tgel}
T_{\mathrm{gel}} = \frac{1}{\sum\limits_{\ell\ge 1}\ell v_\ell(0)}.
\end{equation}
\end{theorem}

We now state a somewhat extended version of Theorem 1.2 of R\'ath \cite{RathFP}, concerning convergence of $(v^N(\cdot))$ towards solutions of \eqref{eq:smol}. R\'ath's original version assumes that $v(0)$ has finite support, but this may be extended by following closely an argument of Merle and Normand \cite{MerleNormand14} for a related process where components are frozen when they reach a certain threshold size. We refer the reader to Chapter 4 of the author's doctoral thesis \cite{Yeothesis} for the proof.

\begin{theorem}[\cite{Yeothesis}, Theorem 4.2]\label{FPconvtosmol}Consider a sequence $(v^N(\cdot))$ of mean-field frozen percolation processes satisfying \eqref{eq:lambdascaling}, for which $v^N(0)\stackrel{d}\rightarrow v(0)\in\ell_1$, and let $v$ be the unique solution to \eqref{eq:smol} started from $v(0)$, as given by Theorem \ref{smolunique}. Then $v^N\rightarrow v$ in distribution in $\mathbb{D}([0,\infty),\ell_1)$, with respect to the uniform topology. In particular, $\Phi^N\rightarrow \Phi$ in distribution in $\mathbb{D}([0,\infty))$, again with respect to the uniform topology.
\end{theorem}

\begin{note}
We emphasise that this convergence result does not depend on the exact asymptotic scaling of $\lambda(N)$, so long as \eqref{eq:lambdascaling} holds. R\'ath also studies other scalings for $\lambda(N)$, for which one does not observe self-organised criticality.
\end{note}


\subsubsection{Inhomogeneous random graphs}

In \eqref{eq:textGNkedges}, we described a generalisation of the Erd\H{o}s--R\'enyi random graph $G(N,p)$, where vertices have one of $k$ types, and the probability that a particular edge is present depends on the pair of types of its incident vertices. This model of \emph{inhomogeneous random graphs} (also known as the \emph{stochastic block model} \cite{HLL83}) was introduced by S\"{o}derberg \cite{Soderberg02}, and has been studied in many contexts both theoretically and in applications. Rather than attempt to survey the vast literature, we direct the reader to Abbe's recent review \cite{Abbe18} and the many references therein. We will follow closely the notation and language of Bollob\'as, Janson and Riordan \cite{BJRinhomog}, who gave the first rigorous treatment of this model, in a version with more general type-spaces.

Throughout, we fix a positive integer $k$. A \emph{graph with $k$ types} is a graph $G$ together with a \emph{type function}, $\tau: V(G)\rightarrow[k]$. A \emph{kernel} is a $k\times k$ real symmetric matrix with non-negative entries.

\begin{definition}\label{defnIRGs}
For each $N\in\N$, $p=(p_1,\ldots,p_k)\in\N_0^k$ and $\kappa$ a kernel, the \emph{inhomogeneous random graph} $G^N(p,\kappa)$ is a random graph with $k$ types defined as follows:
\begin{itemize}
\item $G^N(p,\kappa)$ has vertex set $\left\{1,2,\ldots,\sum_{i=1}^k p_i\right \}$.
\item The type function $\tau$ is chosen uniformly at random from the $\binom{\sum p_i}{p_1,\,\ldots\, ,p_k}$ functions $f:[\sum p_i]\rightarrow [k]$ such that $|f^{-1}(\{i\})|=p_i$ for each $i$.
\item Conditional on $\tau$, each edge $\{v,w\}$ (for $v\ne w \in[\sum p_i]$) is present with probability
$$1-\exp(-\kappa_{\tau(v),\tau(w)}/N),$$
independently of all other pairs.
\end{itemize}
\end{definition}



When we consider the proportions of vertices of each type, we will refer to the sets
\begin{equation}\label{eq:defnPi}\Pi_1:= \left\{\pi\in \R_{\ge 0}^k\,:\, \sum_{i\in[k]} \pi_i=1\right\},\quad \Pi_{\le 1}:= \left\{\pi\in \R_{\ge 0}^k\,:\, \sum_{i\in[k]} \pi_i\le 1\right\},\end{equation}
of probability distributions and subdistributions on $[k]$.

As is the case for many random graph models, with the canonical example being the Erd\H{o}s--R\'enyi graph $G(N,c/N)$ around $c=1$, IRGs undergo a phase transition whereby an asymptotically positive proportion of the vertices form a \emph{giant component}. Recall from Definition \ref{defntypeflow} the matrix notation $[\kappa\circ\pi]_{i,j}:= \kappa_{i,j}\pi_j$, and $\rho(A)$ the principal eigenvalue of positive matrix $A$, as given by Perron--Frobenius theory. We also write $L_1(G)$ for the size of the largest component in a graph $G$. The following result of \cite{BJRinhomog} shows that $\rho(\kappa\circ\pi)$ acts as the analogue of $c$ in $G(N,c/N)$ in this phase transition.

\begin{prop}[\cite{BJRinhomog}, Theorem 3.1]\label{BJRL1basic}
Fix a positive kernel $\kappa\in\R_+^{k\times k}$ and subdistribution $\pi\in\Pi_{\le 1}$. Suppose a sequence $p^N\in\N_0^k$ satisfies $p^N/N \rightarrow \pi$ as $N\rightarrow\infty$. Then, with high probability as $N\rightarrow\infty$,
\begin{equation}\label{eq:L1phasetrans}L_1 \left(G^N(p^N,\kappa)\right) =\begin{cases}o(N)&\quad \rho(\kappa\circ \pi)\le 1\\ \Theta(N) &\quad \rho(\kappa\circ \pi)>1.\end{cases}\end{equation}
\end{prop}
As for $G(N,c/N)$, we follow \cite{BJRinhomog} in saying that an IRG $G^N(p,\kappa)$ is \emph{subcritical} if $\rho(\kappa\circ\pi)<1$, \emph{critical} if $\rho(\kappa\circ\pi)=1$, and \emph{supercritical} if $\rho(\kappa\circ\pi)$, with $\pi=p/N$ as before.

To motivate Proposition \ref{BJRL1basic}, note that the number of type $j$ neighbours of a type $i$ vertex in $G^N(p,\kappa)$ is distributed as $\mathrm{Bin}(p_j, 1-\exp(-\kappa_{i,j}/N))$ when $j\ne i$, and $\mathrm{Bin}(p_i-1,1-\exp(-\kappa_{i,i}/N))$ when $i=j$. In both cases, when $p$ is large, this distribution is approximately $\mathrm{Poisson}([\kappa\circ \pi]_{i,j})$ and, in particular, its expectation is approximately $[\kappa\circ\pi]_{i,j}$. The authors of \cite{BJRinhomog} extend this into a precise comparison of the local structure of the IRG and a multitype branching process with Poisson offspring distributions.

As we will discuss in greater detail later as Proposition \ref{BJRL1enhanced}, the size of a giant component in the graph corresponds to survival probability of the corresponding branching process, which is controlled by $\rho(\kappa\circ\pi)$ exactly as in \eqref{eq:L1phasetrans}.

\subsubsection{Inhomogeneous frozen percolation}\label{inhomogFP}

We now define more rigorously a frozen percolation process where the initial graph is a IRG with $k$ types, which is the subject of Theorem \ref{weaklimitthm}.

\begin{definition}\label{multitypeFP}
Let $\kappa$ be a kernel, $\lambda>0$ a freezing rate, and let $p$ be a (possibly random) element of $\N_0^k$ satisfying $\sum p_i=N$. We then define $(\mathcal{G}^{N,p,\kappa,\lambda}(t),t\ge 0)$, the \emph{$k$-type frozen percolation process with index $N$} to be the mean-field FPP started from a realisation of $G^N(p,\kappa)$, along with the type function $\tau:[N]\to[k]$.

\par
We will typically suppress dependence on $\pi,\kappa$ and $\lambda=\lambda(N)$ for convenience. We associate to $\mathcal{G}^N(t)$, the vector $\pi^N(t)$ recording proportions of alive vertices of each type, as in \eqref{eq:defnpiN}, and $\Phi^N(t):= ||\pi^N(t)||_1$.
\end{definition}

Our motivation for studying this family of IRGs as the initial configuration is that they form a stable class under the dynamics of frozen percolation. We make this precise in the following statement, which is proved briefly in Section \ref{GNtdistproof}, and underpins the proof of Theorem \ref{weaklimitthm}.

\begin{prop}\label{GNtdistprop}Conditional on $(\pi^N(s),\, s\in[0,t])$, let $\hat{\mathcal{G}}^N(t)$ be obtained from $\mathcal G^{N,p,\kappa,\lambda}(t)$, by uniformly relabelling the alive vertices with labels $\{1,\ldots,N\Phi^N(t)\}$. With this conditioning, $\hat{\mathcal{G}}^N(t)$ has the same distribution as $G^N(N\pi^N(t),\kappa(t))$, on the set of graphs with $k$ types.\end{prop}

\subsection{Discussion}
\subsubsection{Motivating type flows}
Before starting the details of the proof, we justify briefly why the three conditions \eqref{eq:pregelation}, \eqref{eq:critcondition} and \eqref{eq:flowDE} are reasonable as a description of the limit of $k$-type FP processes.

An initial time-interval $[0,T_{\mathrm{gel}}]$ during which asymptotically zero mass is lost is a feature of solutions to \eqref{eq:smol}, and so \eqref{eq:pregelation} should hold, with $t_c=T_{\mathrm{gel}}$, with $T_{\mathrm{gel}}$ given by \eqref{eq:Tgel} for $v(0)$ corresponding to limits of $\mathcal{G}^N(0)$.

Thereafter, the graphs which are present during the process should be critical. In \cite{RathFP}, R\'ath characterises this via a power-law condition on the tail of $v_\ell(t)$ as $\ell\rightarrow\infty$. However, in the more specific setting of IRGs, criticality can be characterised by $\rho=1$ as in \eqref{eq:critcondition}, from which the power-law tail follows \cite{MiermontMultitype}. Note that an edge between alive vertices at time $t$ was either present in the initial graph, or was added during $[0,t]$, hence the kernel describing $\mathcal{G}^N(t)$ should be $\kappa(t)=\kappa(0)+t\mathbf{1}$, as in \eqref{eq:critcondition}.

The most interesting property is \eqref{eq:flowDE}, which describes the proportion of types amongst the mass \emph{frozen} at time $t$, which is typically not the same as the proportion of types \emph{alive} at time $t$. We motivate this as follows:
\begin{itemize}
\item A key property of FP with critical scaling is that asymptotically almost all mass is lost as a result of freezing large (but not giant) components. This can be seen from the continuity of $\Phi(t)$ in Theorem \ref{smolunique}, describing limits of the total mass process.
\item Note that $\mu(\kappa\circ\pi)$ is a fixed point for the branching operator in the Poisson multitype branching process which approximates the local structure of an IRG, as discussed after Proposition \ref{BJRL1basic}. Results about proportions of types in large realisations of the branching process can be lifted to corresponding results for IRGs. We will see that the asymptotic proportion of types in the largest components of a suitable sequence of critical IRGs is also given by the left-eigenvector $\mu(\kappa\circ\pi)$. A more precise treatment occupies much of Section \ref{largecptssection}.
\end{itemize}

Combining these observations suggests that $\mu(\kappa(t)\circ\pi(t))$ should describe the proportion of types amongst mass lost at time $t$.

\subsubsection{Positivity and irreducibility}
It is worth stating that the conditions of Theorem \ref{DEtheorem} are a technical convenience rather than a requirement. We are keen to avoid the situation that at the critical time, the kernel $\kappa(t_c)$ has \emph{reducible} block form $\begin{pmatrix}\kappa^{(1)}(t_c)&0\\0&\kappa^{(2)}(t_c)\end{pmatrix}$, as then the eigenspace corresponding to the Perron root of $\kappa(t_c)\circ\pi(t_c)$ can have dimension greater than one.

It's possible to adjust the definition of $\mu$ appropriately: the mass of $\mu$ should be split between the irreducible type sets (that is, the types making up each block in the block form) in the same proportion as $\pi$, with the (unique) principal eigenvector applying on each block. However, we feel that the extra notation required to handle this one case at time $t=t_c=0$ would be an unwelcome distraction from the central argument.


\subsubsection{Relation to mean-field forest fires, and other models}

{\bf Forest fires:} The forest fire \cite{DrosselSchwabl} is a model proposed to describe situations where occasional `disastrous events' damage the infrastructure required for future disastrous events to propagate widely through a population. See \cite{RhodesAnderson} for a discussion of the forest fire as a model for the spread of infectious disease.

Mathematically, the forest fire is a graph-valued process similar to frozen percolation. The main difference is that in the forest fire, when a vertex's clock rings, all the \emph{edges} of that vertex's component are removed (or \emph{burned}), but not the vertices themselves as in frozen percolation. As a consequence, the total number of vertices remains constant, and the system is recurrent.

R\'ath and T\'oth study the forest fire in the mean-field setting \cite{RathTothFF}, and show a result analogous to Theorem \ref{FPconvtosmol} from a broad class of initial conditions. The limit is described by a family of coupled ODEs, where the Smoluchowski equations \eqref{eq:smol} apply for $\ell \ge 2$, but an adjustment is required for $\ell=1$ to account for the input of singleton vertices through the burning dynamics. Unlike for frozen percolation, the non-monotonicity of the system makes handling both these \emph{critical forest fire equations}, and the convergence of the discrete models considerably more challenging.

For example, it is conjectured that the forest fire equations should converge to an identified \emph{stationary solution} from a broad class of initial conditions, but this remains open, even in the monodisperse case where the initial graph is empty.

However, it has been observed (see \cite{Yeothesis} \textsection5.1) that if one conditions on the sequence of \emph{ages} of the vertices (that is, how much time has elapsed since each vertex was last burned), then the graph in the forest fire model is also an IRG. Here the type space is continuous, possibly with atoms. In a forthcoming paper with Crane and R\'ath, we prove a concentration result for the empirical distribution of these ages analogous to Theorem \ref{weaklimitthm}, and show that the limit satisfies a measure-valued ODE similar to \eqref{eq:flowDE}, for which $\phi$ can be characterised explicitly.

Although there exist results concerning type-proportions in IRGs with (countably) infinitely many types \cite{deRaph14}, establishing that the measure-valued age ODE is well-posed is more taxing in this setting. However, convergence can be obtained more cheaply using results on the cluster process tracking the size of the component containing a fixed vertex, studied by Crane, Freeman and T\'oth \cite{CraneFF}. It is hoped that convergence (in time) of the limiting age distributions may be easier to handle than convergence of the component size densities.

{\bf Inhomogeneous edge addition:} In our model introduced in Definition \ref{multitypeFP}, all the inhomogeneity of the vertices is captured in the initial graph. One could equally ask about a frozen percolation process where the edge-arrival process is inhomogeneous, for example started from an empty graph. We can define such a model using exactly the same notation, but now with the condition that $\kappa(t)=t\kappa$, for $\kappa$ some fixed kernel.

A special case of this is $k=2$, and a bipartite edge-addition process corresponding to $\kappa=\begin{pmatrix}0&1\\1&0\end{pmatrix}$. Many aspects of such processes (without deletion mechanisms) have been studied in the literature \cite{ER64,Rucinski81,BollobasThomason85,Saltykov95,DevroyeMorin03}.

Our reason for studying inhomogeneous initial conditions first, rather than inhomogeneous edge-addition, is that R\'ath's results apply, especially the continuity of the limiting total mass $\Phi(t)$ in Theorem \ref{FPconvtosmol}. Since the model with inhomogeneous edge-addition is not an example of frozen percolation, this would require significant extra work. However, the arguments of Sections \ref{largecptssection} and \ref{weaklimitsection} would carry over directly to this alternative model.

It has been conjectured by Federico and R\'ath \cite{FedericoRath} that Theorem \ref{flowtimelimitthm} does not hold in this alternative model, with proposed counterexamples when $\kappa$ has some zero entries.

{\bf The configuration model:} In a master's thesis, Acz\'el \cite{Aczelthesis} studies frozen percolation with an initial graph given by Bollob\'as's configuration model \cite{Bollobas79}, which prescribes a uniform choice among graphs with given degree sequence. The key to this analysis is a version of our Proposition \ref{GNtdistprop}, in which it is shown that a family of so-called \emph{hybrid random graphs} combining the edges of a configuration model and, independently, an Erd\H{o}s--R\'enyi random graph, form a stable class under the dynamics of frozen percolation.

\subsection{Outline of paper}

In Section \ref{GNtdistproof}, we prove Proposition \ref{GNtdistprop}, that IRGs form a stable class under the dynamics of mean-field frozen percolation, which underpins the remainder of the argument.

In Section \ref{uniquesection} we show that solutions to the type flow equations are unique. We will introduce formally the multitype branching processes which approximate IRGs, as discussed informally in this introduction. For a given type flow, we study an associated process of such branching processes and show that their total progeny size distributions give solutions to the Smoluchowski equations \eqref{eq:smol}. Uniqueness of these solutions can be lifted (via $\Phi$, the total mass process) to uniqueness of type flow solutions.

We prove Theorem \ref{weaklimitthm} in Section \ref{weaklimitsection}, from which the existence result of Theorem \ref{DEtheorem} also follows. We show tightness of the family of processes $(\pi^N)$ in $\mathbb{D}^k([0,T])$, and argue that any limit must correspond to critical graphs as in \eqref{eq:critcondition}, else the continuity and monotonicity properties of $\Phi$ from Theorem \ref{smolunique} will fail.

Checking that weak limits $\pi$ satisfy \eqref{eq:flowDE} is, unsurprisingly, the most technical argument. We require precise concentration estimates for the proportion of types lost at each freezing time, which are mostly obtained in Section \ref{largecptssection} by studying the distribution of types \emph{far away from} a uniformly chosen vertex in an IRG. Then in Section \ref{integralsection} we use martingale arguments to control the accumulation of errors over all the freezing times.

Finally, we prove Theorem \ref{flowtimelimitthm} in the short Section \ref{timelimitssection}.

At various stages we require non-probabilistic results about the eigenvalues and eigenvectors of positive matrices. Some of these proofs are collected in Section \ref{technicalproofs} to avoid breaking the flow of the argument.

\subsection{Acknowledgments}
The author would like to thank Christina Goldschmidt and Bal\'azs R\'ath for reading several drafts, and productive discussions, and also Oliver Riordan for many suggested revisions, especially concerning Section \ref{ktypeBPs} and a simplification to the argument of Lemma \ref{Yxilemma}. The author was supported by EPSRC doctoral training grant EP/K503113, ISF grant 1325/14, and by the Joan and Reginald Coleman--Cohen Fund.

\section{Proof of Proposition \ref{GNtdistprop}}\label{GNtdistproof}
We restate and prove Proposition \ref{GNtdistprop}, for $\mathcal{G}^{N,p,\kappa,\lambda}(t)$ some $k$-type FPP with index $N$.

\begin{prop*}Conditional on $(\pi^N(s),\, s\in[0,t])$, let $\hat{\mathcal{G}}^N(t)$ be obtained from $\mathcal G^{N,p,\kappa,\lambda}(t)$, by uniformly relabelling the alive vertices with labels $\{1,\ldots,N\Phi^N(t)\}$. With this conditioning, $\hat{\mathcal{G}}^N(t)$ has the same distribution as $G^N(N\pi^N(t),\kappa(t))$, on the set of graphs with $k$ types.\end{prop*}

\begin{proof}
We consider a filtration that is finer than the natural filtration of $(\pi^N(t))$, but coarser than the natural filtration of $(\mathcal{G}^N(t))$. Given the frozen percolation process $\mathcal{G}^{N,p,\kappa,\lambda}$, for each $t\ge 0$, we define the sigma-algebra $\hat{\mathcal F}^N(t)$ generated by $(\pi^N(s),s\in[0,t])$, \emph{and} the types of all vertices assigned at time 0, \emph{and} $\mathcal{A}(t)$, the set of alive vertices at time $t$. We claim that conditional on $\hat{\mathcal{F}}^N(t)$, $\hat{\mathcal G}^N(t)$ has the same distribution as $G^N\left(N\pi^N(t),\kappa+t\mathbf{1}\right)$ on the set of graphs with $k$ types, which implies the statement of the proposition.

For this proof, we set $P:=\sum p_i$, and let $\eta$ denote counting measure on any discrete set. In the definition of a frozen percolation process $\mathcal{G}^{N,p,\kappa,\lambda}$, we can consider the edge-arrival process to be a Poisson point process $\mathcal{E}$ on $\binom{[P]}{2}\times [0,\infty)$, with intensity $\eta\otimes \frac{1}{N}\cdot\mathrm{Leb}$ and, independently, the \emph{lightning process} to be a PPP $\mathcal{L}$ on $[P]\times [0,\infty)$, with intensity $\eta\otimes \lambda\cdot\mathrm{Leb}$, describing the vertex clocks which initiate freezing. Given $\mathcal{G}^{N,p,\kappa,\lambda}(0)$, we can recover the whole frozen percolation process $\mathcal{G}^{N,p,\kappa,\lambda}$ on $[0,\infty)$ using $\mathcal{E}$ and $\mathcal{L}$, ignoring atoms in $\mathcal{E}$ corresponding to edges which are already present, and atoms in both $\mathcal{E}$ and $\mathcal{L}$ that correspond to already-frozen vertices. In particular, the evolution of $\mathcal{G}^{N,p,\kappa,\lambda}$ on $[0,t]$ is independent of the restrictions of $\mathcal{E}$ and $\mathcal{L}$ to $\binom{[P]}{2}\times[t,\infty]$ and $[P]\times [t,\infty]$, respectively.

\par
Let $\tau$ be the (random) type function $[P]\rightarrow[k]$ of the initially-alive vertices. Let $\mathbf{t}$ be any function $[P]\rightarrow[k]$ satisfying $|\mathbf{t}^{-1}(i)|=p_i$ for all $i\in[k]$, and let $A\subset [P]$ (which are the conditions for $\tau$). Then the event $B=\{\tau=\mathbf{t}, \,\mathcal{A}(t)=A\}$ depends precisely on
\begin{enumerate}
\item the types in $\mathcal{G}^{N,p,\kappa,\lambda}(0)$;
\item the restriction of the edge set of $\mathcal{G}^{N,p,\kappa,\lambda}(0)$ to $\binom{[P]}{2}\backslash \binom{[A]}{2}$;
\item the restriction of $\mathcal{E}$ to $\left(\binom{[P]}{2}\backslash \binom{[A]}{2}\right)\times [0,t]$;
\item the restriction of $\mathcal{L}$ to $[P]\times [0,t]$.
\end{enumerate}
Furthermore, conditional on $\{\mathrm{type}=\mathbf{t}, \, \mathcal{A}(t)=A\}$, the restriction of $\pi^N$ to $[0,t]$ also depends only on these four structures. Therefore, $B$ is independent of the restriction of the edge set of $\mathcal{G}^{N,p,\kappa,\lambda}(0)$ to $\binom{[A]}{2}$, and the restriction of $\mathcal{E}$ to $\binom{[A]}{2}\times [0,\infty)$. Thus, conditional on $B$, $\pi^N$ is also independent of the restriction of the edge set of $\mathcal{G}^{N,p,\kappa,\lambda}(0)$ to $\binom{[A]}{2}$, and the restriction of $\mathcal{E}$ to $\binom{[A]}{2}\times [0,\infty)$.

\par
It follows that, conditional on $\hat{\mathcal{F}}^N(t)$, the vertex set of $\mathcal{G}^{N,p,\kappa,\lambda}(t)$ is $\mathcal{A}(t)$, and the types are given by the restriction of $\tau$ to $\mathcal{A}(t)$. Since $|\mathcal{A}(t)|=N||\pi^N(t)||_1$, the distribution of the latter is the uniform distribution among functions $f:\mathcal{A}(t)\rightarrow[k]$ satisfying $|f^{-1}|=N\pi^N_i(t)$ for all $i\in[k]$. With this conditioning, the presence of an edge between vertices $x,y\in\mathcal{A}(t)$ depends only on the presence of an edge between $x,y$ in $\mathcal{G}^{N,p,\kappa,\lambda}(0)$, and the restriction of $\mathcal{E}$ to $\{x,y\}\times[0,t]$, and so occurs with probability
$$1- \exp\left(-\frac{\kappa_{\tau(x),\tau(y)}}{N}\right)\cdot \exp(-t/N).$$
Furthermore, since conditional on $\tau$, edges in $\mathcal{G}^{N,p,\kappa,\lambda}(0)$ are independent, and the restrictions of $\mathcal{E}$ to different first arguments are independent, it follows that different edges between vertices in $\mathcal{A}(t)$ are independent also. Therefore, conditional on $\hat{\mathcal F}^N(t)$, after uniformly random relabelling of the vertices, $\mathcal{G}^{N,p,\kappa,\lambda}(t)$ has precisely the distribution of $G^N(N\pi^N(t),\kappa+t\mathbf{1})$ on the set of graphs on $[P]$ with $k$ types, as required.
\end{proof}
\begin{remark}
Defining the edge probabilities in $G^N(p,\kappa)$ to be $1-\exp(-\kappa_{i,j}/N)$, rather than $1\wedge \frac{\kappa_{i,j}}N$ makes this result considerably simpler.
\end{remark}

\section{Uniqueness of frozen percolation type flows}\label{uniquesection}

In this section, we prove the following proposition. 

\begin{prop}\label{uniqueprop}Consider kernel $\kappa\in\R_{\ge 0}^{k\times k}$ and $\pi(0)\in\Pi_{\le 1}$ satisfying one of the conditions in Theorem \ref{DEtheorem}. Suppose there are frozen percolation type flows $\pi,\nu$, both with initial kernel $\kappa$ started from subdistribution $\pi(0)$. Then $\pi=\nu$.
\end{prop}

The proof proceeds by constructing a solution to the Smoluchowski equations \eqref{eq:smol} from a frozen percolation type flow. We will use Theorem \ref{smolunique} to conclude that these are the same for both $\pi$ and $\nu$, and in particular, the associated $\Phi$s are the same, from which $\pi=\nu$ will follow using \eqref{eq:flowDE}.

\subsection{Bounding $\pi$ away from zero}
In the following lemma, we show that every component of $\pi(t)$ stays positive for all $t\ge 0$. This natural condition avoids the requirement for an awkward case distinction in the main argument of this section.

\begin{lemma}\label{flowsstaypositive}Any frozen percolation type flow $(\pi(t))_{t\ge 0}$, with initial kernel $\kappa\in\R_{\ge 0}^{k\times k}$ and positive initial subdistribution $\pi(0)\in \Pi_{\le 1}$, is positive for all times $t\ge 0$.

\begin{proof}
We write $\mu(t)$ as an abbreviation for $\mu(\kappa(t)\circ \pi(t))$ and $\kappa_{\max}:= \max_{i,j\in[k]}\kappa_{i,j}$. The result is clear for $t\le t_c$. Now suppose that
$$T:= \inf\{t> t_c \,:\, \exists i\in[k], \pi_i(t)=0\}<\infty.$$
Observe that $T>t_c$ since $\pi(t_c)=\pi(0)$. Then consider any $t\in[t_c,T)$. Since $\rho(\kappa(t)\circ \pi(t))=1$ and $\kappa(t)\circ \pi(t)$ is positive, the eigenvector $\mu(t)$ is well-defined, and satisfies
$$\mu_i(t) = \pi_i(t) \sum_{j=1}^k \mu_j(t) \cdot (\kappa_{j,i}+t).$$ 
So, since $\pi_i\le 1$ and $\sum \mu_j=1$, we have
$$\mu_i(t)\le \pi_i(t)\left[\kappa_{\max}+t\right],$$
So from \eqref{eq:flowDE}
$$\frac{\mathrm{d}}{\mathrm{d}t}\pi(t)\ge -\pi(t)\cdot \phi(t)[\kappa_{\max}+t].$$
Thus
\begin{align*}
\pi(t)&\ge \pi(t_c)\exp\left(-\int_{t_c}^t [\kappa_{\max}+s]\phi(s)\mathrm{d}s\right)\\
&\ge \pi(0) \exp\left(-[\kappa_{\max}+t]\right),
\end{align*}
since $\int_{t_c}^t\phi(s)\mathrm{d}s=\Phi(t_c)-\Phi(t)\le 1$. This holds for all $t\in[t_c,T)$, and thus
$$\pi(T)\ge \pi(0)\exp\left(-[\kappa_{\max}+T]\right),$$
since $\pi$ is continuous (because $\pi$ is a frozen percolation type flow).
\end{proof}
\end{lemma}

\subsection{$k$-type branching processes}\label{ktypeBPs}

Given a frozen percolation type flow, we construct a solution to the Smoluchowski equations. The hydrodynamic limit proposed by Theorem \ref{weaklimitthm} is our motivation. For a sequence of $k$-type FP processes $(\mathcal{G}^{N,p^N,\kappa})$ which approximate a type flow $\pi$, from Proposition \ref{GNtdistprop} we might conjecture that $\mathcal{G}^{N,p^N,\kappa}(t)$ and $G(N\pi(t),\kappa(t))$ have the same local limit as $N\rightarrow\infty$. We will shortly describe the local limit of IRGs in general, and use this to construct $(v_\ell(\cdot),\ell\ge 1)$ associated to a type flow $\pi(\cdot)$.

\begin{definition}\label{defnxi}
Given $\kappa\in \R^{k\times k}_{\ge 0}$ and $\pi\in\Pi_{\le 1}$, we define $\Xi^{\pi,\kappa}$ to be a multitype Galton-Watson branching tree where the vertices have $k$ types, as follows:
\begin{itemize}
\item With probability $1-\sum_{i=1}^k \pi_i$, set $\Xi^{\pi,\kappa}=\varnothing$, the empty tree.
\item For each $i$, with probability $\pi_i$, declare the root of $\Xi^{\pi,\kappa}$ to have type $i$.
\item For each $r\ge 0$, we generate the tree at generation $r+1$ recursively from the vertices at generation $r$. Each vertex with type $i$ has a Po($\kappa_{i,j}\pi_j$) number of type $j$ offspring, where these counts are independent across types $j$ and choice of parent, and the history of the process.
\end{itemize}
\end{definition}

The following result about the survival of $\Xi^{\pi,\kappa}$ is discussed by Mode \cite{Modebook}, and proved for more general type-spaces in \cite{BJRinhomog}.
\begin{prop}[\cite{BJRinhomog}, Theorem 6.1]\label{Xisurvivallemma}
We have $\Prob{|\Xi^{\pi,\kappa}|=\infty}>0$ iff $\rho(\kappa\circ\pi)>1$.
\end{prop}

As for IRGs, we say that $\Xi^{\pi,\kappa}$ is \emph{subcritical}, \emph{critical}, and \emph{supercritical} when $\rho=\rho(\kappa\circ\pi)$ satisfies $\rho<1$, $\rho=1$, and $\rho>1$, respectively.

It holds that $\Xi^{\pi,\kappa}$ are weak local limits in the sense of Benjamini and Schramm \cite{BS01} for IRGs (see \cite{vdHRGCN} for details) but the following weaker statement will suffice for us.

\begin{prop}[\cite{BJRinhomog}, Theorem 9.1]\label{BJRweaklimit}
For any graph $G$, let $\mathcal{N}_\ell(G)$ be the number of vertices of $G$ which lie in a component of size exactly $\ell$. Given a kernel $\kappa$ and $\pi\in\Pi_{\le 1}$ and any sequence $p^N\in \N_0^k$ such that $p^N/N\rightarrow\pi$,
\begin{equation}\label{eq:BJRloclim}\frac{1}{N}\mathcal{N}_\ell\left(G^N(p^N,\kappa)\right)\quad \stackrel{\mathbb{P}}\rightarrow\quad \Prob{|\Xi^{\pi,\kappa}|=\ell}.\end{equation}
\end{prop}

In Section \ref{vfrompisection}, we will define $v_\ell(t):= \Prob{|\Xi^{\pi(t),\kappa(t)}|=\ell}$, and show that this satisfies the Smoluchowski equations \eqref{eq:smol}, when $\pi$ is a FP type flow. First, we explain how to treat $\Prob{|\Xi^{\pi,\kappa}|=\ell}$ as a sum over trees. For use in the rest of this section, for any finite set $A$ we define $T_A$ to be the set of \emph{unrooted}, unordered trees, labelled by $A$, and we define $T^\rho_A$ to be the set of \emph{rooted}, unordered trees, again labelled by $A$.

\begin{lemma}\label{intoverltrees}
Let $\kappa\in\R^{k\times k}_{\ge 0}$ and $\pi\in\Pi_{\le 1}$. Then
\begin{equation}\label{eq:ProbXi=ell}\Prob{|\Xi^{\pi,\kappa}|=\ell} = \frac{1}{\ell !} \sum_{T\in T^\rho_{[\ell]}} \sum\limits_{\substack{i_1,\ldots,i_\ell\\ \in[k]}} \left[ \prod_{(m,n)\in E(T)}\kappa_{i_m,i_n} \right] \prod_{m=1}^\ell \pi_{i_m} \exp\left(-\sum_{j=1}^k \kappa_{i_m,j}\pi_j\right).\end{equation}
\begin{proof}
As motivation for some of the expressions to follow, note that $\sum \kappa_{i,j}\pi_j$ is the expected number of offspring (of all types) of a type $i$ parent in $\Xi^{\pi,\kappa}$.

To simplify the argument, we will use a slightly different construction of an inhomogeneous random graph with index $N$, where the set of vertices is also random, corresponding to the type subdistribution $\pi$. More formally, we define a random variable
\begin{equation}\label{eq:defnXforGtilde}X_1=\begin{cases}i&\quad \text{with probability}\quad\pi_i,\quad i\in[k]\\ 0&\quad \text{with probability}\quad 1-\Phi= 1-\sum_{i=1}^k \pi_i,\end{cases}\end{equation}
and let $X_2,\ldots,X_N$ be IID copies of $X_1$. We then construct a random graph $\tilde G^N(\pi,\kappa)$ conditional on $(X_1,\ldots,X_N)$ as follows. The vertex set is $M:=\{m\in[N]\, :\, X_m\ne 0\}$, and the type of any $i$ in the vertex set is $X_i$. Then, (as in the original definition of $G^N(p,\kappa)$) each edge $\{i,j\}\in M^{(2)}$ is present with probability $1-\exp\left(-\kappa_{X_i,X_j}/N\right)$, independently of all other pairs.

\par
Shortly, we will consider the quantities
\begin{equation}\label{eq:defnbarpNell}\bar p_i^{N,\ell}:= \#\Big\{ m\in\{\ell+1,N\}\,:\, X_m=i \Big\},\quad i\in[k], \quad 0\le \ell\le N-1,\end{equation}
associated with a realisation of $\tilde G^N(\pi,\kappa)$. We will consider local limits in $\tilde G^N(\pi,\kappa)$. In this setting, we say that $|C(1)|$, the size of the component containing 1, is zero if $X_1=0$, that is if $1$ is not in the vertex set of $\tilde G^N(\pi,\kappa)$.

For any $\ell\le N$, we have
\begin{align}\Prob{|C(1)|=\ell \text{ in }\tilde G^N(\pi,\kappa)} &= \binom{N-1}{\ell-1} \Prob{C(1)=[\ell]\text{ in }\tilde G^N(\pi,\kappa)}\nonumber\\
&=\binom{N-1}{\ell-1} \sum_{T\in T_{[\ell]}}\sum_{\substack{i_1,\ldots,i_\ell\\ \in[k]}} \prod_{m\in[\ell]} \pi_{i_m}\prod_{(m,n)\in E(T)}\left(1-\exp(-\kappa_{i_m,i_n}/N)\right)\nonumber\\
&\quad \prod_{\substack{(m,n)\in[\ell]^{(2)}\\ (m,n)\not\in E(T)}} \exp(-\kappa_{i_m,i_n}/N)\;\E{ \prod_{m=1}^\ell \prod_{j=1}^k \exp\left(-\frac{\kappa_{i_m,j}\bar p_j^{N,\ell}}{N}\right)}\nonumber\\
&\qquad + \Prob{|C(1)|=\ell\text{ and }C(1)\text{ includes a cycle in }\tilde G^N(\pi,\kappa)}.\label{eq:C=linGtilde}
\end{align}
In the first two lines, these products govern, respectively, the probabilities that the vertices in $[\ell]$ are present and have types $(i_1,\ldots,i_\ell)$; that the correct edges are present within $[\ell]$; that the correct non-edges are present within $[\ell]$; and the expectation (over random variables $(\bar p_j^{N,\ell})_{j\in[k]}$) gives the probability there are no edges between $[\ell]$ and $[N]\backslash[\ell]$, given the types of vertices $[\ell]$.

\par
The following convergence results hold immediately for all $i_1,\ldots,i_\ell\in[k]$:
\begin{align}
\lim_{N\rightarrow\infty}\prod_{\substack{(m,n)\in[\ell]^{(2)}\\ (m,n)\not\in E(T)}} \exp(-\kappa_{i_m,i_n}/N)&= 1,\label{eq:limforedges}\\
\lim_{N\rightarrow\infty}\binom{N-1}{\ell-1} \prod_{(m,n)\in E(T)} \left(1-\exp(-\kappa_{i_m,i_n}/N)\right) &= \frac{1}{(\ell-1)!}\prod_{(m,n)\in E(T)} \kappa_{i_m,i_n}.\label{eq:limfornonedges}
\end{align}

Now to treat the expectation term in \eqref{eq:C=linGtilde}, we rewrite $\bar p_j^{N,\ell}$ as $\sum_{j=\ell+1}^N \mathds{1}_{\{ X_n=j\}}$, and recall that $(X_n)$ as defined at \eqref{eq:defnXforGtilde} are IID.
\begin{align*}
\E{ \prod_{m=1}^\ell \prod_{j=1}^k \exp\left(-\frac{\kappa_{i_m,j}\bar p_j^{N,\ell}}{N}\right) }&=\E{ \prod_{m=1}^\ell \prod_{j=1}^k \prod_{n=\ell+1}^N\exp\left(-\frac{\kappa_{i_m,j}\mathds{1}_{\{ X_n=j \}}}{N}\right) }\\
&= \prod_{n=\ell+1}^N\E{ \prod_{j=1}^k \prod_{m=1}^\ell \exp\left(-\frac{\kappa_{i_m,j}\mathds{1}_{\{ X_n=j \}}}{N}\right) }\\
&= \prod_{n=\ell+1}^N \left[1- \Phi + \sum_{j=1}^k \pi_j \prod_{m=1}^\ell \exp\left(-\frac{\kappa_{i_m,j}}{N}\right) \right]\\
&= \left[1- \Phi + \sum_{j=1}^k \pi_j\left( 1-\frac{\sum_{m=1}^\ell \kappa_{i_m,j}}{N} +O(N^{-2}) \right)\right]^{N-\ell}\\
&= \left[ 1 - \frac{\sum_{m=1}^\ell \sum_{j=1}^k\kappa_{i_m,j}\pi_j}{N} + O\left(N^{-2}\right)\right]^{N-\ell}.\\
\end{align*}
And since $\Phi=\sum \pi_j$, recalling that $\ell$ is fixed, we obtain
\begin{equation}\label{eq:limEforGtilde}\lim_{N\rightarrow\infty}\E{ \prod_{m=1}^\ell \prod_{j=1}^k \exp\left(-\frac{\kappa_{i_m,j}\bar p_j^{N,\ell}}{N}\right) }=\prod_{m=1}^\ell \exp\left( -\sum_{j=1}^k \kappa_{i_m,j}\pi_j \right).
\end{equation}

Finally, we treat the extra term in \eqref{eq:C=linGtilde}, namely the probability that $C(1)$ includes a cycle, via a standard calculation for showing that local limits are trees. If $C(1)$ includes a cycle and $|C(1)|=\ell$, then it includes at least $\ell$ edges. So we bound this probability as
$$\Prob{|C(1)|=\ell\text{ and }C(1)\text{ includes a cycle in }\tilde G^N(\pi,\kappa)}\le \tbinom{N-1}{\ell -1}\sum_{E=\ell}^{\binom{\ell}{2}} \tbinom{\binom{\ell}{2}}{E} \left(1-\exp\left(-\tfrac{\kappa_{\max}}{N}\right)\right)^E.$$
Recall again that $\ell$ is fixed, so $\binom{N-1}{\ell-1}=\Theta(N^{\ell-1})$. Each summand has magnitude $\Theta(N^{-E})$, so
\begin{equation}
\lim_{N\rightarrow\infty}\Prob{|C(1)|=\ell\text{ and }C(1)\text{ includes a cycle in }\tilde G^N(\pi,\kappa)}=0.\label{eq:nocyclesinlimit}
\end{equation}

Combining \eqref{eq:nocyclesinlimit} with \eqref{eq:limforedges}, \eqref{eq:limfornonedges}, and \eqref{eq:limEforGtilde},
\begin{align}
\lim_{N\rightarrow\infty} \Prob{|C(1)|=\ell \text{ in }\tilde G^N(\pi,\kappa)} &= \frac{1}{(\ell-1)!}\sum_{T\in T_{[\ell]}} \sum_{\substack{i_1,\ldots,i_\ell\\ \in[k]}}\left[ \prod_{(m,n)\in E(T)} \kappa_{i_m,i_n} \right]\nonumber\\
&\qquad \prod_{m\in[\ell]} \pi_{i_m}\exp\left(-\sum_{j=1}^k \kappa_{i_m,j}\pi_j \right).\nonumber\\
\lim_{N\rightarrow\infty} \Prob{|C(1)|=\ell \text{ in }\tilde G^N(\pi,\kappa)}&=\frac{1}{\ell !}\sum_{T\in T^\rho_{[\ell]}} \sum_{\substack{i_1,\ldots,i_\ell\\ \in[k]}}\left[ \prod_{(m,n)\in E(T)} \kappa_{i_m,i_n} \right] \label{eq:limCinGtilde}\\
&\qquad \prod_{m\in[\ell]} \pi_{i_m}\exp\left(-\sum_{j=1}^k \kappa_{i_m,j}\pi_j \right),\nonumber
\end{align}
where the second equality holds by considering the natural $1$-to-$\ell$ mapping from $T_{[\ell]}$ to $T_{[\ell]}^\rho$, under which the summands are preserved.

\par
We now apply Proposition \ref{BJRweaklimit}. Although the statement of this result in \cite{BJRinhomog} specifically excludes random graphs on what the authors term \emph{generalised vertex spaces}, of which $\tilde G^N(\pi,\kappa)$ is an example, this is not a major problem. In $\tilde G^N(\pi,\kappa)$, consider the sequence $\bar p^{N,0}:=(\bar p_1^{N,0},\ldots \bar p_k ^{N,0})$ as defined in \eqref{eq:defnbarpNell}, which records the number of vertices of each type present in the graph. Conditional on $\bar p^{N,0}$, $\tilde G^N(\pi,\kappa)$ has the same distribution on the space of graphs with $k$ types, up to relabelling of the vertices, as $G^N(\bar p^{N,0},\kappa)$. However, $\bar p^{N,0}/N$ converges in probability to $\pi$ as $N\rightarrow\infty$. Therefore, we can lift \eqref{eq:BJRloclim} to obtain
\begin{equation}\label{eq:BJRloclimGtilde}\frac{1}{N}\mathcal{N}_\ell\left(\tilde G^N(\pi,\kappa)\right)\quad \stackrel{\mathbb{P}}\rightarrow\quad \Prob{|\Xi^{\pi,\kappa}|=\ell},\end{equation}
and since $\mathcal{N}_\ell (\tilde G^N(\pi,\kappa))/N\le 1$ almost surely, this convergence holds in expectation also. But the vertices and absent vertices of $\tilde G^N(\pi,\kappa)$ are exchangeable by construction (recall that some vertices in $[N]$ could be absent if $\pi$ is a strict subdistribution), and so
$$\E{\mathcal{N}_\ell\left(\tilde G^N(\pi,\kappa)\right)} = N \Prob{|C(1)|=\ell \text{ in }\tilde G^N(\pi,\kappa)}.$$
From this, we obtain
\begin{equation}\label{eq:limEC1Gtilde}\lim_{N\rightarrow\infty}\Prob{|C(1)|=\ell \text{ in }\tilde G^N(\pi,\kappa)} = \Prob{|\Xi^{\pi,\kappa}|=\ell}.\end{equation}
Then, by combining \eqref{eq:C=linGtilde} and \eqref{eq:limEC1Gtilde}, the required result \eqref{eq:ProbXi=ell} follows immediately.
\end{proof}
\end{lemma}

\subsection{Constructing solutions to Smoluchowski's equations from type flows}\label{vfrompisection}
Now we are in a position to show that $(v_\ell(t))$ constructed from $\Xi^{\pi(t),\kappa(t)}$ indeed satisfies the Smoluchowski equations.

\begin{prop}\label{flowissmolprop}Given $(\pi(t))_{t\ge 0}$ a frozen percolation type flow with initial kernel $\kappa$, set $v_\ell(t)= \Prob{|\Xi^{\pi(t),\kappa(t)}|=\ell}$ as before. Then $(v(t))_{t\ge 0}$ satisfies the Smoluchowski equations \eqref{eq:smol}, with $T_{\mathrm{gel}}=t_c$. Furthermore, we have
\begin{equation}\label{eq:barPhiPhi}\sum_{\ell=1}^\infty v_\ell(t) =\sum_{i=1}^k \pi_i(t),\end{equation}
and so it is consistent to call both of these quantities $\Phi(t)$.
\begin{proof}
We use the notation $\Xi^{(t)}$ as a shorthand for $\Xi^{\pi(t),\kappa(t)}$. We show \eqref{eq:barPhiPhi} first. We know that $\rho(\kappa(t)\circ\pi(t))\le 1$, so by Proposition \ref{Xisurvivallemma}, we have $\Prob{|\Xi^{(t)}|=\infty}=0$. Therefore
$$\sum_{\ell=1}^\infty v_\ell(t) = 1 - \Prob{\Xi^{(t)}=\varnothing} - \Prob{|\Xi^{(t)}| = \infty}=1 - \Prob{\Xi^{(t)}=\varnothing}= \sum_{i=1}^k \pi_i(t).$$

Now we consider the derivatives of $v_\ell(t)$. We write $\mu(t)$ as a shorthand for $\mu(\kappa(t)\circ\pi(t))$. First we observe that, for $t<t_c$,
$$\frac{\mathrm d}{\mathrm{d}t}\left[ \sum_{j=1}^k \kappa_{i,j}(t)\pi_j(t)\right]=1,\quad \forall i\in[k],$$
and for $t>t_c$,
\begin{align}
\frac{\mathrm d}{\mathrm{d}t}\left[ \sum_{j=1}^k \kappa_{i,j}(t)\pi_j(t)\right]&\stackrel{\eqref{eq:flowDE}}= \sum_{j=1}^k \pi_j(t) - \phi(t)\sum_{j=1}^k\kappa_{i,j}(t)\mu_j(t)\nonumber\\
&= \Phi(t) - \phi(t)\frac{\mu_i(t)}{\pi_i(t)},\quad i\in[k],\label{eq:doffspringdt}
\end{align}
from the definition of $\mu(t)$, and where by Lemma \ref{flowsstaypositive}, $\pi_i(t)>0$.

\par
Then, from Lemma \ref{intoverltrees}, $v_\ell(t)$ is given by:
$$\ell ! v_\ell(t) = \sum_{T\in T^\rho_{[\ell]}} \sum\limits_{\substack{i_1,\ldots,i_\ell\\ \in[k]}} \left[ \prod_{(m,n)\in E(T)}\kappa_{i_m,i_n(t)} \right] \prod_{m=1}^\ell \pi_{i_m}(t) \exp\left(-\sum_{j=1}^k \kappa_{i_m,j}(t)\pi_j(t)\right).$$

We differentiate directly with the product rule, and use \eqref{eq:flowDE} and \eqref{eq:doffspringdt}. For brevity, we set
$$A(t):= \sum_{m=1}^\ell \sum_{j=1}^k \kappa_{i_m,j}(t)\pi_j(t).$$
Note throughout that $A(t)$ is a function of $(i_1,\ldots,i_\ell)$. Then, for $t>t_c$,
\begin{align}
\ell !\frac{\mathrm d}{\mathrm d t} v_\ell(t) &= \sum_{T\in T^\rho_{[\ell]}} \sum\limits_{\substack{i_1,\ldots,i_\ell\\ \in[k]}} \exp(-A(t))\left[ \prod_{(m,n)\in E(T)}\kappa_{i_m,i_n}(t) \right]\nonumber\\
&\qquad\qquad\qquad\times
 \left[ \phi(t) \sum_{m'=1}^\ell\frac{\mu_{i_{m'}}(t)}{\pi_{i_{m'}}(t)} - \ell\Phi(t)\right] \left[\prod_{m=1}^\ell \pi_{i_m}(t)\right]\label{eq:dvldt} \\ 
&\quad - \phi(t)\sum_{T\in T^\rho_{[\ell]}} \sum\limits_{\substack{i_1,\ldots,i_\ell\\ \in[k]}} \exp(-A(t))\left[ \prod_{(m,n)\in E(T)}\kappa_{i_m,i_n}(t) \right] \sum_{m=1}^\ell \mu_{i_m}(t) \prod_{\substack{m'=1\\m'\ne m}}^\ell \pi_{i_m}(t) \nonumber \\ 
&\quad + \sum_{T\in T^\rho_{[\ell]}}\sum\limits_{\substack{i_1,\ldots,i_\ell\\ \in[k]}} \exp(-A(t))\left[\sum_{(m,n)\in E(T)} \prod_{\substack{(m',n')\in E(T)\\ (m',n')\ne (m,n)}}\kappa_{i_{m'},i_{n'}}(t) \right] \prod_{m=1}^\ell \pi_{i_m}(t). \nonumber
\end{align}
The first line comes from differentiating $\exp(-A(t))$ using \eqref{eq:doffspringdt}; the second line from differentiating $\pi_{i_m}(t)$ using \eqref{eq:flowDE}; and the final line from $\prod_{(m,n)\in E(T)} \kappa_{i_m,i_n}(t)$ directly. In the first two lines, the terms involving $\phi(t)$ cancel, leaving $-\ell \cdot \ell! v_\ell(t)\Phi(t)$. (This applies equally on $t<t_c$, where $\Phi(t)\equiv 1$; and $t=t_c$, as the left- and right-derivatives match.)

To deal with the third line, given $T$ and $(m,n)\in E(T)$, consider the pair of disjoint trees $T^m,T^m$ formed by removing the edge $(m,n)$ from $T$, where $m\in T^m$ and $n\in T^n$. Then the sum in the third line of \eqref{eq:dvldt} splits as a product across these two trees:
\begin{align}
&\sum_{T\in T^\rho_{[\ell]}}\sum_{(m,n)\in E(T)}\sum\limits_{\substack{i_1,\ldots,i_\ell\\ \in[k]}}\left[\prod_{\substack{(m',n')\in E(T)\\(m',n')\ne (m,n)}} \kappa_{i_{m'},i_{n'}}(t) \right]\prod_{m=1}^\ell \pi_{i_m}(t) \exp\left(-\sum_{j=1}^k \kappa_{i_m,j}(t)\pi_j(t)\right) \nonumber\\
& = \sum_{T\in T^\rho_{[\ell]}}\sum_{(m,n)\in E(T)}\left ( \sum\limits_{\substack{i_{m'}\in[k]\\ m'\in T^m\\}}\left[\prod_{(m',n')\in E(T^m)} \kappa_{i_{m'},i_{n'}}(t) \right]\left[\prod_{m'\in T^m}\pi_{i_{m'}}(t) \exp\left(-\sum_{j=1}^k \kappa_{i_{m'},j}(t)\pi_j(t)\right)\right]  \right) \nonumber\\
&\quad\times\left ( \sum\limits_{\substack{i_{n'}\in[k]\\ n'\in T^n\\}}\left[\prod_{(m',n')\in E(T^n)} \kappa_{i_{m'},i_{n'}}(t) \right]\left[\prod_{n'\in T^n}\pi_{i_{n'}}(t) \exp\left(-\sum_{j=1}^k \kappa_{i_{n'},j}(t)\pi_j(t)\right)\right]  \right). \label{eq:splitastreeprod}
\end{align}
Consider the set of rooted trees on $[\ell]$ with an identified edge
$$\mathbb{T}_{[\ell]}:= \left\{ \left(T,\{m,n\}\right)\,:\, T\in T^\rho_{[\ell]}, \{m,n\}\in E(T)\right\}.$$
Recall a \emph{rooted forest} is a disjoint union of rooted trees. Let $T_{[\ell]}^{(2)}$ be the set of rooted forests on $[\ell]$ with exactly two trees. Consider the map from $\mathbb{T}_{[\ell]}$ to $T_{[\ell]}^{(2)}$ given by removing the identified edge $\{m,n\}$ from $T$, and rooting the two resulting trees at $m$ and $n$. This map is $\ell$-to-1, since the root of $T$ plays no role in the map!
\par
So in \eqref{eq:splitastreeprod}, we may replace the double sum
$$\sum_{T\in T^\rho_{[\ell]}} \sum_{\{m,n\}\in E(T)}\quad\text{with the sum}\quad\ell \sum_{\substack{T^1\sqcup T^2\\ \in T_{[\ell]}^{(2)}}}.$$
Then, by considering which elements of $[\ell]$ belong to each of the two trees, we can replace the latter sum with
$$\frac{\ell}{2} \sum_{r=1}^{\ell-1} \sum_{A\in \binom{[\ell]}{r}} \sum_{T^1\in T^\rho_A} \sum_{T^2\in T^\rho_{[\ell]\backslash A}},$$
where, recall, $T^{\rho}_A$ is the set of rooted trees labelled by $A$. Note that in this sum, the $\frac12$ appears because the order of trees $T^1,T^2$ does not matter. So we rewrite \eqref{eq:splitastreeprod} as
$$\frac{\ell}{2}\sum_{r=1}^{\ell-1} \sum_{A\in\binom{[\ell]}{r}} \left ( \sum_{T^1\in T^\rho_A}\; \sum\limits_{\substack{i_{m'}\in[k]\\ m'\in T^1\\}}\left[\prod_{\{m',n'\}\in E(T^1)} \kappa_{i_{m'},i_{n'}}(t) \right]\left[\prod_{m'\in T^1}\pi_{i_{m'}}(t) \exp\left(-\sum_{j=1}^k \kappa_{i_{m'},j}(t)\pi_j(t)\right)\right]  \right)$$
$$\times \left ( \sum_{T^2\in T^\rho_{[\ell]\backslash A}}\sum\limits_{\substack{i_{n'}\in[k]\\ n'\in T^2\\}}\left[\prod_{\{m',n'\}\in E(T^2)} \kappa_{i_{m'},i_{n'}}(t) \right]\left[\prod_{n'\in T^2}\pi_{i_{n'}} \exp\left(-\sum_{j=1}^k \kappa_{i_{n'},j}(t)\pi_j(t)\right)\right]  \right).$$

Relabelling variables inside each large bracket, and moving factorials around, we obtain
$$\frac{\ell}{2}\cdot \ell !\sum_{r=1}^{\ell-1} \left (\frac{1}{r!} \sum_{T^1\in T^\rho_{[r]}} \;\;\sum\limits_{\substack{i_1,\ldots,i_r\\ \in[k]}} \left[ \prod_{\{m,n\}\in E(T)}\kappa_{i_m,i_n}(t) \right] \prod_{m=1}^r \pi_{i_m}(t) \exp\left(-\sum_{j=1}^k \kappa_{i_m,j}(t)\pi_j(t)\right)  \right)$$
$$\times \left ( \frac{1}{(\ell-r)!} \sum_{T^2\in T^\rho_{[\ell-r]}} \;\;\sum\limits_{\substack{i_1,\ldots,i_{\ell-r}\\ \in[k]}} \left[ \prod_{\{m,n\}\in E(T)}\kappa_{i_m,i_n}(t) \right] \prod_{m=1}^{\ell-r} \pi_{i_m}(t) \exp\left(-\sum_{j=1}^k \kappa_{i_m,j}(t)\pi_j(t)\right)   \right),$$

which is equal to
$$\ell ! \cdot \frac{\ell}{2} \sum_{r=1}^{\ell-1} v_r(t)v_{\ell-r}(t).$$
We have already seen that the first two lines of \eqref{eq:dvldt} are equal to $-\ell \cdot \ell ! v_\ell(t)\Phi(t)$. Therefore, cancelling the $\ell!$ terms, we conclude from \eqref{eq:dvldt} that
$$\frac{\mathrm d}{\mathrm d t}v_\ell(t)= \frac{\ell}{2} \sum_{r=1}^{\ell-1} v_r(t)v_{\ell-r}(t) - \ell\Phi(t)v_\ell(t),$$
for all $t\ge 0$, as required.
\end{proof}
\end{prop}

\subsection{FP type flows are unique}

Now we can finish the proof of Proposition \ref{uniqueprop} using the following lemma, whose proof is postponed to Section \ref{Lipschitzproof}.
\begin{lemma}\label{muLipprop1}For any $0<\eta<T<\infty$, there exists $C(\eta,T)<\infty$ such that, for all matrices $A,A'\in[\eta,T]^{k\times k}$,
\begin{equation}\label{eq:mulocLip1}||\mu(A)-\mu(A')||_1 \le C(\eta,T)\max_{i,j\in[k]} |A_{i,j}-A'_{i,j}|.\end{equation}
\end{lemma}

Now suppose we have FP type flows $\pi(\cdot)$ and $\nu(\cdot)$ with the same initial kernel $\kappa$ and $\pi(0)=\nu(0)$. Set $\Phi^\pi(\cdot):= ||\pi(\cdot)||_1$ and $\Phi^\nu(\cdot):= ||\nu(\cdot)||_1$. Then, consider the associated solutions to the Smoluchowski equations given by Proposition \ref{flowissmolprop}, $(v^\pi(\cdot)),(v^\nu(\cdot))$. Crucially, $\pi(0)=\nu(0)$ implies $v^\pi(0)=v^\nu(0)$. Theorem \ref{smolunique} concerning uniqueness of solutions to Smoluchowski's equations then gives $v^\pi(t)=v^\nu(t)$ for all times $t\ge 0$. Furthermore, from \eqref{eq:barPhiPhi}, $\Phi^\pi(t)=\Phi^\nu(t)$ for all $t\ge 0$, and $t^\pi_c=t^\nu_c$, with $\phi^\pi(t)=\phi^\nu(t)$ for all $t\ge t^\pi_c$.

\par
We may now use the classical technique for verifying uniqueness of solutions to ODEs, using the local Lipschitz property of $\mu$ from Lemma \ref{muLipprop1}. The flow $\pi(\cdot)$ satisfies the integral version of \eqref{eq:flowDE},
\begin{equation}\label{eq:integralDE}\pi(t)=\pi(t_c) - \int_{t_c}^t \mu(\kappa(s)\circ \pi(s)) |\mathrm{d}\Phi^\pi(s)|,\quad t\ge t_c,\end{equation}
and similarly for $\nu(\cdot)$. So
$$\pi(t)-\nu(t) = \int_{t_c}^t \left[\mu(\kappa(s)\circ \nu(s))-\mu(\kappa(s)\circ \pi(s))\right] |\mathrm{d}\Phi(s)|,\quad t\ge t_c.$$

For a fixed time $T>t_c$, by Lemma \ref{flowsstaypositive}, we can choose $\eta>0$ such that $\pi_i(T),\nu_i(T)\ge\eta$ for all $i\in[k]$. Now set $\delta:= t_c\wedge \kappa_{\min}$. From the assumptions we made about the initial conditions, $\kappa_{i,j}(t)\ge \delta>0$ whenever $t\ge t_c$. So with constant $C(\delta \eta,\kappa_{\max}+T)$ given by \eqref{eq:mulocLip1}, for $t\in[t_c,T]$,
\begin{align}
||\mu(\kappa(t)\circ \pi(t)) - \mu(\kappa(t)\circ \nu(t))||_1 &\le C(\delta \eta, \kappa_{\max}+T) \max_{i,j\in[k]} \left[ \kappa(t)\circ \pi(t) - \kappa(t) \circ \nu(t) \right]_{i,j}\nonumber\\
&\le C(\delta \eta, \kappa_{\max}+T) \cdot (\kappa_{\max}+T) ||\pi(t)-\nu(t)||_1.\label{eq:usingmuLip}\\
\intertext{Therefore, for $t\in[t_c,T]$,}
||\pi(t)-\nu(t)||_1 &\le C(\delta \eta,T+\kappa_{\max}) \int_{t_c}^t ||\pi(s)-\nu(s)||_1 \phi(s) \mathrm{d}s. \nonumber
\end{align}

We have $\pi(t_c)=\nu(t_c)$, so applying Gronwall's Lemma gives $\pi(t)=\nu(t)$ for all $t\in[t_c,T]$. But $T$ was arbitrary, and so in fact we may conclude $\pi(t)=\nu(t)$ for all $t\ge 0$. This completes the proof of Proposition \ref{uniqueprop}.\qed

\section{Large components in inhomogeneous random graphs}\label{largecptssection}

During a $k$-type frozen percolation process, most vertices are frozen at a moment when they are in a large component of an `almost-critical' IRG. In this section, we derive concentration estimates on the proportion of types seen in such components. We will use these estimates in Section \ref{weaklimitsection} to show that weak limits of $k$-type FPPs satisfy the eigenvector property of the type flow equation \eqref{eq:flowDE}.

In the course of this section, we will require several technical results about positive matrices and their eigenvectors, and proofs of some of these will be postponed to Section \ref{technicalproofs} to avoid breaking the flow of the probabilistic argument.

\subsection{Exponential bounds on component sizes in $G^N(p,\kappa)$}\label{expbdsetupsection}

Recall Definition \ref{defnxi} of the multitype branching process $\Xi^{\pi,\kappa}$, and Proposition \ref{Xisurvivallemma} concerning its survival probability. Define $\zeta^{\pi,\kappa}_i:= \Prob{|\Xi^{\pi,\kappa}|=\infty\,|\,\mathrm{type}(\mathrm{root})=i}$. Since the offspring distributions are Poisson, it is shown in \cite{BJRinhomog} that $\zeta^{\pi,\kappa}$ is the maximal solution to
\begin{equation}\label{eq:fixedpointzeta}
\zeta^{\pi,\kappa}_i = 1 - \exp\left(-[(\kappa\circ\pi)\zeta^{\pi,\kappa}]_j\right),
\end{equation}
which we study in detail in Lemma \ref{O1boundzeta}. For now, we introduce a more detailed version of Proposition \ref{BJRL1basic}, complementing Proposition \ref{BJRweaklimit}.

\begin{prop}[\cite{BJRinhomog}, Theorem 3.1]\label{BJRL1enhanced}
Fix a positive kernel $\kappa\in\R_+^{k\times k}$ and subdistribution $\pi\in\Pi_{\le 1}$. Suppose a sequence $p^N\in\N_0^k$ satisfies $\sum_{i\in[k]} p^N_i = N$ and $p^N/N \rightarrow \pi$. Then
\begin{equation}\label{eq:L1conv}\frac{1}{N}L_1 \left(G^N(p^N,\kappa)\right) \quad\stackrel{\mathbb{P}}\longrightarrow \quad \pi\cdot\zeta^{\pi,\kappa},\qquad N\rightarrow\infty. \end{equation}
\end{prop}

The limiting quantity in \eqref{eq:L1conv} vanishes as $\rho(\kappa\circ\pi)\downarrow 1$. For our purposes, we require exponential bounds on the probability that $\frac{1}{N}L_1(G^N(p^N,\kappa))$ is large, which hold uniformly among kernel-distribution pairs for which $\rho(\kappa\circ\pi)\le 1+\epsilon$, as follows.

\begin{prop}\label{LDPprop}
Fix $\eta,\epsilon\in(0,1/2)$. Then there exist $N_0=N_0(\epsilon,\eta)\in\N$ and constants $M=M(\eta)<\infty$ and $\Gamma=\Gamma(\epsilon,\eta)>0$, such that for any $N\ge N_0$ and
\begin{itemize}
\item any kernel $\kappa\in[\eta,\infty)^{k\times k}$;
\item any vector $p\in\N^k$ such that $\sum p_i=N$ and $p_i/N\ge \eta$ for each $i$;
\item and such that the eigenvalue condition $\rho(\kappa\circ p/N)\le 1+\epsilon$ is satisfied;
\end{itemize}
the following bound on the largest component in $G^N(\rho,\kappa)$ holds:
\begin{equation}\label{eq:uniformLDPbd}
\Prob{L_1\left(G^N(p,\kappa) \right) \ge M\epsilon N}\le \exp(-\Gamma N).
\end{equation}
\end{prop}

A version of this result is proved in the author's doctoral thesis (\cite{Yeothesis}, Chapter 3) using a $\Z^k$-valued multitype exploration process of the graph. We give a shorter proof here, by adapting exponential tail bounds shown in \cite{BJR10}, and using a majorisation lemma to reduce the uniform statement to a statement for finitely many pairs $(\pi,\kappa)$. Since the argument is self-contained and less novel, we postpone this proof to Section \ref{LDPproof}.

\subsection{Distribution of types at large radius}
We require a result about the proportion of types in \emph{all large components} of a near-critical IRG; that is, not just the giant component (if it exists). We will approach this by considering the types of vertices connected at large distance from a uniformly chosen vertex. For many choices of the root vertex there will be no vertices at large radius. But for large components, the \emph{majority} of the vertices in such components will be a large distance from a uniformly chosen vertex.

\par
For any graph $G$ with $k$ types on $N$ vertices, we will take $v$ to be a uniformly chosen vertex. Then, for $r=0,1,\ldots,N-1$, define $W^r\in\N_0^k$ by,
$$W^r_i:=\#\{\text{type $i$ vertices distance $r$ from $v$}\},\quad i\in[k],$$
and $W^{\ge R}_i := \sum_{r=R}^{N-1} W^r_i$.

The goal of this section is the following theorem.
\begin{theorem}\label{typecontrolthm}Fix constants $0<\eta<T<\infty$. Now, for any $\delta>0$, there exists $\epsilon=\epsilon(\delta,\eta,T)>0$, and $R=R(\delta,\eta,T),N_0=N_0(\delta,\eta,t)\in\N$ satisfying the following. Consider any $\kappa\in [\eta,T]^{k\times k}$ and $p\in\N_0^k$ satisfying $\sum p_i=N\ge N_0$ and $p_i\ge \eta N$, with $\rho(\kappa\circ \frac{p}{N})\le 1+\epsilon$. Then $W^{\ge R}$ corresponding to $v$, a uniformly chosen vertex in $G^N(p,\kappa)$ satisfies
\begin{equation}\label{eq:propsbeyondR}\left|\left|\E{W^{\ge R}} - \mu(\kappa\circ \pi) \left|\left|\E{W^{\ge R}}\right|\right|_1 \right|\right|_1 \le \delta \left|\left|\E{W^{\ge R}}\right|\right|_1.\end{equation}

Furthermore, recall the definition of $M=M(\eta)$ from Proposition \ref{LDPprop}, and set $\chi=M\epsilon$. Define the event $A_\chi:=\{||W^{\ge R}||_1\le \chi N\}$, that the component containing $v$ includes at most $\chi N$ vertices with radius at least $R$ from $v$. Then we also have a constant $N_1=N_1(\delta,\eta,T)$ such that whenever $N\ge N_1$,
\begin{equation}\label{eq:types<chi}\left|\left|\E{W^{\ge R}\mathds{1}_{A_\chi}} - \mu(\kappa\circ \pi) \left|\left|\E{W^{\ge R}\mathds{1}_{A_\chi}}\right|\right|_1 \right|\right|_1 \le \delta \left|\left|\E{W^{\ge R}\mathds{1}_{A_\chi}}\right|\right|_1.\end{equation}
\end{theorem}

\begin{remark}
Neither the statements nor the proofs consider the value, or even the scale of $\left|\left|\E{W^{\ge R}}\right|\right|_1$. The results \eqref{eq:propsbeyondR} and \eqref{eq:types<chi} deal only with the \emph{direction} of $\E{W^{\ge R}}$, since they hold uniformly over $\rho\le 1+\epsilon$, which includes subcritical, critical and supercritical regimes. That is, for $\rho=1+\epsilon$, $||\E{W^{\ge R}}||_1 = \Theta(N)$ for any fixed $R$, whereas for fixed $\rho<1$, $||\E{W^{\ge R}}||_1 = \Theta(1)$ as $N\rightarrow\infty$.
\end{remark}

\subsection{Preliminary results about products of matrices}

We begin with two lemmas. We show that convex combinations of matrices close to a fixed matrix can be written as a single matrix close to that fixed matrix. Such matrices will appear as expected multiplicative increments for a discrete-time process, and this lemma allows us the control the increments across multiple time-steps. We then state a lemma saying that applying a large enough product of matrices close to a fixed positive matrix to any vector has direction close to the principal eigenvector of the fixed matrix.

\par
The combination of these results allows us to control the expected \emph{proportion} of types at some large radius from a fixed vertex in an inhomogeneous random graph, irrespective of the expected \emph{number} of vertices at this radius. Considering all distances from the fixed vertex simultaneously proves the required concentration result for the proportion of types in a typical large component, and all the estimates hold uniformly among graphs with bounded Perron root.

\begin{definition}For $A\in\R^{k\times k}_+$, and $\theta>0$, define
\begin{equation}\label{eq:defnBthetaA}\mathbb{B}_\theta(A):= \{B\in \R_+^{k\times k} \, :\, |B_{i,j}-A_{i,j}|\le \theta,\,\forall i,j\in[k] \},\end{equation}
the set of positive matrices whose entries differ from those of $A$ by at most $\theta$.
\end{definition}

\begin{lemma}\label{mixmatrices}
Given $A\in \R_+^{k\times k}$, $L\in\N$, and $\theta>0$ such that $\min_{i,j} A_{i,j}>\theta$, consider non-negative non-zero vectors $x^{(1)},\ldots,x^{(L)}\in \R_{\ge 0}^k\backslash\{0\}$, and any $L$ matrices
$$D^{(1)},\ldots,D^{(L)}\in\mathbb{B}_\theta(A),$$
and positive real numbers $p_1,\ldots,p_L$ satisfying $\sum p_l=1$. Then there exists a matrix $\bar D\in\mathbb{B}_\theta(A)$ such that,
\begin{equation}\label{eq:meanBetaA}p_1 x^{(1)}D^{(1)}+\ldots+ p_L x^{(L)}D^{(L)} = (p_1x^{(1)}+\ldots +p_L x^{(L)}) \bar D.\end{equation}

\begin{proof}
Let $\mathbf{1}$ be the $k\times k$ matrix where every entry is 1. Since the $x^{(i)}$s are non-negative, and each $D^{(l)}\in \mathbb{B}_\theta(A)$,
\begin{align*}
p_1 x^{(1)}D^{(1)}+\ldots+ p_L x^{(L)}D^{(L)}&\le p_1x^{(1)}(A+\theta \mathbf{1}) + \ldots + p_L x^{(L)}(A+\theta \mathbf{1})\\
&\le (p_1x^{(1)}+\ldots+ p_L x^{(L)}) (A+\theta \mathbf{1}).
\intertext{Similarly,}
p_1 x^{(1)}D^{(1)}+\ldots +p_L x^{(L)}D^{(L)}&\ge  (p_1x^{(1)}+\ldots + p_L x^{(L)}) (A-\theta \mathbf{1}).
\end{align*}

For ease of notation, set $y:= p_1 x^{(1)}D^{(1)}+\ldots+ p_L x^{(L)}D^{(L)}$ and $z:=p_1x^{(1)}+\ldots+ p_L x^{(L)}$, so \begin{equation}\label{eq:meanvector}z(A-\theta\mathbf{1})\le y\le z(A+\theta \mathbf{1}).\end{equation}
Now, for each $j\in[k]$, set $c_j:= \frac{y_j - [zA]_j}{||z||_1}$, so that $\sum_{i=1}^k z_i(A_{i,j}+c_j) = y_j$. Since the LHS is increasing in $c_j$, from \eqref{eq:meanvector} we have $|c_j|\le \theta$.

\par
So we may define $\bar D\in \mathbb{B}_\theta(A)$ via $\bar D_{i,j}=A_{i,j}+c_j$, and this satisfies \eqref{eq:meanBetaA}.
\end{proof}
\end{lemma}

The following lemma studies large products of matrices close to a fixed matrix (such as $\kappa\circ\pi$). It shows that that the images of all vectors under such a matrix product are close to the principal eigenspace of the fixed matrix. Because $\kappa\circ\pi$ could be non-diagonalisable, the proof is rather involved, and is postponed to Section \ref{Perronproofs}.
\begin{lemma}\label{convevector}
For all $0<\eta<T<\infty$ with $\eta<1$, and $\delta>0$, there exists $\theta=\theta(\delta,\eta,T)\in(0,\eta^2)$ and $R=R(\delta,\eta,T)<\infty$ such that
\begin{equation}\label{eq:prodapproxmu}\left|\left| \frac{v D^{(1)}\ldots D^{(R)}}{||v D^{(1)}\ldots D^{(R)}||_1} - \mu(\kappa\circ\pi) \right|\right|_1 <\delta,\end{equation}
for all $v \in\R^k_{\ge 0}\backslash\{0\}$, $\kappa\in[\eta,T]^{k\times k}$, $\pi\in \Pi_{\le 1}\cap[\eta,1]^k$, and $D^{(1)},\ldots,D^{(R)}\in \mathbb{B}_\theta(\kappa\circ \pi)$.
\end{lemma}

\subsection{Proof of Theorem \ref{typecontrolthm}}

\begin{proof}

We may insist $T>2$, and first choose any $\epsilon>0$ small enough that $M(\eta)$ as defined in Proposition \ref{LDPprop} and $\theta(\delta,\eta,T)$ as defined in Lemma \ref{convevector} satisfy
\begin{equation}\label{eq:chithetacondition}M(\eta)\epsilon T^2\le \theta(\delta,\eta,T)<\eta^2.\end{equation}
Set $\chi=M\epsilon$. Assume throughout that we have a graph $G^N(p,\kappa)$ satisfying the conditions of the statement. For each $j\in[k]$, conditional on $(W^0,W^1,\ldots,W^r$), $W^{r+1}_j$ has distribution
\begin{equation}\label{eq:distWrplus1}\Bin{p_j-(W^0+W^1+\ldots+W^r)_j}{1-e^{-(W^r \kappa)_j /N}}.\end{equation}

The first parameter counts type $j$ vertices in the graph that are not within distance $r$ from $v$. For each of these vertices independently, the probability that it is connected to none of the vertices at distance $r$ from $v$ is $\prod_{i=1}^k (e^{-\kappa_{i,j}/N})^{W^r_i} = e^{-(W^r \kappa)_j /N}$. So,
\begin{align*}
\E{W^{r+1}_j \,\big|\, W^0,\ldots,W^r}&= \left[p_j-(W^0+W^1+\ldots+W^r)_j\right]\cdot \left( 1-e^{-(W^r \kappa)_j /N} \right)\\
&= \left[\pi_j - \frac{(W^0+\ldots+W^r)_j}{N}\right] \cdot N\left( 1-e^{-(W^r \kappa)_j /N} \right)\\
&\le \pi_j [W^r \kappa]_j = [W^r (\kappa\circ \pi)]_j.
\end{align*}
And so we conclude that
\begin{equation}\E{W^{r+1} \,\big|\, W^0,\ldots,W^r} \le W^r (\kappa\circ \pi). \label{eq:matrixincrbd1}\end{equation}
Define the matrix $D^{(r)}\in \R^{k \times k}$ by
$$D^{(r)}_{i,j}=(\kappa\circ \pi)_{i,j} - \frac{1}{||W^r||_1}\left[ W^r(\kappa\circ \pi) - \E{W^{r+1} \,\big|\, W^0,\ldots,W^r}\right]_j.$$
So we may write this conditional expectation as
\begin{equation}\label{eq:Wmgale}\E{W^{r+1}\,|\, W^0,\ldots,W^r} = W^rD^{(r)}.\end{equation}

\par
We define $S^r:= ||W^0+\ldots+W^r||_1$ to be the total number of vertices within radius $r$ of the root and recall that $x-x^2/2\le 1-e^{-x}$ for $x\ge 0$. We can then derive a bound in the opposite direction to \eqref{eq:matrixincrbd1}.
$$\E{W^{r+1}_j \,\big|\, W^0,\ldots,W^r}\ge \left[\pi_j- \frac{S^r}{N}\right]\left( [W^r \kappa]_j - \frac{([W^r \kappa]_j)^2}{2N} \right).$$

Recall $\chi=M\epsilon$. When $S^r\le \chi N$, since $\kappa\in[\eta,T]^{k\times k}$ we have
\begin{align}
0\le \left(\kappa\circ \pi - D^{(r)}\right)_{i,j}&\le \frac{1}{||W^r||_1}\left [ \frac{S^r}{N}[W^r\kappa]_j + \pi_j \frac{([W^r \kappa]_j)^2}{2N}  \right]\nonumber\\
0\le \left(\kappa\circ \pi - D^{(r)}\right)_{i,j}&\le \chi T + \frac{\chi T^2}{2}\le \chi T^2\le \theta,\label{eq:matrixincrbd2}
\end{align}
for all $i,j\in[k]$, from the assumptions \eqref{eq:chithetacondition} made at the start of the proof.

\par
Now fix some $r$ between 0 and $N-R-1$. We combine all the previous ingredients to show that $\E{W^{r+R}}$ has direction within $\delta$ of $\mu(\kappa\circ \pi)$. We first define a version of process $W$ for which the matrices governing the expected one-step evolution of $W$ are always within $\mathbb{B}_{\theta}(\kappa\circ \pi)$. We do this to show that the contributions from the rare event $\{|C(v)|>\chi N\}$ are negligible as $N\rightarrow\infty$.

\par
Let $\tilde W^r=W^r$. Then, inductively, for $m=0,1,\ldots,R-1$, let
\begin{equation}\label{eq:defntildeW}\tilde W^{r+m+1}= \begin{cases}W^{r+m+1} &\quad S^{r+m}\le \chi N\\ \tilde W^{r+m}(\kappa\circ \pi) &\quad S^{r+m}>\chi N.\end{cases}\end{equation}
That is, $\tilde W$ tracks $W$ until the first time that $S$ exceeds $\chi N$, and thereafter evolves deterministically, with transitions given by right-multiplication with $(\kappa\circ \pi)$. Later we will be particularly interested in $\tilde W^{r+R}$ as $r$ varies, so we let $Y^r:=\tilde W^{r+R}$. (Note that for different values of $r$, $(\tilde W^{r+m})_{m\ge 0}$ are formally \emph{different} processes.)

\par
For $m\ge 0$, define $\F_{r+m}:=\sigma(W^0,\ldots,W^{r+m})$. So, in particular
$$(W^0, W^1,\ldots,W^r, \tilde W^{r+1},\ldots,\tilde W^{r+m})$$
is $\F_{r+m}$-measurable. Because of \eqref{eq:Wmgale} and \eqref{eq:defntildeW}, we have
\begin{equation}\label{eq:tildemgale}\E{\tilde W^{r+m+1} \,\big|\, W^0,\ldots, W^r, \tilde W^{r+1},\ldots, \tilde W^{r+m}} = \tilde W^{r+m}D^{(r+m)},\end{equation}
where $D^{(r+m)}$ is $\F_{r+m}$-measurable. On the $\F^{r+m}$-measurable event $\{S^{r+m}>\chi N\}$, $D^{(r+m)}=\kappa\circ\pi$, and otherwise $D^{(r+m)}\in \mathbb{B}_{\theta} (\kappa \circ \pi)$ from \eqref{eq:matrixincrbd2}. Therefore $D^{(r+m)}\in \mathbb{B}_{\theta} (\kappa \circ \pi)$ almost surely. We will now show that expected $R$-step transitions are given by a product of $R$ matrices in a similar way, using Lemma \ref{mixmatrices}.

\par
\emph{Claim:} For any $1\le m\le R$, there exist $\F_r$-measurable matrices $D^{(1)},\ldots,D^{(m)}\in \mathbb{B}_\theta(\kappa\circ \pi)$ such that
\begin{equation}\label{eq:inductmatrices}\E{\tilde W^{r+m} \,|\, \F_r} = W^r D^{(1)}\ldots D^{(m)}.\end{equation}

We prove the claim by induction on $m$. Suppose the claim is true for a particular value of $m$. Clearly $\supp{\tilde W^{r+m}}$ is finite, and for each $w\in\supp{\tilde W^{r+m}}$ by \eqref{eq:tildemgale}, we have (after a superficial change of notation - recall $r$ is currently fixed)
$$\E{\tilde W^{r+m+1} \,\big|\, W^0,\ldots, W^r, \tilde W^{r+1},\ldots, \tilde W^{r+m-1}, \tilde W^{r+m}=w} = w\bar D^{(m+1)},$$
where $\bar D^{(m+1)}$ is $\F_{r+m}$-measurable and in $\mathbb{B}_\theta(\kappa\circ \pi)$. So
$$\E{\tilde W^{r+m+1} \,\big|\, W^0,\ldots, W^r, \tilde W^{r+m}=w} = w \E{\bar D^{(m+1)} \,\big|\,W^0,\ldots, W^r, \tilde W^{r+m}=w}.$$
The expectation on the RHS is a convex combination of elements of the convex set $\mathbb{B}_\theta(\kappa\circ \pi)$. So $D^{(m+1,w)}:=\E{\bar D^{(m+1)} \,\big|\,W^0,\ldots, W^r, \tilde W^{r+m}=w}$ is $\F_r$-measurable, and is almost surely in $\mathbb{B}_\theta(\kappa\circ \pi)$.

\par
We now apply the tower law:
\begin{align*}
\E{\tilde W^{r+m+1} \,|\, \F_r} &= \sum_{w\in\supp{\tilde W^{r+m}}}\E{\tilde W^{r+m+1}\,| \, W^0,\ldots,W^r,\tilde W^{r+m}=w}\Prob{\tilde W^{r+m}=w\,\big|\,\F_r}\\
&= \sum_{w\in\supp{\tilde W^{r+m}}} w D^{(m+1,w)} \Prob{\tilde W^{r+m}=w\,\big|\,\F_r}.
\end{align*}
So by Lemma \ref{mixmatrices}, there exists an $\F_r$-measurable matrix $D^{(m+1)}\in \mathbb{B}_\theta(\kappa\circ \pi)$ such that
\begin{align*}
\E{\tilde W^{r+m+1} \,|\, \F_r}&= \left(\sum_{w\in\supp{\tilde W^{r+m}}} w\Prob{\tilde W^{r+m}=w\,\big|\,\F_r}\right) D^{(m+1)}\\
&= \E{\tilde W^{r+m} \,|\, \F_r} D^{(m+1)},\\
\intertext{a conditional version of \eqref{eq:Wmgale}. Then, using the assumed inductive hypothesis,}
\E{\tilde W^{r+m+1} \,|\, \F_r}&= W^r D^{(1)}\ldots D^{(m)}D^{(m+1)}.
\end{align*}
The claim \eqref{eq:inductmatrices} follows for all $m\le R$ by induction. In particular, the case $m=R$ gives
\begin{equation}\label{eq:WrplusR}\E{\tilde W^{r+R} \,|\, \F_r} = W^r D^{(1)}\ldots D^{(R)}.\end{equation}

\par
Since each $D^{(m)}\in\mathbb{B}_{\theta}(\kappa\circ\pi)$, we now have precisely the conditions to use Lemma \ref{convevector}, whenever $W^r\ne 0$. Fix $\delta>0$, and also note that $\kappa\circ \pi\in[\eta^2,T]^{k\times k}$ by assumption. The lemma specifies $R=R(\delta/2,\eta,T)$ and we conclude that
$$\left|\left|\E{\tilde W^{r+R}\,|\,\F^r} - \mu(\kappa\circ \pi)||\mathbb{E}[\tilde W^{r+R}\,|\,\F^r]||_1 \right|\right|_1\le \frac{\delta}{2}\left|\left|\E{\tilde W^{r+R}\,|\,\F^r}\right|\right|_1,$$
almost surely (including the trivial case $W^r=0$), and in particular,
\begin{equation}\label{eq:Yrdirbd}\left|\left|\E{\tilde W^{r+R}} - \mu(\kappa\circ \pi)||\mathbb{E}[\tilde W^{r+R}]||_1 \right|\right|_1\le \frac{\delta}{2}\left|\left|\E{\tilde W^{r+R}}\right|\right|_1.\end{equation}

Recall that $r$ was fixed throughout, and $Y^r:=\tilde W^{r+R}$. In particular, if $|C(v)|\le \chi N$, then $Y^r=W^{r+R}$. Now we may sum \eqref{eq:Yrdirbd} over $r$.

\begin{equation}\label{eq:Ydirbd}\left|\left| \E{\sum_{r=0}^{N-R-1} Y^r} - \mu(\kappa\circ\pi) \left|\left| \E{\sum_{m=0}^{N-R-1} Y^r}\right|\right|_1 \right|\right|_1 \le \frac{\delta}{2} \left|\left|\E{\sum_{r=0}^{N-R-1}Y^r}\right|\right|_1.\end{equation}

By considering \eqref{eq:WrplusR} in the case $r=0$, we have
\begin{equation}\label{eq:lowerboundsumYr}0<(\eta^2-\theta)^R\le\E{Y^0}\le \left|\left|\E{\sum_{r=0}^{N-R-1}Y^r}\right|\right|_1.\end{equation}

We now deal with the case when $|C(v)|>\chi N$. By construction, we have the very crude bound,
$$||W^R+W^{R+1}+\ldots +W^{N-1}||_1 \le N.$$

Then, for each $r$, there are $R+1$ possibilities for the value of $Y^r=\tilde W^{r+R}$ in terms of $W$, depending on when $S^{r+m}$ first exceeds $\chi N$, as given by \eqref{eq:defntildeW}. So we have another crude bound,
$$Y^r\le  W^{r+R} + W^{r+R-1}(\kappa\circ \pi) +\ldots + W^r (\kappa\circ \pi)^R,$$
since all of these quantities are non-negative. Furthermore, since all entries of $\kappa\circ \pi$ are at most $T$, we obtain,
$$\left|\left|\sum_{r=0}^{N-R-1}Y^r \right|\right|_1 \le (1+(kT)+\ldots + (kT)^R ) N \le (kT)^{R+1} N.$$

Therefore, both the following hold:
\begin{align}
\left|\left| \E{\sum_{r=0}^{N-R-1} Y^r} - \E{\sum_{r=R}^{N-1} W^r} \right|\right|_1 &\le \left(1+(kT)^{R+1}\right) N \Prob{|C(v)|>\chi N},\label{eq:sumYvssumW}\\
\left|\left| \E{\sum_{r=0}^{N-R-1} Y^r} - \E{\mathds{1}_{A_\chi}\sum_{r=0}^{N-R-1} Y^r} \right|\right|_1 &\le \left(1+(kT)^{R+1}\right) N \Prob{|C(v)|>\chi N}.\label{eq:sumYvssum1AY}
\end{align}

By Proposition \ref{LDPprop} the RHS of \eqref{eq:sumYvssumW} is, for large enough $N$, much smaller than all the terms in \eqref{eq:Ydirbd} (recall from \eqref{eq:lowerboundsumYr} that we have a positive lower bound on the RHS of \eqref{eq:Ydirbd}), and so we may replace $\E{\sum_{r=0}^{N-R-1} Y^r}$ with $ \E{\sum_{r=R}^{N-1} W^r}$ to conclude
\begin{equation}\label{eq:Wdirbd}\left|\left| \E{\sum_{r=R}^{N-1} W^r} - \mu(\kappa\circ\pi) \left|\left| \E{\sum_{r=R}^{N-1} W^r}\right|\right|_1 \right|\right|_1 \le \delta \left|\left|\E{\sum_{r=R}^{N-1} W^r}\right|\right|_1,\end{equation}
for $N\ge N_0=N_0(\delta,\eta,T)\in\N$, as required for \eqref{eq:propsbeyondR}, since $\sum_{r=R}^{N-1}W^r = W^{\ge R}(v)$.

We apply a similar argument using \eqref{eq:sumYvssum1AY} to conclude \eqref{eq:types<chi}.
\end{proof}

\section{Proof of Theorem \ref{weaklimitthm}}\label{weaklimitsection}

Theorem \ref{weaklimitthm} concerns a family of frozen percolation processes $(\mathcal{G}^{N,p^N,\kappa,\lambda(N)}(t))_{t\ge 0}$ with $k$ types, for which $||p^N||_1=N$, and $p^N/N\stackrel{d}\rightarrow \pi(0)$. In addition, $\lambda(N)$ satisfies the critical scaling \eqref{eq:lambdascaling}. From now on, we abbreviate as $\mathcal{G}^N(t)$. We also have $\pi^N(t)$ and $\Phi^N(t)$ defined as before. Defining
$$v^N_\ell(t):= \frac{1}{N}\#\left\{\text{alive vertices in }\mathcal{G}^N(t)\text{ with component size }\ell \right\},$$
we obtain from Proposition \ref{BJRweaklimit} and Scheff\'e's Lemma that $v^N(0)\stackrel{d}\rightarrow v(0)$ in $\ell_1$, where $v_\ell(0)=\Prob{|\Xi^{\pi(0),\kappa}|=\ell}$. Then Theorem \ref{FPconvtosmol} shows that $\Phi^N\rightarrow\Phi$ in distribution uniformly on $[0,T]$. Here, the total-mass function $\Phi$ satisfies $\Phi(t)=1$ for $t\in[0,T_{\mathrm{gel}}]$ and $\Phi$ is strictly decreasing and uniformly continuous on $(T_{\mathrm{gel}},\infty)$. Note that $T_{\mathrm{gel}}$ depends on $\kappa$ and $\pi(0)$ through $v(0)$ using \eqref{eq:Tgel}.

We will show that $\pi^N$ has a limit $\pi$, which satisfies the conditions to be a FP type flow with initial kernel $\kappa$, initial distribution $\pi(0)$, and critical time $t_c=T_{\mathrm{gel}}$. Since we may construct a suitable family of frozen percolation processes with $k$ types for any $\pi(0)$ and $\kappa$, this will also complete the proof of Theorem \ref{DEtheorem}. (Recall that so far we have only shown the uniqueness result as Proposition \ref{uniqueprop}.)

\subsubsection*{Outline of argument}
First we check that the sequence of processes $(\pi^N(\cdot))$ is tight in $\mathbb{D}^k([0,T])$, and that every component of any weak limit is bounded away from zero. We deduce from Theorem \ref{FPconvtosmol} that weak limits are continuous and after $t_c$ are strictly decreasing and critical. We will argue that seeing supercritical periods give rise to jumps, and subcritical periods are locally constant in the weak limits, neither of which is allowed.

\par
Finally, we show that any weak limit satisfies the equation \eqref{eq:flowDE}. Our argument will be that in the limit the majority of mass is lost as a result of freezing large components, and the proportion of types within in a large component is well-approximated by the appropriate left-eigenvector, precisely as shown in Theorem \ref{typecontrolthm}.

\subsection{Tightness and simple properties of weak limits}
Throughout this section, we assume $T>0$, and that both the initial kernel $\kappa$ and the initial distribution $\pi(0)$ are fixed.

\subsubsection{Tightness and Theorem \ref{DEtheorem}}
Note that each $\pi^N$ is cadlag, and non-increasing, and $\pi^N(0)$ lies in a compact set, since it satisfies $||\pi^N_i(0)||_1=1$. It follows that the set of possible trajectories of any $(\pi^N(t))_{t\in[0,T]}$ is compact in $\mathbb{D}^k([0,T])$, and so certainly the sequence of processes $(\pi^N(\cdot))$ is tight.

\par
Therefore, $(\pi^N(\cdot))$ has weak limits. The remainder of this proof of Theorem \ref{weaklimitthm} establishes that any such weak limit satisfies the conditions of Definition \ref{defntypeflow} to be a frozen percolation type flow with the correct initial conditions. As a result, the full statement of Theorem \ref{DEtheorem} follows from Section \ref{uniquesection} and the proof of Theorem \ref{weaklimitthm} to follow in this section.

\subsubsection{Before $t_c$}
From now on, let $\pi$ be any weak limit of $(\pi^N)$ in $\mathbb{D}^k([0,T])$ as $N\rightarrow\infty$. By assumption $\pi^N(0)\stackrel{d}\rightarrow \pi(0)$, so the two meanings of $\pi(0)$ are consistent!

We know that $\Phi^N\stackrel{d}\rightarrow 1$ on $[0,t_c]$. Therefore, since for each $i\in[k]$, $\pi_i^N$ is non-increasing, the same must be true for each $\pi_i$. Therefore $\sum_{i\in[k]}\pi_i^N (t)\stackrel{d}\rightarrow 1$ for $t\in[0,t_c]$ implies $\pi^N(t)\stackrel{d}\rightarrow \pi(t)$ for the same range of $t$. In particular, any weak limit $\pi$ satisfies $\pi(t)=\pi(0)$ for $t\in[0,t_c]$, as required.

\subsubsection{Continuity}
Again, we know $\Phi^N\stackrel{d}\rightarrow \Phi$, which is continuous. Any weak limit $\pi$ must satisfy $||\pi(t)||_1=\Phi(t)$ for $t\in[0,T]$, and every component $\pi_i(t)$ is non-increasing with $t$. Therefore, if with positive probability, for some $i\in[k]$, $\pi_i(\cdot)$ has a (downward) jump, so does $\Phi(\cdot)$. This is a contradiction, and thus $\pi(\cdot)$ is almost surely continuous.

\subsubsection{Lower bounds on $\pi^N(T)$}
As in the analysis of type flows, in order to use the Lipschitz condition, it is convenient to show the following lemma, which asserts that the proportion of alive vertices of each type is bounded below in probability uniformly on compact time intervals.

\begin{lemma}\label{defnetalemma}For any $T>0$, there exists $\eta=\eta(T)>0$ such that
\begin{equation}\label{eq:notypetoosmall}\lim_{N\rightarrow\infty}\Prob{\exists i\in[k]\text{ s.t. }\pi_i^N(T)<\eta}=0.\end{equation}
\begin{proof}
We consider the proportion of isolated alive vertices of type $i$ in the frozen percolation process, as a lower bound on the proportion of all alive vertices of type $i$. We use a second-moment method, under a coupling with the classical Erd\H{o}s--R\'enyi dynamics with no freezing.

\par
Each possible edge carries an exponential clock with parameter $1/N$. Because of the dynamics of the frozen percolation process, sometimes we do not add the edge when the corresponding clock rings, because at least one of the incident vertices is already frozen. We say a vertex $v$ is \emph{highly isolated} at time $T$ if it was isolated in $\mathcal{G}^N(0)$, and none of the $N-1$ clocks on edges incident to $v$ ring before time $T$. Certainly if a vertex is highly isolated, then it is also isolated, provided it is alive.

\par
Let $v$ be a uniformly chosen vertex in $[N]$, and let $\mathcal{H}^N_v(T,i)$ be the event that $v$ has type $i$, and is alive and highly isolated at time $T$ in $\mathcal{G}^N(T)$. For $\mathcal{H}^N_v(T,i)$ to hold, $v$ must be assigned type $i$; and $v$ must be isolated in the initial graph $\mathcal{G}^N(0)$; and none of the $N-1$ clocks on edges incident to $v$ may ring before time $T$; and $v$ must not be struck by lightning. So
$$\Prob{\mathcal{H}^N_v(T,i)\,\big|\, \pi^N(0)} =  \pi^N_i(0) \left( \prod_{j=1}^k \exp\left(-\kappa_{i,j}\left[ \pi^N_j(0)-\tfrac{1}{N}\mathds{1}_{\{i=j\}}\right]  \right) \right)\cdot \left( e^{-T/N}\right)^{N-1} \cdot e^{-\lambda(N)T},$$
and since $\pi^N(0)\stackrel{d}\rightarrow \pi(0)$ as $N\rightarrow\infty$, we have
$$\Prob{\mathcal{H}^N_v(T,i)}\rightarrow \pi_i(0)\alpha_i,\qquad\text{where }\alpha_i:= \exp\left( -T-\sum_{j=1}^k \left[\kappa\circ \pi(0)\right]_{i,j} \right).$$

Now let $v,w$ be a uniformly chosen pair of distinct vertices in $[N]$, and let $\mathcal{H}^N_{v,w}(T,i)$ be the event that both $v$ and $w$ have type $i$, and are alive and highly isolated at time $T$. Similarly,
\begin{align*}
\Prob{\mathcal{H}^N_{v,w}(T,i)\,\big|\, \pi^N(0)} &= \pi^N_i(0)\left[ \pi^N_i(0) - \tfrac{1}{N}\right] \left( \prod_{j=1}^k \exp\left(-\kappa_{i,j}\left[2\pi_j^N(0) - \tfrac{2}{N} \mathds{1}_{\{i=j\}} \right] \right) \right)\\
&\quad \times \left(e^{-T/N}\right)^{2N-3} \cdot e^{2\lambda(N) T},
\end{align*}
from which as before we have, as $N\rightarrow\infty$,
$$\Prob{\mathcal{H}^N_{v,w}(T,i)}\rightarrow \pi_i(0)^2 \alpha_i^2.$$

Now let $H^N(T,i)$ be the number of alive, highly isolated vertices with type $i$ in $\mathcal{G}^N(T)$. We have $\E{\frac{H^N(T,i)}{N}}\rightarrow \pi_i(0)\alpha_i$ and $\var{{\frac{H^N(T,i)}{N}}}\rightarrow 0$.

\par
So for any $\eta\in(0,\pi_i(0)\alpha_i)$, applying Chebyshev's inequality to $\frac{H^N(T,i)}{N}$,
$$\limsup_{N\rightarrow\infty}\Prob{\pi_i^N(T)<\eta}\le \limsup_{N\rightarrow\infty}\Prob{\tfrac{H^N(T,i)}{N}<\eta}\le \limsup_{N\rightarrow\infty} \frac{\var{{\frac{H^N(T,i)}{N}}}}{(\pi_i(0)\alpha_i-\eta)^2}=0.$$
The statement \eqref{eq:notypetoosmall} follows by taking $\eta<\min_i \pi_i(0)\alpha_i$.
\end{proof}

\end{lemma}

\subsection{Weak limits are critical after $t_c$}

We now show that for any weak limit $\pi$, the criticality condition $\rho(\kappa(t)\circ \pi(t))=1$ holds for all $t\ge t_c$. We first show that this eigenvalue cannot ever be greater than one, and then that it cannot be less than one. In both cases, the argument is by contradiction. If $\mathcal{G}^N(t)$ is ever supercritical, then with high probability giant components will be frozen, and so weak limits of $\Phi^N$ will not be continuous. If $\mathcal{G}^N(t)$ is subcritical, then not enough vertices will be frozen to ensure weak limits of $\Phi^N$ are strictly decreasing.

\subsubsection{Weak limits are never supercritical}

We want to control the size of the giant component, uniformly among relevant kernel-distribution pairs $(\pi,\kappa)$ for which $\rho(\kappa\circ\pi)\ge 1+\epsilon$. To reduce this problem to finite number of pairs, the following minorisation lemma will be useful.

\begin{lemma}\label{minorisation}
For any $0<\bar\Lambda<\Lambda$, and $K<\infty$ there exist $M\in \N$, and $\pi^{(1)},\ldots,\pi^{(M)}\in\Pi_{\le 1}$ and kernels $\kappa^{(1)},\ldots,\kappa^{(M)}\in\R_{\ge 0}^{k\times k}$ such that
\begin{itemize}
\item $\rho(\kappa^{(m)}\circ \pi^{(m)})=\bar \Lambda$ for each $m\in[M]$;
\item for any subdistribution $\pi\in\Pi_{\le 1}$ and kernel $\kappa\in[0,K]^{k\times k}$ with $\rho(\kappa\circ \pi)\ge \Lambda$, there is some $m\in[M]$ for which $\pi^{(m)}\le \pi$ and $\kappa^{(m)}\le \kappa$.
\end{itemize}
\end{lemma}

The proof of this non-probabilistic lemma is postponed to Section \ref{majminsection}.

\begin{prop}\label{neversupercritprop}For any $\epsilon>0$,
\begin{equation}\label{eq:limnotsupercrit}\Prob{\sup_{t\in [0,T]} \rho(\kappa(t)\circ \pi(t)) \ge 1+\epsilon}=0.\end{equation}
\begin{proof}The principal eigenvalue $\rho(\cdot)$ is continuous. The kernel $\kappa(\cdot)$ is continuous, and we have shown that $\pi(\cdot)$ is almost surely continuous. On the event $\{\sup_{t\in [0,T]} \rho(\pi(t)\circ \kappa(t)) \ge 1+\epsilon\}$, either $\pi$ has a discontinuity, or there exists a time-interval of positive width, during which $\rho\ge 1+\epsilon/2$. So either \eqref{eq:limnotsupercrit} holds, or there exists a fixed time $s\in[0,T)$, and an infinite subsequence $\mathcal{N}\subset \N$ such that
\begin{equation}\label{eq:superc4contrad}\liminf_{\substack{N\rightarrow\infty\\ N\in \mathcal{N}}}\Prob{\rho\left(\kappa(s)\circ \pi^{N}(s)\right)\ge 1+\epsilon/2}> 0.\end{equation}

We assume \eqref{eq:superc4contrad} holds, and apply Lemma \ref{minorisation}. We obtain that there exist $M\in\N$ and $\pi^{(1)},\ldots,\pi^{(M)}\in \Pi_{\le 1}$ and kernels $\kappa^{(1)},\ldots,\kappa^{(m)}\in \R_{\ge 0}^{k\times k}$ such that $\rho(\kappa^{(m)}\circ \pi^{(m)})=1+\epsilon/3$, and for any $\pi\in\Pi_{\le 1}$ and $\kappa\in [0,\kappa_{\max}+T]^{k\times k}$ with $\rho(\kappa\circ \pi)\ge 1+\epsilon/2$, there exists $m\in[M]$ such that $\pi^{(m)}\le \pi$ and $\kappa^{(m)}\le \kappa$.

\par
Recall that $(\mathcal F^N(t))_{t\ge 0}$ is the natural filtration of $(\pi^N(t))$. In particular, the event $\{\rho(\kappa(s)\circ \pi^N(s))\ge 1+\epsilon/2\}$ is $\mathcal{F}^N(s)$-measurable. On this event, at least one of the events $\{\pi^{(m)}\le \pi^N(s)\}$ holds. From Proposition \ref{GNtdistprop}, conditional on $\mathcal{F}^N(s)$, up to labelling, $\mathcal{G}^N(s)$ has the same distribution as $G^N(N\pi^N(s),\kappa)$. Therefore, for any $\theta\in(0,1)$,
\begin{align}
&\Prob{L_1\left(\mathcal{G}^N(s)\right)\ge \theta N \,\Big|\, \rho(\kappa(s)\circ \pi^N(s))\ge 1+\epsilon/2}\nonumber\\
&\qquad\ge \min_{m\in[M]} \Prob{L_1\left(G^N(\fl{ N\pi^{(m)}},\kappa^{(m)})\right)\ge \theta N},\label{eq:L1stochdom}
\end{align}
where the floor function is applied component-wise. However, for each $m\in[M]$, Proposition \ref{BJRL1enhanced} says that, as $N\rightarrow\infty$,
$$\frac{1}{N}L_1\left(G^N((\fl{N\pi^{(m)}},\kappa^{(m)})\right) \quad \stackrel{d}{\rightarrow}\quad  \sum_{i\in[k]}\pi_i \zeta_i^{\pi^{(m)},\kappa^{(m)}}>0,$$
for each $m\in[M]$. We take $\theta>0$ satisfying
$$\theta< \min_{m\in[M]}\Prob{|\Xi^{\pi^{(m)},\kappa^{(m)}}|=\infty}.$$
Returning to \eqref{eq:L1stochdom} with this value of $\theta$, we find
$$\lim_{N\rightarrow\infty}\Prob{L_1\left(\mathcal{G}^N(s)\right)\ge \theta N \,\Big|\, \rho(\kappa(s)\circ \pi^N(s))\ge 1+\epsilon/2} =1.$$
So, if \eqref{eq:superc4contrad} holds, we have
$$\liminf_{\substack{N\rightarrow\infty\\ N\in\mathcal{N}}} \Prob{L_1\left(\mathcal{G}^N(s)\right) \ge \theta N}>0.$$
Conditional on the event $\left\{L_1\left(\mathcal{G}^N(s)\right)\ge \theta N\right\}$, the probability that this largest component is not struck by lightning before any fixed time $s'>s$ vanishes as $N\rightarrow\infty$, since the lightning rate $\lambda(N)\gg \frac{1}{N}$. So
$$\liminf_{\substack{N\rightarrow\infty\\ N\in\mathcal{N}}}\Prob{\Phi^N(s+)\le \Phi^N(s)-\theta}>0,$$
and so the same holds for the limit,
$$\Prob{\Phi(s+)\le \Phi(s)-\theta}>0,$$
which contradicts the almost-sure continuity of any weak limit $\Phi$. So \eqref{eq:limnotsupercrit} holds.
\end{proof}
\end{prop}

\subsubsection{Weak limits are not subcritical after $t_c$}

We start with a lemma concerning the expected size of the component of a uniformly chosen vertex in a subcritical inhomogeneous random graph. The final step includes a bound which is rather weak, but will be sufficient for the main proposition which follows.

\begin{lemma}\label{subcritECvbd}
Fix $N\in \N$, $p\in\N_0^k$ and $\kappa\in\mathbb{R}_{\ge 0}^{k\times k}$ satisfying $\rho(\kappa\circ p/N)<1$. Let $C(v)$ be the component containing a uniformly chosen vertex in $G^N(p,\kappa)$. Then
$$\E{|C(v)|} \le  \frac{1}{\kappa_{\min}||\pi||_1} \cdot \frac{\rho(\kappa\circ\pi)}{1-\rho(\kappa\circ \pi)},$$
where $\kappa_{\min}:= \min_{i,j\in[k]}\kappa_{i,j}$, and $\pi:= p/N$. 
\begin{proof}
Set $\bar \pi:= \pi/||\pi||_1= p/||p||_1$. Consider the branching process tree with $k$ types, $\bar \Xi^{\pi,\kappa}$, where the type of the root has distribution $\bar \pi$, and the offspring distributions are the same as for $\Xi^{\pi,\kappa}$. By coupling Binomial and Poisson distributions (see, for example, \cite{vdHRGCN} \textsection 2.3), we obtain $|C(v)|\le_{st} |\bar\Xi^{\pi,\kappa}|$, and in particular, $\E{|C(v)|}\le  \E{|\bar\Xi^{\pi,\kappa}|}$.

Similarly, consider another branching process tree with $k$ types, $\hat \Xi^{\pi,\kappa}$, where the type of the root has distribution $\mu(\kappa\circ \pi)$, and the offspring distributions are again the same as for $\Xi^{\pi,\kappa}$. By considering the number of offspring at each generation of $\hat \Xi^{\pi,\kappa}$ we have
$$\E{|\hat \Xi^{\pi,\kappa}|}= 1+\rho(\kappa\circ \pi)+\rho(\kappa\circ \pi)^2+\ldots = \frac{1}{1-\rho(\kappa\circ \pi)}.$$
However, the distribution of $\bar\Xi^{\pi,\kappa}$ conditional on the tree being non-empty and the root having type $i$ is the same as the distribution of $\hat \Xi^{\pi,\kappa}$ conditional on the root having type $i$. Therefore, by the law of total expectation,
$$\E{|\bar\Xi^{\pi,\kappa}|}\le \max_{i\in[k]}\frac{\pi_i/||\pi||_1}{\mu_i(\kappa\circ \pi)}\E{|\hat \Xi^{\pi,\kappa}|}.$$
However, since $\mu(\kappa\circ \pi)$ is a left-eigenvector, we have
$$\frac{\mu_j(\kappa\circ\pi)}{\pi_j}= \frac{1}{\rho(\kappa\circ \pi)} \sum_{i=1}^k \mu_i(\kappa\circ \pi) \kappa_{i,j}\ge \frac{\kappa_{\min}}{\rho(\kappa\circ \pi)},\quad j\in[k],$$
and the result follows immediately.
\end{proof}
\end{lemma}

\begin{prop}For any $\epsilon>0$,
\begin{equation}\label{eq:limnotsubcrit}\Prob{\sup_{t\in [t_c,T]} \rho(\pi(t)\circ \kappa(t)) \le 1-\epsilon}=0.\end{equation}

\begin{proof}
By the same argument as in Proposition \ref{neversupercritprop}, either \eqref{eq:limnotsubcrit} holds, or there exists $s\in[t_c,T)$ and an infinite subsequence $\mathcal{N}\subset \N$ such that
$$\liminf_{\substack{N\rightarrow\infty \\ N\in\mathcal{N}}} \Prob{\rho\left(\kappa(s)\circ \pi^N(s)\right) \le 1-\epsilon/2 } >0.$$
Now, we can choose $\delta>0$ such that $s+\delta<T$ and
$$\frac{\kappa_{i,j}(s+\delta)}{\kappa_{i,j}(s)} \le \frac{1-\epsilon/3}{1-\epsilon/2},\quad \forall i,j\in[k].$$
Since $\rho(\cdot)$ is increasing as a function of each entry of its argument (which is discussed in more detail in \eqref{eq:CWformula})
\begin{equation}\label{eq:upperevaluebd}\liminf_{\substack{N\rightarrow\infty \\ N\in\mathcal{N}}} \Prob{\rho\left(\kappa(s+\delta)\circ \pi^N(s)\right) \le 1-\epsilon/3 } >0.\end{equation}
We now consider how many vertices are frozen during the time-interval $[s,s+\delta]$, in expectation. By construction of the lightning processes, and Proposition \ref{GNtdistprop}, for any $t>0$, it is the case that conditional on $\mathcal{F}^N(t-)$ and the event that an alive vertex is struck by lightning at time $t$, the number of vertices frozen $\left[\Phi^N(t-) - \Phi^N(t)\right]$ has the same law as $|C(1)|$ in the IRG $G^N(N\pi^N(t-),\kappa(t))$. In particular, in our setting, for any lightning strike at time $s'\in[s,s+\delta]$,
\begin{align*}
&\E{\Phi^N(s'-) - \Phi^N(s')\,\Big|\, \mathcal{F}^N(s'-), \Phi^N(s'-)-\Phi^N(s')>0}\\
&\qquad \le \frac{1}{N}\E{|C(1)|\text{ in }G^N(N\pi^N(s'-),\kappa(s'))\,\Big|\, \mathcal{F}^N(s'-)}\\
&\qquad \le \frac{1}{N}\E{ |C(1)|\text{ in }G^N(N\pi^N(s),\kappa(s+\delta))\, \Big|\, \mathcal{F}^N(s'-)},\\
\intertext{almost surely, since $|C(1)|$ is an increasing function of the graphs. But the quantity in the final expectation is actually $\F^N(s)$-measurable, and so}
&\E{\Phi^N(s'-) - \Phi^N(s')\,\Big|\, \mathcal{F}^N(s'-), \Phi^N(s'-)-\Phi^N(s')>0}\\
&\qquad\le \frac{1}{N}\E{ |C(1)|\text{ in }G^N(N\pi^N(s),\kappa(s+\delta))\, \Big|\, \mathcal{F}^N(s)}.
\end{align*}
In particular, this upper bound is independent of behaviour on the interval $[s,s')$. Since the process recording all lightning strikes is dominated by a Poisson process with rate $N\lambda(N)$, we obtain
\begin{equation}\label{eq:freezingoninterval}
\E{\Phi^N(s)-\Phi^N(s+\delta) \,\Big|\, \mathcal{F}^N(s)} \le \delta\lambda(N) \E{ |C(1)|\text{ in }G^N(N\pi^N(s),\kappa(s+\delta))\, \Big|\, \mathcal{F}^N(s)}.\end{equation}

Using Lemma \ref{subcritECvbd}, the expectation of this component size conditional on $\mathcal{F}^N(s)$ is at most
$$\frac{1}{\delta ||\pi^N(s)||_1} \cdot \frac{\rho\left(\kappa(s+\delta)\circ \pi^N(s)\right)}{1-\rho\left(\kappa(s+\delta)\circ \pi^N(s)\right)},$$
almost surely. Consider $\eta$ as given by Lemma \ref{defnetalemma}, and define the event
$$\mathcal{A}^N:= \left\{ \rho\left(\kappa(s+\delta)\circ \pi^N(s)\right) \le 1-\epsilon/3, \, \Phi^N(s)\ge k\eta\right\},$$
which is certainly $\mathcal{F}^N(s)$-measurable. By \eqref{eq:notypetoosmall} and \eqref{eq:upperevaluebd}, as $\mathcal{N}\ni N\rightarrow\infty$, $\liminf \Prob{\mathcal{A}^N}>0$. But then using \eqref{eq:freezingoninterval} we have
$$\E{\Phi^N(s)-\Phi^N(s+\delta) \,\Big |\, \mathcal{A}^N}\le \frac{1}{k\delta \eta} \cdot \frac{3}{\epsilon}\cdot \delta \lambda(N)\ll 1  .$$
It follows by the law of total probability and by Markov's inequality that for any $\theta>0$,
$$\liminf_{\substack{N\rightarrow\infty\\N\in\mathcal{N}}}\Prob{\Phi^N(s)-\Phi^N(s+\delta) \le \theta}>0,$$
and so
$$\Prob{\Phi(s)-\Phi(s+\delta)=0}>0,$$
which contradicts the requirement that any weak limit $\Phi$ is almost surely strictly decreasing on $[t_c,T]$.
\end{proof}
\end{prop}

We have shown that any weak limit $\pi$ is continuous, and satisfies $\rho(\kappa(t)\circ\pi(t))=1$ for $t\ge t_c$, and satisfies $\pi(t)=\pi(0)$ for $t\le t_c$. Thus we have shown that \eqref{eq:pregelation} and \eqref{eq:critcondition} hold.

\subsection{Asymptotic proportions of types of frozen vertices}\label{integralsection}

To complete the proof of Theorem \ref{weaklimitthm} it remains to show that any weak limit satisfies \eqref{eq:flowDE}.

\subsubsection{Weak convergence towards integral equation}
Throughout this final section, $\kappa$ remains fixed, and so $\kappa(t)$ is fixed for all $t\ge 0$. To emphasise this, and for brevity, we will write $\mu(t,\pi)$ for $\mu(\kappa(t)\circ \pi)$.

Suppose we have
\begin{equation}\label{eq:discrinteq} \sup_{t\in[t_c,T]}\left|\left|\pi^N(t_c)-\pi^N(t) + \int_{t_c}^t \mu(s,\pi^N(s-)) \mathrm{d}\Phi^N(s) \right|\right|_1\stackrel{\mathbb{P}}\rightarrow 0,\end{equation}
as $N\rightarrow\infty$. We will show that this is a sufficient requirement for any weak limit $\pi(\cdot)$ to satisfy the following integral version of the differential equation \eqref{eq:flowDE} governing the evolution of the type distribution:
\begin{equation}\label{eq:flowintegral}\pi(t_c)-\pi(t)+\int_{t_c}^t \mu(s,\pi(s))\mathrm{d}\Phi(s)=0,\quad t\in[t_c,T].\end{equation}
This is sufficient for \eqref{eq:flowDE} since $\Phi$ is differentiable on $(t_c,\infty)$, and $\mu(s,\pi(s))$ is almost surely continuous. We have $\Phi^N\rightarrow \Phi$ uniformly on $[0,T]$, and again let $\pi(\cdot)$ be a weak limit of $\pi^N(\cdot)$ along the subsequence $\mathcal{N}\subset\N$. Since $\pi(\cdot)$ and $\kappa(\cdot)$ are continuous, $\mu(\cdot,\pi(\cdot))$ is uniformly continuous on $[0,T]$. Therefore
$$ \sup_{t\in[t_c,T]}\left |\left| \int_{t_c}^t \mu(s,\pi(s)) \mathrm{d}\left[\Phi^N(s)-\Phi(s) \right] \right|\right|_1 \stackrel{\mathbb{P}}\rightarrow 0,$$
as $N\rightarrow\infty$. To conclude \eqref{eq:flowintegral} from \eqref{eq:discrinteq}, it remains to show that
\begin{equation}\label{eq:piN-vspi} \sup_{t\in[t_c,T]} \left|\left| \int_{t_c}^T \left[ \mu(s,\pi^N(s-)) - \mu(s,\pi(s)) \right]\mathrm{d}\Phi^N(s)\right|\right|_1\stackrel{\mathbb{P}}\rightarrow 0,\end{equation}
as $\mathcal{N}\ni N\rightarrow\infty$. But certainly for any $t\in[t_c,T]$ we have
$$\left|\left| \int_{t_c}^t \left[ \mu(s,\pi^N(s-)) - \mu(s,\pi(s)) \right]\mathrm{d}\Phi^N(s)\right|\right|_1 \le \int_{t_c}^T \left |\left| \mu(s,\pi^N(s-)) - \mu(s,\pi(s)) \right|\right|_1 \mathrm{d}\Phi^N(s).$$
Consider $\eta>0$ as given by Lemma \ref{defnetalemma}. It follows directly from \eqref{eq:notypetoosmall} that
$$\Prob{\exists i\in[k] \text{ s.t. } \pi_i(T)<\eta}=0.$$
Conditional on $\pi^N_i(s-)\ge \eta$ for all $i\in[k]$, Lemma \ref{muLipprop1} gives, as in \eqref{eq:usingmuLip},
$$ \left |\left| \mu(s,\pi^N(s-)) - \mu(s,\pi(s)) \right|\right|_1\le  (\kappa_{\max}+T)C\big(\eta(t_c\vee\kappa_{\min}),\kappa_{\max}+T\big)||\pi^N(s-)- \pi(s)||_1.$$
Therefore, writing $C$ for $(\kappa_{\max}+T)C\big(\eta(t_c\vee \kappa_{\min}),\kappa_{\max}+T\big)$, on the event\\$\{\pi^N_i(T)\ge \eta, \,\forall i\in[k]\}$,
$$\left|\left| \int_{t_c}^t \left[ \mu(s,\pi^N(s-)) - \mu(s,\pi(s)) \right]\mathrm{d}\Phi^N(s)\right|\right|_1 \le C\int_{t_c}^T || \pi^N(s-) - \pi(s)||_1 \mathrm{d}\Phi^N(s).$$
As $\mathcal{N}\ni N\rightarrow\infty$, both $\pi^N\rightarrow \pi$, and $\Phi^N\rightarrow \Phi$ uniformly in distribution on $[0,T]$, so the RHS vanishes in probability. By Lemma \ref{defnetalemma}, $\Prob{\pi^N_i(T)\ge \eta \;\forall i\in[k]}\rightarrow 1$. Thus \eqref{eq:piN-vspi} follows, and we may conclude \eqref{eq:flowintegral} from \eqref{eq:discrinteq}. It remains to show \eqref{eq:discrinteq}. We will show \eqref{eq:discrinteq} in the next section, after a preliminary result.

\subsubsection{A result about coupled processes}
First, we prove a general result about coupled processes.

\par
The motivation for the setup is the following. Every time a component is frozen in the multitype frozen percolation process, the distribution of types in this frozen component is not \emph{exactly} the same as the left-eigenvector of the appropriate kernel, but the difference is close to zero so long as the component is fairly large. The expression on the LHS of \eqref{eq:discrinteq} records the accumulation of this error. Each time $\Phi^N$ has a downward jump, the expected extra error accumulated is small relative to the expected size of the jump of $\Phi^N$. The following result will show that this is enough to conclude that the total error is small in probability, uniformly in time.

\par
For some $N\in\N$, consider $(\xi_m)_{0\le m\le N}$ and $(Y_m)_{0\le m\le N}$, $\R$-valued processes adapted to filtration $\F=(\F_m)_{0\le m\le N}$. We assume that $\xi_0=Y_0=0$, and that $(\xi_m)$ is non-decreasing. We assume also that $\xi_N\le 1$, and that for some $\delta\in(0,1)$ and $K\in \N$,
\begin{equation}\label{eq:xiboundsY}|Y_{m+1}-Y_m| \le \xi_{m+1}-\xi_m\le \frac{1}{K^2},\quad \text{a.s.}\quad m= 0,1,\ldots,N-1,\end{equation}
\begin{equation}\label{eq:bdEY}\text{and}\quad \big|\E{Y_{m+1}-Y_m | \F_{m}}\big|\le \delta  \E{\xi_{m+1}-\xi_m | \mathcal{F}_{m}},\quad \text{a.s.}\quad m= 0,1,\ldots,N-1.\end{equation}

That is, the increments of $\xi$ are bounded, and dominate the increments of $Y$. Furthermore these increments of $Y$ have have smaller expectation than those of $\xi$, uniformly in time and the history of the process.

\begin{lemma}\label{Yxilemma}Whenever \eqref{eq:xiboundsY} and \eqref{eq:bdEY} hold, we have:
\begin{equation}\label{eq:finalbound}\E{\sup_{0\le m\le N} |Y_n|}\le \frac{2}{K}+\delta. \end{equation}
\begin{proof}
We consider the Doob--Meyer decomposition of the process $(Y_m)$. That is,
$$W_0:=0,\qquad W_{m+1}:= W_m+ Y_{m+1} - \E{Y_{m+1}\,\big|\, \mathcal{F}_m},\quad m\ge 0,$$
$$A_0:=0,\qquad A_{m+1}:= A_m + \E{Y_{m+1}-Y_m\,\big|\,\mathcal{F}_m},\quad m\ge 0,$$
for which $(W_m)$ is an $\mathcal{F}$-martingale, and $(A_m)$ is a predictable process, and $Y_m=W_m+A_m$. All the statements which follow hold almost surely. First we consider $(A_m)$. Using \eqref{eq:bdEY}, we have
$$\left|A_{m+1}-A_m\right|\le \delta \E{\xi_{m+1}-\xi_m \,\big|\,\mathcal{F}_m},$$
from which, 
\begin{equation}\label{eq:boundsupAm}
\E{\sup_{0\le m\le N}\left |A_m\right|}\le \E{ \sum_{m=0}^{N-1} \left| A_{m+1}-A_m\right| } \le \delta\E{\xi_N}\le\delta.\end{equation}

Now we turn to $(W_m)$. Certainly, for any $0\le m\le N-1$, conditional on $\mathcal{F}_m$,
$$W_{m+1}-W_m = Y_{m+1}-Y_m - \E{ Y_{m+1}-Y_m \,\big|\,\mathcal{F}_m},$$
and so
$$\E{\left(W_{m+1}-W_m\right)^2 \,\Big|\, \mathcal{F}_m}\le \E{\left(Y_{m+1}-Y_m\right)^2 \,\Big|\,\mathcal{F}_m}.$$
Using \eqref{eq:xiboundsY}, for any $0\le m\le N-1$, 
$$\E{\left(W_{m+1}-W_m\right)^2 \,\Big|\, \mathcal{F}_m} \le  \E{\left(\xi_{m+1}-\xi_m\right)^2\,\Big|\,\mathcal{F}_m}\le \frac{1}{K^2}\E{\xi_{m+1}-\xi_m\,\big|\,\mathcal{F}_m}.$$
Since $(W_m)$ is a martingale bounded in $L^2$, by orthogonality of increments (see \textsection12.1 in \cite{Williamsbook}),
$$\E{W_N^2}= \sum_{m=0}^{N-1}\E{\left(W_{m+1}-W_m\right)^2} \le \frac{1}{K^2}\sum_{m=0}^{N-1}\E{\xi_{m+1}-\xi_m} \le \frac{1}{K^2},$$
since $\xi_N\le 1$. Finally, using Doob's $L^2$ inequality,
\begin{equation}\label{eq:boundsupWm}
\E{\sup_{0\le m\le N}|W_m|}\le\sqrt{\E{\sup_{0\le m\le N}W_m^2}} \le \sqrt{4\E{W_N^2}}\le\frac{2}{K}.  
\end{equation}
Since $Y_m=W_m+A_m$, it follows immediately from \eqref{eq:boundsupAm} and \eqref{eq:boundsupWm} that
$$\E{\sup_{0\le m\le N}|Y_m|}\le \frac{2}{K}+\delta,$$
as required.
\end{proof}
\end{lemma}

\subsection{Decomposition via freezing times}
We now prove \eqref{eq:discrinteq}, which is equivalent to
\begin{equation}\label{eq:discrinteq2} \sup_{t\in[t_c,T]} \left|\left| \int_{t_c}^t \left[\mu(s,\pi^N(s-)) \mathrm{d}\Phi^N(s) - \mathrm{d}\pi^N(s)\right] \right|\right|_1 \stackrel{\mathbb{P}}\rightarrow 0.\end{equation}

\par
To address this, we categorise each frozen vertex by its type and by its distance from the associated vertex which was struck by lightning. This will allow us to use Theorem \ref{typecontrolthm}. In the process $\mathcal{G}^N$, for each vertex $v\in[N]$, say $s_v$ is the time at which $v$ is frozen, as a result of some vertex $w$ being struck by lightning. (Note that $w=v$ is possible.) Define $d(v)=d(v,w)$ to be the graph distance in $\mathcal{G}^{N}(s_v)$ between $v$ and this vertex $w$.

\par
We now define for each $i\in[k]$ and any $r\in\{0,1,\ldots,N-1\}$,
\begin{equation}\label{eq:defnPsiN}\Psi^N(r,i,t) = \frac{1}{N}\#\left\{v\in[N]: \mathrm{type}(v)=i,\, s_v\in[0,t]\text{ and }d(v)= r \right\}.\end{equation}
Also define $\Psi^N(r,t):= \sum_{i=1}^k \Psi^N(r,i,t)$, the total proportion of vertices of any type frozen up to time $t$ which were distance $r$ from the vertex struck by lightning. We have
$$\sum_{r=0}^{N-1}\mathrm{d}\Psi^N(r,s)=-\mathrm{d}\Phi^N(s), \quad \sum_{r=0}^{N-1} \mathrm{d}\Psi^N(r,i,s)=-\mathrm{d}\pi^N_i(s),$$
and so \eqref{eq:discrinteq2} is further equivalent to
$$\sup_{t\in[t_c,T]}\left| \int_{t_c}^T \sum_{r=0}^{N-1}\left[ \mu_i(s,\pi^N(s-))\mathrm{d}\Psi^N(r,s) - \mathrm{d}\Psi^N(r,i,s) \right]\right| \stackrel{\mathbb{P}}\rightarrow 0,\quad \forall i\in[k],$$
as $N\rightarrow\infty$. Therefore, to prove \eqref{eq:discrinteq} and complete the proof of Theorem \ref{weaklimitthm} it will suffice to show the following lemma.

\begin{lemma}\label{freezedecomplemma}For each type $i\in[k]$,
\begin{equation}\label{eq:limsofPsiN}\lim_{N\rightarrow\infty}  \E{ \sup_{t\in[t_c,T]}\left| \int_{t_c}^t \sum_{r=0}^{N-1}\left[ \mu_i(s,\pi^N(s-))\mathrm{d}\Psi^N(r,s) - \mathrm{d}\Psi^N(r,i,s) \right]\right| }=0.
\end{equation}

\begin{proof}
Throughout the proof, we fix $i\in[k]$. We start by showing that small values of $r$ do not contribute on this scale in the limit. Fix some $R\in\N$. Then, at any time $t\le T$, the expected number of vertices within distance $R-1$ of a uniformly chosen alive vertex in $\mathcal{G}^{N}(t)$ is at most
$$1+\left[N (1-e^{-(\kappa_{\max}+T)/N})\right] + \ldots + \left[N (1-e^{-(\kappa_{\max}+T)/N})\right]^{R-1}.$$

Therefore
\begin{align*}
\E{\sum_{r=0}^{R-1} \Psi^N(r,T)} &\le \frac{1}N\cdot [\lambda(N)N]T \left[1+\left[N (1-e^{-(\kappa_{\max}+T)/N})\right] + \ldots\right.\\
&\qquad\qquad+ \left.\left[N (1-e^{-(\kappa_{\max}+T)/N})\right]^{R-1} \right],
\end{align*}
and since $\lambda(N)$ satisfies the critical scaling \eqref{eq:lambdascaling}, this vanishes as $N\rightarrow \infty$.

So it remains to show that
\begin{equation}\label{eq:limPsiNlargeM}\lim_{N\rightarrow\infty}\E{ \sup_{t\in[t_c,T]} \left| \int_{t_c}^t \sum_{r=R}^{N-1}\left[ \mu_i(s,\pi^N(s-))\mathrm{d}\Psi^N(r,s) - \mathrm{d}\Psi^N(r,i,s) \right]\right| }=0,
\end{equation}
for a fixed value of $R\in\N$ to be chosen shortly.

\par
Recall that $(\mathcal{F}^N(t))_{t\ge 0}$ is the natural filtration of the random type flow process $\pi^N$. We now define $(\bar{\mathcal{F}}^{N}(t))_{t\ge 0}$ to be the natural filtration of the collection of processes
$$\pi^N(\cdot) \quad\text{and}\quad \Psi^N(r,i,\cdot),\text{ for all }r\ge 0,\,i\in[k].$$
Note that, in a multitype frozen percolation process, conditional on the set of alive vertices and their types at time $t$, the graph structure of the frozen vertices is independent of $\mathcal{G}^N(t)$, the graph with types on alive vertices. So, although this filtration $(\bar{\mathcal{F}}^{N})$ is finer than $(\mathcal{F}^N)$, Proposition \ref{GNtdistprop} remains true after replacing conditioning on $\mathcal{F}^N(t)$ with conditioning on $\bar{\mathcal{F}}^N(t)$.

\par
We now use some notation from Section \ref{largecptssection} and Theorem \ref{typecontrolthm}. Recall that for some vertex $v$ in some graph $G$ with $k$ types, we let $W^{\ge R}_i$ be the number of type $i$ vertices in $G$ at distance at least $ R$ from $v$. Now, for each $s\in[0,T]$, and $R\in\{0,\ldots,N-1\}$, conditional on $\bar{\mathcal{F}}^N(s-)$ and the event that there is a lightning strike at time $s$, the distribution of $\sum_{r= R}^{N-1}\left( \Psi^N(r,s) - \Psi^N(r,s-)\right)$ is the same as the distribution of $W^{\ge R}$ corresponding to a uniformly chosen vertex in $G^N(N\pi^N(s-),\kappa(s))$.

\par
We take $\tau_0:=t_c$, and let the times that lightning strikes an alive vertex after $t_c$ be $t_c<\tau_1<\tau_2<\ldots$. Set $\alpha:=\max\{m: \tau_m\le T\}$ to be the number of such lightning strikes until time $T$. Since $\tau_1,\ldots,\tau_\alpha$ are precisely those times $t\in(t_c,T]$ for which $\pi^N(t-)-\pi^N(t)>0$, each $\tau_m$ is an $(\mathcal{F}^N)$-stopping time, and thus an $(\bar{\mathcal{F}}^N)$-stopping time too. Now consider for $m= 0,1,\ldots,\alpha$, the discrete process
\begin{align*}
Y^N_m&:=\int_{t_c}^{\tau_m}\sum_{r\ge R}^{N-1} \left[\mu_i(s,\pi^N(s-))\mathrm{d}\Psi^N(r,s) - \mathrm{d}\Psi^N(r,i,s)\right],\\
&=\sum_{\ell=1}^m \left\{\mu_i(\tau_\ell, \pi^N(\tau_i-))\left[ \sum_{r\ge R} \Psi^N(r,\tau_\ell)-\sum_{r\ge R}\Psi^N(r,\tau_\ell-)\right] \right. \\
&\qquad \left. -\left[ \sum_{r\ge R} \Psi^N(r,i,\tau_\ell)-\sum_{r\ge R}\Psi^N(r,i,\tau_\ell-)\right]\right\}.
\end{align*}
Then $(Y^N_m)_{0\le m\le \alpha}$ is adapted to $\left(\bar{\mathcal{F}}^N(\tau_m)\right)_{0\le m\le \alpha}$, and records the accumulation of error between the true proportion of types lost beyond radius $R$, and the proportion expected from the left-eigenvectors, as successive components are frozen.

\par
We also define, for $m= 0,1,\ldots,\alpha$,
$$\xi^N_m:= \int_{t_c}^{\tau_m} \sum_{r\ge R}\mathrm{d}\Psi^N(r,s)= \sum_{\ell=1}^m\left[ \sum_{r\ge R} \Psi^N(r,\tau_\ell)-\sum_{r\ge R}\Psi^N(r,\tau_\ell-)\right],$$
the discrete process recording the proportion of mass lost beyond radius $R$ after successive lightning strikes. This process $(\xi^N_m)_{0\le m\le \alpha}$ is also adapted to $\left(\bar{\mathcal{F}}^N(\tau_m)\right)_{0\le m\le \alpha}$.

\par
We will now compare the increments of $Y^N$ and the increments of $\xi^N$ in expectation using Theorem \ref{typecontrolthm}. In particular, we will need to exclude the possibility that any component of $\pi^N$ becomes too small, or that $\rho(\pi^N(t)\circ \kappa(t))$ becomes too large. Furthermore, to apply Lemma \ref{Yxilemma} we will have to ignore increments where the total mass lost is too large. All of these events happen with vanishing probability, and the quantities under consideration are uniformly bounded. Rather than condition that none of these events occur (which would affect the individual increments), we will exclude any pathological behaviour step-by-step for each freezing event, so as to preserve the Markov property.

\par
Recall the definition of $\eta$ from Lemma \ref{defnetalemma}. Set $\eta'=\min(\eta, \kappa_{\min}\vee t_c)>0$. Choose some $\delta\in(0,1)$, and consider $\epsilon=\epsilon(\delta,\eta',T+\kappa_{\max})$ and $R=R(\delta,\eta',T+\kappa_{\max})$ as defined in Theorem \ref{typecontrolthm}. Consider the events
$$\Theta^{N,\eta',\epsilon}_m:=\left\{\pi_j^N(\tau_m-)\ge \eta',\forall j\in[k],\, \sup_{t\in[0,\tau_m)} \rho(\pi^N(t)\circ \kappa(t))\le 1+\epsilon\right\},$$
each of which is $\F^N(\tau_m-)$-measurable, and thus also $\bar{\F}^N(\tau_m)$-measurable. On the event $\Theta^{N,\eta',\epsilon}_m$, the graphs $\mathcal{G}^N(s)$ satisfy the conditions of Theorem \ref{typecontrolthm} for all $s\in[0,\tau_m)$. Note that
$$\Theta^{N,\eta',\epsilon}_1\supset \ldots\supset \Theta^{N,\eta',\epsilon}_\alpha \supset \Theta^{N,\eta',\epsilon}:= \left\{\pi_j^N(T)\ge \eta',\forall j\in[k],\, \sup_{t\in[0,T]} \rho(\pi^N(t)\circ \kappa(t))\le 1+\epsilon\right\}.$$
We know from \eqref{eq:notypetoosmall} and \eqref{eq:limnotsupercrit} that
$$\lim_{N\rightarrow\infty} \Prob{\Theta^{N,\eta',\epsilon}}=1.$$

We also have $\chi=M(\eta')$ given by Proposition \ref{LDPprop}. We define
$$\xi^{N,\chi}_m:= \sum_{\ell=1}^m \mathds{1}_{\{\xi^N_\ell-\xi^N_{\ell -1} \le \chi\}}\left(\xi^N_\ell - \xi^N_{\ell-1}\right),\quad m= 0,1,\ldots,\alpha,$$
which counts the proportion of vertices frozen from beyond radius $R$, ignoring those occasions when the number of such vertices is greater than $\chi N$. (Recall that $\xi^N$ has been rescaled like $\Phi^N$, so that losing more than $\chi N$ vertices beyond radius $R$ corresponds to $\xi^N_\ell-\xi^N_{\ell-1}> \chi$.) Analogously, we define
$$Y^{N,\chi}_m := \sum_{\ell=1}^m  \mathds{1}_{\{\xi^N_\ell-\xi^N_{\ell -1}\le \chi\}}\mathds{1}_{\Theta^{N,\eta',\epsilon}_\ell}\left( Y^N_\ell - Y^N_{\ell-1} \right),\quad m=0,1,\ldots,\alpha,$$
which describes the accumulation of error in \eqref{eq:discrinteq} when components of size at most $\chi N$ are frozen, \emph{and} when the graph satisfies the conditions for Theorem \ref{typecontrolthm}. Observe that $\alpha\le N$ by construction, so we also define
$$\xi^{N,\chi}_m=\xi^{N,\chi}_\alpha,\quad Y^{N,\chi}_m=Y^{N,\chi}_\alpha,\quad m=\alpha+1,\ldots,N.$$

This pair of processes $(\xi^{N,\chi},Y^{N,\chi})$ is adapted to the filtration $\mathcal{H}^N=(\mathcal{H}^N_m)_{0\le m\le N}$ defined by $\mathcal H^N_m:=\bar{\F}^N(\tau_{m+1}-)$, for $m< \alpha$ and $H^N_m=\bar{\F}^N(\tau_\alpha)$ for $m\ge\alpha$. Observe that $\xi^{N,\chi}$ is non-decreasing and
$$\left|Y^{N,\chi}_{m+1}- Y^{N,\chi}_m\right|\le \xi^{N,\chi}_{m+1}-\xi^{N,\chi}_m \le \chi.$$

Furthermore, on $\Theta^{N,\eta',\epsilon}_{m+1}$ (which is $\mathcal{H}^N_m$-measurable),
$$\left( \xi^{N,\chi}_{m+1}-\xi^{N,\chi}_m \,\Big|\, \mathcal{H}^N_m \right)\;\equidist\; W^{\ge R}\mathds{1}_{A_{\chi}},$$
where $A_\chi=\{||W^{\ge R}||\le \chi N\}$, with the IRG taken to be $G^N(N\pi^N(\tau_{m+1}-),\kappa(\tau_{m+1}))$. Similarly, again on $\Theta^{N,\eta',\epsilon}_{m+1}$,
$$\left(Y^{N,\chi}_{m+1}-Y^{N,\chi}_m \,\Big|\, \mathcal{H}^N_m\right) \;\equidist\; W^{\ge R}_i \mathds{1}_{A_\chi} - \mu_i(\tau_{m+1},\pi^N(\tau_{m+1}-)) W^{\ge R}\mathds{1}_{A_\chi}.$$
On $(\Theta^{N,\eta',\epsilon}_{m+1})^c$, the increment $\left(Y^{N,\chi}_{m+1}-Y^{N,\chi}_m\,\Big|\, \mathcal{H}^N_m\right)$ is zero. Therefore, taking expectations and applying Theorem \ref{typecontrolthm}, we obtain
$$\left|\E{Y^{N,\chi}_{m+1}- Y^{N,\chi}_m\,\big|\, \mathcal{H}^N_m}\right|\le \delta \E{\xi^{N,\chi}_{m+1}-\xi^{N,\chi}_m\,\big|\, \mathcal{H}^N_m}, \quad \mathrm{a.s.},\quad m\ge 0.$$

Thus, for $K=\fl{\sqrt{\frac{1}{\chi}}}$, the processes $\xi^{N,\chi}$ and $Y^{N,\chi}$ precisely satisfy the conditions for Lemma \ref{Yxilemma}. On the event $\Theta^{N,\eta,\epsilon}$, 
$$ \sup_{t\in[t_c,T]} \left| \int_{t_c}^t \sum_{r=R}^{N-1}\left[ \mu_i(s,\pi^N(s-))\mathrm{d}\Psi^N(r,s) - \mathrm{d}\Psi^N(r,i,s) \right]\right| = \sup_{0\le m\le N} \left|Y^N_m \right|.$$
Therefore, by Lemma \ref{Yxilemma}
$$\limsup_{N\rightarrow\infty}\E{ \sup_{t\in[t_c,T]} \left| \int_{t_c}^t \sum_{r=R}^{N-1}\left[ \mu_i(s,\pi^N(s-))\mathrm{d}\Psi^N(r,s) - \mathrm{d}\Psi^N(r,i,s) \right]\right| }$$
$$\qquad \le 2K^{-1}+\delta + \limsup_{N\rightarrow\infty}2\left(1-\Prob{\Theta^{N,\eta',\epsilon}}\right)=2K^{-1}+\delta.$$

Our choice of $\delta$ was arbitrary, but as we take $\delta\rightarrow 0$, we may assume $\epsilon\rightarrow 0$ and thus $\chi\rightarrow 0$ also. Hence $K\rightarrow\infty$. So \eqref{eq:limPsiNlargeM} and \eqref{eq:limsofPsiN} follow, and the proof of Lemma \ref{freezedecomplemma}, and Theorem \ref{weaklimitthm} is complete.
\end{proof}

\end{lemma}

\section{Limits in time for frozen percolation type flows}\label{timelimitssection}
In this short section, we study the behaviour of a frozen percolation type flow as $t\rightarrow\infty$ and, in particular, prove Theorem \ref{flowtimelimitthm}.

First, we give a quick argument why $\Phi(t)\rightarrow 0$ as $t\rightarrow \infty$ in the framework of type flows. The Collatz--Wielandt formula \cite{Collatz42,Wielandt50} asserts that for $A\in\R^{k\times k}_{\ge 0}$,
\begin{equation}\label{eq:CWformula}\rho(A)= \max_{x\in \R^k_{\ge 0}\backslash \{\mathbf{0}\}} f(x),\quad\text{where }f(x):= \min_{\substack{1\le i\le n\\x_i\ne 0}} \frac{[xA]_i}{x_i}.
\end{equation}
In particular, $\rho(A)$ is non-decreasing as a function of positive matices $A$. Then, from the criticality condition \eqref{eq:critcondition}, for $t\ge t_c$,
$$1=\rho(\kappa(t)\circ \pi(t)) \ge \rho(t\mathbf{1} \circ \pi(t)) = t \rho(\mathbf{1}\circ \pi(t)).$$
Therefore $\rho(\mathbf{1}\circ \pi(t)) \le 1/t$. But note that $(1,\ldots,1)^T$ is a right-eigenvector of $\mathbf{1}\circ \pi(t)$, with eigenvalue $\Phi(t)$. Therefore
\begin{equation}\label{eq:upbdPhi}\Phi(t)\le 1/t.\end{equation}

\par
Now we prove that the proportion of types among the alive vertices converges as $t\rightarrow\infty$. We first state a generalisation of Lemma \ref{muLipprop1} which will be proved in Section \ref{Lipschitzproof}.
\begin{lemma}\label{muLipprop2}Let $\mathbb{A}$ be a compact subset of $\R_{\ge 0}^{k\times k}$ with the property that for any $A\in\mathbb{A}$, the Perron root of $A$ is simple. Then there exists a constant $C(\mathbb{A})<\infty$ such that, for all matrices $A,A'\in\mathbb{A}$,
\begin{equation}\label{eq:mulocLip2}||\mu(A)-\mu(A')||_1 \le C(\mathbb{A})\max_{i,j\in[k]} |A_{i,j}-A'_{i,j}|.\end{equation}
\end{lemma}

\begin{prop}\label{limexistsprop}For any frozen percolation type flow $\pi$, $\lim_{t\rightarrow\infty}\frac{\pi(t)}{\Phi(t)}$ exists..
\begin{proof}Directly from \eqref{eq:flowDE}, $\frac{\mathrm{d}}{\mathrm{d}t}\Phi(t)=-\phi(t)$. Therefore
$$\frac{\mathrm{d}}{\mathrm{d}t} \left( \frac{\pi(t)}{\Phi(t)}\right) \stackrel{\eqref{eq:flowDE}}{=} \frac{\phi(t)}{\Phi(t)} \left( \frac{\pi(t)}{\Phi(t)} - \mu((\kappa+t\mathbf{1})\circ \pi(t))\right).$$
Note that $\frac{\pi(t)}{\Phi(t)}= \mu(t\mathbf{1}\circ \pi(t))$, and so
\begin{equation}\label{eq:dpioverPhidt}\frac{\mathrm{d}}{\mathrm{d}t} \left( \frac{\pi(t)}{\Phi(t)}\right) = \frac{\phi(t)}{\Phi(t)} \Big[ \mu\left( t\mathbf{1}\circ \pi(t) \right) - \mu\left( (\kappa+t\mathbf{1})\circ \pi(t) \right) \Big].\end{equation}

Consider the sets of positive matrices
$$\mathcal{A}:= \left\{t\mathbf{1}\circ \pi(t): t\ge t_c\right\},\quad \mathcal{B}:= \left\{(\kappa+t\mathbf{1})\circ \pi(t):t\ge t_c\right\}.$$
Now, for any $A\in\mathcal{A}\cup\mathcal{B}$,
$$A_{i,j}\le (\kappa_{\max} +t)\pi_j(t)\le (\kappa_{\max} +t)\Phi(t)\le \frac{\kappa_{\max}}{t_c}+1,$$
where the final inequality follows from \eqref{eq:upbdPhi}. Hence matrices in $\mathcal{A}\cup\mathcal{B}$ are bounded in $\R_{\ge 0}^{k\times k}$ and thus the closure $\overline{\mathcal{A}\cup \mathcal{B}}$ is compact. Any matrix in $\overline{\mathcal{A}\cup\mathcal{B}}$ has the property that any row has either all positive entries, or all zero entries, and at least one row has all positive entries. Thus the Perron root of any matrix in $\overline{\mathcal{A}\cup\mathcal{B}}$ is a simple eigenvalue, and Lemma \ref{muLipprop2} applies, with $\mathbb{A}=\overline{\mathcal{A}\cup\mathcal{B}}$. In particular, there exists a constant $C=C(\mathbb{A})$ (depending on $\kappa$ and $\pi(0)$) such that
$$||\mu(A)-\mu(B)||_1 \le C \max_{i,j\in[k]} \left| A_{i,j}-B_{i,j}\right|,\quad A,B\in\overline{\mathcal{A}\cup\mathcal{B}}.$$

So from \eqref{eq:dpioverPhidt},
$$\left|\left| \frac{\mathrm{d}}{\mathrm{d}t} \left( \frac{\pi(t)}{\Phi(t)}\right)\right|\right|_1 \le C\cdot \frac{\phi(t)}{\Phi(t)} \cdot\max_{i,j}\kappa_{i,j}\pi_j(t) \le C \phi(t)\kappa_{\max}.$$

Therefore, if we write $g(t):=\frac{\mathrm{d}}{\mathrm{d}t} \left( \frac{\pi(t)}{\Phi(t)}\right)$, we have
$$\int_{t_c}^\infty ||g(t)||_1 \mathrm{d}t \le C \kappa_{\max} \int_{t_c}^\infty \phi(t)\mathrm{d}t \le C\kappa_{\max}\Phi(0)<\infty,$$
and it follows that $\frac{\pi(t)}{\Phi(t)}$ converges as $t\rightarrow\infty$.\end{proof}

\begin{proof}[Proof of Theorem \ref{flowtimelimitthm}]
It remains to show that the limit given by Proposition \ref{limexistsprop} is positive. For this, we use an argument similar to the proof of Lemma \ref{defnetalemma}, but now using the statement of Theorem \ref{weaklimitthm} to give stronger bounds involving $\Phi$.

\par
Recall that $\pi(0)$ and $\kappa$ are fixed. Now, for each $N\in\N$, we take $N$ IID samples from $\pi(0)$, and let $p^N\in \N_0^k$ be the vector recording the number of occurences of each type. Clearly, by WLLN $p^N/N \stackrel{d}\rightarrow \pi(0)$ as $N\rightarrow\infty$. We will consider coupling frozen percolation processes with initial types given by $p^N$, as $N$ varies.

\par
Fix some sequence $\lambda:\N\rightarrow(0,\infty)$ satisfying the usual critical scaling \eqref{eq:lambdascaling}. Observe that there is a natural coupling between the processes $\mathcal{G}^{N,p^N,\kappa,\lambda(N)}$ and $\mathcal{G}^{N+1,p^{N+1},\kappa,\lambda(N)}$ under which the restriction of the latter to $[N]$ is equal to the former until the first time an edge is added between $N+1$ and an alive vertex in $[N]$. (This time might be zero, if there is such an edge in the initial graph $\mathcal{G}^{N+1,p^{N+1},\kappa,\lambda(N)}$.)

Now fix a time $T>0$. Theorem \ref{weaklimitthm} applies to both sequences of processes $(\mathcal{G}^{N,p^N,\kappa,\lambda(N)})$ and $(\mathcal{G}^{N+1,p^{N+1},\kappa,\lambda(N)})$, since we also have $1/N\ll \lambda(N-1)\ll 1$. While this theorem is stated in terms of convergence in distribution, it also holds in expectation since the processes $\pi^N$ are uniformly bounded in $\R^k$. Thus, for each $i\in[k]$,
\begin{align*}
\pi_i(T)&= \lim_{N\rightarrow\infty}\E{\pi_i^{N+1}(T)}\\
&= \lim_{N\rightarrow\infty}\Prob{\mathrm{type}(N+1)=i,\text{ $N+1$ alive in }\mathcal{G}^{N+1,p^{N+1},\kappa,\lambda(N)}(T)}.
\end{align*}

Although it leads to a weaker bound, it is more convenient to consider the probability that vertex $N+1$ is both alive \emph{and isolated} in $\mathcal{G}^{N+1,p^{N+1},\kappa,\lambda(N)}(T)$. This event is particularly tractable under the coupling proposed above. As long as $N+1$ is isolated, an edge forms between $N+1$ and $[N]$ at rate $\frac{1}{N}\#\{\text{alive vertices in }[N]\}$. So, if $\Phi^N(t)$ remains the proportion of alive vertices in $\mathcal{G}^{N,p^N,\kappa,\lambda(N)}$, we can control the probability that $N+1$ remains isolated in $\mathcal{G}^{N+1,p^{N+1},\kappa,\lambda(N)}$ \emph{conditional} on the evolution of $\mathcal{G}^{N,p^N,\kappa,\lambda(N)}$. That is,
\begin{align*}
&\Prob{\text{$N+1$ alive and isolated in }\mathcal{G}^{N+1,p^{N+1},\kappa,\lambda(N)}(T)\,\Big|\, \mathcal{G}^{N,p^N,\kappa,\lambda(N)}\big|_{[0,T]}}\\
&\quad = \Prob{N+1\text{ isolated in }\mathcal{G}^{N+1,p^{N+1},\kappa,\lambda(N)}(0)}\\
&\qquad \times\Prob{N+1\text{ not struck by lightning on }[0,T]}\times \exp\left(-\int_0^T \Phi^N(s)\mathrm{d}s\right).
\end{align*}

Since the second and third probabilities are independent of the type of $N+1$, we can include this in the calculation. Then,
\begin{align}
&\Prob{\mathrm{type}(N+1)=i,\,\text{$N+1$ alive and isolated in }\mathcal{G}^{N+1,p^{N+1},\kappa,\lambda(N)}(T)\,\Big|\, \mathcal{G}^{N,p^N,\kappa,\lambda(N)}\big|_{[0,T]}}\nonumber\\
&\quad = \Prob{\mathrm{type}(N+1)=i,N+1\text{ isolated in }\mathcal{G}^{N+1,p^{N+1},\kappa,\lambda(N)}(0)}\nonumber\\
&\qquad \times\Prob{N+1\text{ not struck by lightning on }[0,T]}\times \exp\left(-\int_0^T \Phi^N(s)\mathrm{d}s\right).\label{eq:PN+1isolated}
\end{align}

Only the third of these terms is random. We now consider its expectation. Note that the map $f \mapsto \exp\left(-\int_0^T f(s)\mathrm{d}s\right)$ from $C_b([0,T])$ to $\R$ is continuous with respect to the uniform topology on $[0,T]$. Since $\Phi^N\stackrel{d}\rightarrow \Phi$ uniformly on $[0,T]$, it follows that
$$\lim_{N\rightarrow\infty}\E{\exp\left(-\int_0^T \Phi^N(s)\mathrm{d}s\right)} = \exp\left(-\int_0^T \Phi(s)\mathrm{d}s\right)\stackrel{\eqref{eq:upbdPhi}}\ge  \exp\left(-1-\int_1^T \frac{\mathrm{d}s}{s} \right)=\frac{1}{Te}.$$

So, from \eqref{eq:PN+1isolated} and the law of total probability,
\begin{align}
\pi_i(T)&\ge \limsup_{N\rightarrow\infty}\,\Prob{\mathrm{type}(N+1)=i,\text{ $N+1$ alive and isolated in }\mathcal{G}^{N+1,p^{N+1},\kappa,\lambda(N)}(T)}\nonumber\\
&\ge \left[\lim_{N\rightarrow\infty} \frac{p_i^{N+1}}{N+1}e^{-\kappa_{\max}}\right] \left[\lim_{N\rightarrow\infty} e^{-\lambda(N+1)T}\right] \lim_{N\rightarrow\infty}\E{\exp\left( -\int_0^T \Phi^N(s)\mathrm{d}s\right)}\nonumber\\
&\ge \pi_i(0)e^{-\kappa_{\max}}\cdot \frac{1}{Te}.\label{eq:lowbdpiT}
\end{align}

Combining \eqref{eq:upbdPhi} and \eqref{eq:lowbdpiT}, we obtain
$$\frac{\pi_i(T)}{\Phi(T)} \ge \frac{\pi_i(0) e^{-(\kappa_{\max}+1)}/T}{1/T}=\pi_i(0)e^{-(\kappa_{\max}+1)},$$
and thus $\lim_{T\rightarrow\infty}\frac{\pi(T)}{\Phi(T)}$ has positive components.
\end{proof}
\end{prop}

\section{Proofs of technical lemmas}\label{technicalproofs}

\subsection{Matrices and eigenvectors}

\subsubsection{Minorisation and majorisation}\label{majminsection}
We restate and prove Lemma \ref{minorisation}, a minorisation result for kernel-subdistribution pair,s which was used in the proof of Proposition \ref{neversupercritprop}. We also state a corresponding majorisation lemma, which will be used in Section \ref{LDPproof} for the proof of Proposition \ref{LDPprop}.

\begin{lemma*}For any $0<\bar\Lambda<\Lambda$, and $K<\infty$ there exist $M\in \N$, and $\pi^{(1)},\ldots,\pi^{(M)}\in\Pi_{\le 1}$ and kernels $\kappa^{(1)},\ldots,\kappa^{(M)}\in\R_{\ge 0}^{k\times k}$ such that
\begin{itemize}
\item $\rho(\kappa^{(m)}\circ \pi^{(m)})=\bar \Lambda$ for each $m\in[M]$;
\item for any subdistribution $\pi\in\Pi_{\le 1}$ and kernel $\kappa\in[0,K]^{k\times k}$ with $\rho(\kappa\circ \pi)\ge \Lambda$, there is some $m\in[M]$ for which $\pi^{(m)}\le \pi$ and $\kappa^{(m)}\le \kappa$.
\end{itemize}
\begin{remark}
The condition $\kappa_{\max}\le K$ is necessary. Otherwise, consider
$$\kappa_{i,j}=\begin{cases}L&\quad i=j=1\\ \tfrac{1}{L}&\quad \text{otherwise},\end{cases}\qquad \pi=\left(\tfrac{\Lambda}{L},\ldots,\tfrac{\Lambda}{L}\right),$$
and allow $L\rightarrow\infty$.
\end{remark}
\begin{proof}[Proof of Lemma \ref{minorisation}]
The result is clear when
$$\mathbb{A}(K,\Lambda):=\left\{(\kappa,\pi)\in[0,K]^{k\times k}\times \Pi_{\le 1}\,:\, \rho(\kappa\circ\pi)\ge\Lambda\right\}$$
is either empty or consists of one measure-kernel pair.

\par
Otherwise, we view $\rho$ as a continuous function $\R^{k\times k}_{\ge 0}\times \Pi_{\le 1} \to \R_{\ge 0}$ via $\kappa\circ\pi$, and so $\mathbb{A}(K,\Lambda)$  is compact. Now, for any $\kappa,\kappa^0\in\R_{\ge 0}^{k\times k}$, we say
$\kappa\triangleright\kappa^0$ if for all $i,j\in[k]$,
$$\begin{cases}\kappa_{i,j}\ge 0&\quad \text{when }\kappa_{i,j}^0=0\\ \kappa_{i,j}>\kappa^0_{i,j} &\quad \text{when }\kappa^0_{i,j}>0.\end{cases}$$
Then, for any $\kappa^0\in\R_{\ge 0}^{k\times k}$, the set $\{\kappa\in\R_{\ge 0}^{k\times k}\,:\, \kappa\triangleright \kappa^0\}$, is open in the subset topology induced on $\R_{\ge 0}^{k\times k}$. We also define the relation $\triangleright$ on $\R^k$ in an exactly equivalent fashion.

\par
Now, for any $(\kappa^0,\pi^0)\in[0,K]^{k\times k}\times\Pi_{\le 1}$, with $\rho(\kappa^0\circ \pi^0)=\bar \Lambda$, the set
$$\left\{(\kappa,\pi)\in[0,K]^{k\times k}\times \Pi_{\le 1}\,:\, \rho(\kappa\circ \pi)>\tfrac{\Lambda+\bar\Lambda}{2},\, \kappa\triangleright \kappa^0,\,\pi\triangleright \pi^0\right\},$$
is open in $[0,K]^{k\times k}\times \Pi_{\le 1}$, and so its restriction to $\mathbb{A}(K,\Lambda)$,
$$N(\kappa^0,\pi^0):=\left\{(\kappa,\pi)\in\R_{\ge 0}^{k\times k}\times \R_{\ge 0}^k\,:\, \rho(\kappa\circ \pi)\ge\Lambda,\, \kappa\triangleright \kappa^0,\,\pi\triangleright \pi^0\right\},$$
is also open in the subset topology induced on $\mathbb{A}(K,\Lambda)$. But for any $(\kappa,\pi)\in \mathbb{A}(K,\Lambda)$, with $\Lambda'=\rho(\kappa\circ\pi)$, we have
$$\rho\left(\sqrt{\tfrac{\bar\Lambda}{\Lambda'}}\kappa\circ \sqrt{\tfrac{\bar\Lambda}{\Lambda'}}\pi\right)=\bar\Lambda,\quad\text{and}\quad (\kappa,\pi)\in N\left(\sqrt{\tfrac{\bar\Lambda}{\Lambda'}}\kappa,\sqrt{\tfrac{\bar\Lambda}{\Lambda'}}\pi\right).$$ Therefore, the sets $N(\kappa^0,\pi^0)$ cover $\mathbb{A}(K,\Lambda)$. Thus there is a finite sub-cover given by some $N(\kappa^{(1)},\pi^{(1)})$, $\ldots$, $N(\kappa^{(M)},\pi^{(M)})$. Certainly if $\pi \triangleright \pi^{(m)}$ and $\kappa\triangleright\kappa^{(m)}$, then $\pi\ge \pi^{(m)}$ and $\kappa\ge \kappa^{(m)}$, as required.
\end{proof}
\end{lemma*}

\begin{lemma}\label{majorisation}
For any $0<\Lambda<\bar\Lambda$ and $\eta\in(0,1)$, there exist $M\in\N$ and $\pi^{(1)},\ldots,\pi^{(M)}\in\Pi_{\le 1}$ and kernels $\kappa^{(1)},\ldots,\kappa^{(M)}\in\R_{\ge 0}^{k\times k}$ such that
\begin{itemize}
\item $\rho(\kappa^{(m)}\circ \pi^{(m)}) =\bar\Lambda$ for each $m\in[M]$;
\item for any subdistribution $\pi\in\Pi_{\le 1}\cap[\eta,1]$ and kernel $\kappa\in\R_{\ge 0}^{k\times k}$, with $\rho(\kappa\circ\pi)\le \Lambda$, there is some $m\in[M]$ for which $\pi\le \pi^{(m)}$ and $\kappa\le \kappa^{(m)}$.
\end{itemize}
\end{lemma}
Since the proof is very similar to that of Lemma \ref{minorisation}, it is omitted.

\subsubsection{Matrix powers and the principal eigenvector}\label{Perronproofs}
The goal of this section is to prove Lemma \ref{convevector}. 

We write $S^{k\times k}([\eta,T])$ for the set of $k\times k$ symmetric matrices with entries in $[\eta,T]$. For $A$ a real symmetric positive matrix, we let $\bar \mu(A)$ be the principal left-eigenvector of $A$, normalised so that $||\bar\mu(A)||_2 = 1$. We will work with $\bar\mu(A)$ in the following result, and convert the statement to the language of $\mu(A)$ (as defined earlier) at the end.
\par
We will also work with $\Pi_{\le 1}\cap [\eta,1]^k$, the set of subdistributions where every component is at least $\eta$.

\begin{lemma}Fix $0<\eta<T<\infty$. Then,
\begin{equation}\label{eq:realsymlimit}\lim_{R\rightarrow\infty}\sup_{A\in S^{k\times k}([\eta,T])}\sup_{v\in\Pi_{\le 1}} \left|\left|\frac{v A^R}{\rho(A)^R} - \langle v,\bar\mu(A)\rangle \bar\mu(A) \right|\right|_1 =0.\end{equation}

\begin{proof}
For a real positive symmetric matrix $A$, we define
$$\Lambda_2(A):= \sup\{|\lambda|: \lambda\text{ an eigenvalue of }A,\, \lambda\ne \rho(A)\},$$
to be the absolute value of the `second-largest' eigenvalue of $A$, which is strictly less than $\rho(A)$. But $\rho(A)$ and $\Lambda_2(A)$ are well-defined and continuous on the compact domain $S^{k\times k}([\eta,T])$. This continuity can be shown by considering the characteristic polynomial of $A$ and applying standard results (see \cite{Zedek65} and references therein) concerning the roots of monic polynomials under continuously varying the coefficients. Then
\begin{equation}\label{eq:supLambda2}\theta(\eta,T):=\sup \left\{\frac{\Lambda_2(A)}{\rho(A)}: A\in S^{k\times k}([\eta,T])\right\} <1.\end{equation}
Now, let $\{\bar\mu(A),\mu^{(2)}(A),\ldots,\mu^{(k)}(A)\}$ be a set of orthonormal eigenvectors of $A$, where $\bar \mu(A)$ corresponds to the Perron root $\rho(A)$. As usual, any $v\in \R^k$ can be expressed as
$$v= \langle v,\bar\mu(A)\rangle \bar \mu(A) + \langle v,\mu^{(2)}(A)\rangle \mu^{(2)}(A) + \ldots + \langle v,\mu^{(k)}(A)\rangle \mu^{(k)}(A),$$
and so
$$\left|\left|\frac{v A^R}{\rho(A)^R} - \langle v,\bar \mu(A)\rangle \bar \mu(A) \right|\right|_1 \le \theta(\eta,T)^R \sum_{i=2}^k \Big| \langle v,\mu^{(i)}(A) \rangle \Big| ||\mu^{(i)}(A)||_1.$$
But since $v\in\Pi_{\le 1}$,
$$\left| \langle v,\mu^{(i)}(A)\rangle \right| \le ||\mu^{(i)}(A)||_1 \le \sqrt{k},$$
by Cauchy--Schwarz, since $||\mu^{(i)}(A)||_2=1$. Therefore
$$\left|\left|\frac{v A^R}{\rho(A)^R} - \langle v,\bar \mu(A)\rangle \bar \mu(A) \right|\right|_1 \le \theta(\eta,T)^R \cdot (k-1)\sqrt{k},$$
and the required result \eqref{eq:realsymlimit} follows.
\end{proof}
\end{lemma}

We can now address the case of Lemma \ref{convevector} where all the matrices $D^{(i)}$ are $\kappa\circ\pi$.
\begin{lemma}\label{unifPerronproj}
Fix $0<\eta<T<\infty$. Then,
\begin{equation}\label{eq:unifperronproj}\lim_{R\rightarrow\infty}\sup_{\substack{\pi\in \Pi_{\le 1}\cap [\eta,1]^k\\\kappa\in[\eta,T]^{k\times k}}}\sup_{v\in\Pi_1} \left|\left|\frac{v (\kappa\circ\pi)^R}{||v(\kappa\circ\pi)^R||_1} - \mu(\kappa\circ\pi) \right|\right|_1 =0.\end{equation}
\begin{remark}The non-uniform version of \eqref{eq:unifperronproj} is due to Perron \cite{Perron07}, and the related limiting matrix is called the \emph{Perron projection}. Similar results appear in the multitype branching process literature, including \cite{JoffeSpitzer67}, for which \cite{AthreyaNey} offers a comprehensive summary.
\end{remark}
\begin{proof}
Instead of considering $\kappa \circ \pi$, we will study $\kappa\bullet \pi$, defined for $\kappa\in \R^{k\times k}, \pi\in \R_+^k$ by
\begin{equation}\label{eq:defnmatrixbullet} \left[\kappa\bullet \pi\right]_{i,j}:= \sqrt{\pi_i}\kappa_{i,j}\sqrt{\pi_j}.\end{equation}
The matrix $\kappa\bullet\pi$ is real and symmetric, which makes a treatment of its spectrum easier. First, we note that if $v$ is \emph{any} left-eigenvector of $\kappa\circ \pi$, with eigenvalue $\lambda$, then
$$\sum_{i=1}^k \left( \frac{v_i}{\sqrt{\pi_i}} \right)[\kappa\bullet \pi]_{i,j} = \sum_{i=1}^k v_i \kappa_{i,j}\sqrt{\pi_j} = \lambda\frac{v_j}{\sqrt{\pi_j}}.$$
That is $(v_i/\sqrt{\pi_i})$ is an eigenvector of $\kappa\bullet\pi$, also with eigenvalue $\lambda$. Therefore the spectrum of $\kappa\circ \pi$ is the same as the spectrum of $\kappa\bullet\pi$. In particular, the Perron roots of $\kappa\circ \pi$ and $\kappa\bullet\pi$ are the same, and $\mu(\kappa\bullet \pi)_i= C \mu(\kappa\circ \pi)_i/\sqrt{\pi_i}$, where $C$ is a positive constant chosen to ensure consistent normalisation.

\par
But then
\begin{equation}\label{eq:circtobullet}[v(\kappa\circ \pi)^R]_j= \left[ \left(\frac{v_1}{\sqrt{\pi_1}},\ldots, \frac{v_k}{\sqrt{\pi_k}}\right)(\kappa\bullet \pi)^R\right]_j \sqrt{\pi_j}.\end{equation}
Note that if $v\in\Pi_1$, then $(\frac{v_1}{\sqrt{\pi_1}},\ldots, \frac{v_k}{\sqrt{\pi_k}})\in \Pi_{\le \eta^{-1/2}}$, and certainly $\kappa\bullet\pi\in S^{k\times k}([\eta^2,T])$. The statement \eqref{eq:realsymlimit} still holds after replacing the supremum over $v\in\Pi_{\le 1}$ with a supremum over $v\in\Pi_{\le \eta^{-1/2}}$. So we can treat the RHS of \eqref{eq:circtobullet}, since
$$\left|\left|\left(\frac{v_1}{\sqrt{\pi_1}},\ldots, \frac{v_k}{\sqrt{\pi_k}}\right)\frac{(\kappa\bullet \pi)^R}{\rho^R} - \left\langle \left(\frac{v_1}{\sqrt{\pi_1}},\ldots, \frac{v_k}{\sqrt{\pi_k}}\right),\bar \mu(\kappa\bullet\pi) \right\rangle \bar\mu(\kappa\bullet\pi) \right|\right|_1\rightarrow 0,$$
as $R\rightarrow\infty$, uniformly across the set of $(\pi,\kappa)$ under consideration, and $v\in\Pi_1$. Note that for each $j\in[k]$, we have $\sqrt{\pi_j}\in[\sqrt{\eta},1]$. Therefore, uniformly in the same sense,
$$\frac{v(\kappa\circ \pi)^R}{\rho^R} \quad\rightarrow\quad \left\langle \left(\frac{v_1}{\sqrt{\pi_1}},\ldots, \frac{v_k}{\sqrt{\pi_k}}\right),\bar \mu(\kappa\bullet\pi) \right\rangle \Big(\bar\mu_1(\kappa\bullet\pi) \sqrt{\pi_1},\ldots,\bar\mu_k(\kappa\bullet\pi) \sqrt{\pi_k}\Big),$$
as $R\rightarrow\infty$, and so also
$$\left|\left|\frac{v(\kappa\circ \pi)^R}{\rho^R}\right|\right|_1 \quad\rightarrow\quad \left|\left|\left\langle \left(\frac{v_1}{\sqrt{\pi_1}},\ldots, \frac{v_k}{\sqrt{\pi_k}}\right),\bar \mu(\kappa\bullet\pi) \right\rangle \Big(\bar\mu_1(\kappa\bullet\pi) \sqrt{\pi_1},\ldots,\bar\mu_k(\kappa\bullet\pi) \sqrt{\pi_k}\Big)\right| \right|_1.$$

We want to show that this limiting quantity has a positive lower bound, so that we can take a limit of the quotients $\frac{v(\kappa\circ \pi)^R}{||[v(\kappa\circ \pi)^R]||_1}$. Since $\kappa\bullet\pi \in S^{k\times k}([\eta^2,T])$, we have $\rho(\kappa\bullet \pi) \le kT$ from \eqref{eq:CWformula}. Then, we can bound the components of $\mu(\kappa\bullet \pi)$ from below explicitly as
$$\mu_j = \frac{1}{\rho(\kappa\bullet\pi)} \sum_{i\in[k]} \mu_i [\kappa\bullet \pi]_{i,j} \ge \frac{1}{kT} \sum_{i\in[k]} \mu_i \eta^2 = \frac{\eta^2}{kT}.$$
Note also that $\bar \mu(\kappa\bullet\pi)\ge \mu(\kappa\bullet\pi)$. So, since $v\in\Pi_1$ and $\sqrt{\pi_i}\le 1$, we obtain
$$\left\langle \left(\frac{v_1}{\sqrt{\pi_1}},\ldots, \frac{v_k}{\sqrt{\pi_k}}\right),\bar \mu(\kappa\bullet\pi) \right\rangle
\ge \frac{\eta^2}{kT},$$
and, since $\sqrt{\pi_i}\ge \sqrt{\eta}$, we also obtain
$$\left|\left| \Big(\bar\mu_1(\kappa\bullet\pi) \sqrt{\pi_1},\ldots,\bar\mu_k(\kappa\bullet\pi) \sqrt{\pi_k}\Big)\right| \right|_1\ge \sqrt{\eta}.$$

Thus
$$\left|\left|\left\langle \left(\frac{v_1}{\sqrt{\pi_1}},\ldots, \frac{v_k}{\sqrt{\pi_k}}\right),\bar \mu(\kappa\bullet\pi) \right\rangle \Big(\bar\mu_1(\kappa\bullet\pi) \sqrt{\pi_1},\ldots,\bar\mu_k(\kappa\bullet\pi) \sqrt{\pi_k}\Big)\right| \right|_1 \ge \frac{\eta^{5/2}}{kT}>0.$$

So we obtain
$$\frac{v(\kappa\circ \pi)^R}{||v(\kappa\circ \pi)^R||_1} \quad\rightarrow\quad \frac{ \Big(\bar\mu_1(\kappa\bullet\pi) \sqrt{\pi_1},\ldots,\bar\mu_k(\kappa\bullet\pi) \sqrt{\pi_k}\Big)}{\left|\left| \Big(\bar\mu_1(\kappa\bullet\pi) \sqrt{\pi_1},\ldots,\bar\mu_k(\kappa\bullet\pi) \sqrt{\pi_k}\Big)\right|\right|_1}.$$

But
$$\Big(\bar\mu_1(\kappa\bullet\pi) \sqrt{\pi_1},\ldots,\bar\mu_k(\kappa\bullet\pi) \sqrt{\pi_k}\Big) \propto \Big(\mu_1(\kappa\bullet\pi) \sqrt{\pi_1},\ldots,\mu_k(\kappa\bullet\pi) \sqrt{\pi_k}\Big)\propto \mu(\kappa\circ \pi),$$
so we have shown
$$\frac{v(\kappa\circ \pi)^R}{||[v(\kappa\circ \pi)^R]||_1}\quad \rightarrow\quad \mu(\kappa\circ \pi),$$
as $R\rightarrow\infty$, uniformly across $v\in\Pi_1$, and $\kappa\in[\eta,T]^{k\times k}$ and $\pi\in\Pi_1$ such that $\pi_i\ge \eta$, exactly as required.
\end{proof}
\end{lemma}

Recalling the definition \eqref{eq:defnBthetaA}
$$\mathbb{B}_\theta(A):= \{B\in\R_+^{k\times k} \, :\, |B_{i,j}-A_{i,j}|\le \theta,\,\forall i,j\in[k] \},$$
of the set of positive kernels whose entries differ from those of $A$ by at most $\theta$, we can now prove Lemma \ref{convevector}.

\begin{proof}[Proof of Lemma \ref{convevector}]
For now we fix $\theta\in(0,\eta^2)$, and will take this small enough at the end. Then, for any $A\in[\eta^2,T]^{k\times k}$ and $D^{(1)},\ldots,D^{(R)}\in \mathbb{B}_{\theta}(A)$,
$$(D^{(1)}D^{(2)}\ldots D^{(R)})_{i,j} = \sum_{i=i_0,i_1,\ldots,i_R=j} \prod_{r=1}^R D^{(r)}_{i_{r-1},i_r} \le \sum_{i=i_0,i_1,\ldots,i_R=j}\prod_{r=1}^R ( A_{i_{r-1},i_r} + \theta).$$
Therefore, defining $\bar D:=D^{(1)}\cdots D^{(R)}$, since $\theta<\eta^2<\eta<T$,
$$(\bar D- A^R)_{i,j}\le k^{R-1}(2^R-1) \cdot \theta T^{R-1}.$$

Similarly, for a lower bound
\begin{align*}
(A^R-\bar D)_{i,j}&\le  \sum_{i=i_0,i_1,\ldots,i_R=j} \prod_{r=1}^R A_{i_{r-1},i_r} - \sum_{i=i_0,i_1,\ldots,i_R=j} \prod_{r=1}^R (A_{i_{r-1},i_r}-\theta).\\
\intertext{The RHS is a polynomial in $\theta$ whose coefficients have alternating signs, and so we can bound using the associated polynomial with every coefficient positive:}
(A^R-\bar D)_{i,j} &\le  \sum_{i=i_0,i_1,\ldots,i_R=j}\prod_{r=1}^R ( A_{i_{r-1},i_r} + \theta) - \sum_{i=i_0,i_1,\ldots,i_R=j}\prod_{r=1}^R A_{i_{r-1},i_r}.
\end{align*}
That is,
$$\Big|(\bar D- A^R)_{i,j}\Big|\le k^{R-1}(2^R-1) \cdot \theta T^{R-1}.$$

Since the fraction in \eqref{eq:prodapproxmu} is unchanged under positive scalar multiplication of $v$, it suffices to show the result for $v\in\Pi_1$. For any $v\in\Pi_1$:
$$\big|\big|v \bar D - v A^R\big|\big|_1 \le k^R(2^R-1) \cdot \theta T^{R-1}.$$

For \eqref{eq:prodapproxmu} we need to control the distance between the normalised vectors instead. Observe first that for each $i$, $||v D^{(i)}||_1\in [k(\eta-\theta),k(T+\theta)]$, whenever $v\in \Pi_1$. Thus $||v\bar D||_1,||v A^R||_1\in[(k(\eta-\theta))^R,(k(T+\theta))^R]$. From the triangle inequality,
\begin{align}
\left|\left| \frac{v\bar D}{||v\bar D||_1} - \frac{v A^R}{||v A^R||_1} \right|\right|_1 &\le \left|\left| \frac{v\bar D-v A^R}{||v\bar D||_1} \right|\right|_1 + \left|\left| \frac{v  A^R}{||v \bar D||_1} - \frac{v  A^R}{||v  A^R||_1} \right|\right|_1\nonumber\\
&\le \frac{||v\bar D-v A^R||_1}{||v\bar D||_1} + ||v  A^R||_1 \left|\frac{1}{||v\bar D||_1}-\frac{1}{||v  A^R||_1} \right|\nonumber\\
&\le \frac{||v\bar D-v A^R||_1}{||v\bar D||_1} + ||v  A^R||_1  \frac{\left|||v\bar D||_1 -||v  A^R||_1\right|}{||v\bar D||_1||v  A^R||_1}  \nonumber\\
&\le  \frac{2||v\bar D-v A^R||_1}{||v\bar D||_1}, \label{eq:l1triangleineq}
\intertext{so for $v\in\Pi_1$,}
\left|\left| \frac{v\bar D}{||v\bar D||_1} - \frac{v A^R}{||v A^R||_1} \right|\right|_1&\le \frac{2(2^R-1)\cdot \theta T^{R-1}}{(\eta-\theta)^R} .\label{eq:Dbarineq}
\end{align}
Finally, we take $A=\kappa\circ\pi$. Lemma \ref{unifPerronproj} determines a value of $R$ such that for all $\kappa\in[\eta,T]^{k\times k}$, $\pi\in\Pi_{\le 1}\cap[\eta,1]^k$, and $v\in\Pi_1$, taking $A=\kappa\circ \pi$, we have
$$\left|\left| \frac{v A^R}{||v A^R||_1} - \mu(A) \right|\right|_1 \le \frac{\delta}{2}.$$
Combining with \eqref{eq:Dbarineq} and taking $\theta$ small enough, the result follows after extending from $v\in\Pi_1$ to $v\in\R^k_{\ge 0}\backslash\{0\}$.
\end{proof}

\subsubsection{Lipschitz property of the principal eigenvector}\label{Lipschitzproof}
We restate Lemma \ref{muLipprop2}, which is also a generalisation of Lemma \ref{muLipprop1}, and prove it by adapting a very similar result from \cite{MagnusNeudecker}.
\begin{lemma*}Let $\mathbb{A}$ be a compact subset of $\R_{\ge 0}^{k\times k}$ with the property that for any $A\in\mathbb{A}$, the Perron root of $A$ is simple. Then there exists a constant $C(\mathbb{A})<\infty$ such that, for all matrices $A,A'\in\mathbb{A}$,
$$||\mu(A)-\mu(A')||_1 \le C(\mathbb{A})\max_{i,j\in[k]} |A_{i,j}-A'_{i,j}|.$$

\begin{proof}[Proof of Lemma \ref{muLipprop2}]
We use a related result about the local smoothness of eigenvalues and eigenvectors as the matrix varies in the neighbourhood of a matrix with a simple eigenvalue.

\begin{theorem*}[\cite{MagnusNeudecker}, \textsection3.9, Theorem 8]Let $\rho_0$ be a simple eigenvalue of a matrix $A_0\in \mathbb{C}^{k\times k}$, and $\mu_0$ an associated left-eigenvector satisfying $\mu_0^\dagger \mu_0=1$. Then, there exists a neighbourhood of $N(A_0)\subset \mathbb{C}^{k\times k}$ of $A_0$, and functions $\rho:N(A_0)\rightarrow\mathbb{C}$ and $\bar\mu :N(A_0)\rightarrow \mathbb{C}^k$, such that
\begin{itemize}
\item $\rho(A_0)=\rho_0$ and $\bar\mu(A_0)=\mu_0$,
\item $\bar\mu(A)A=\rho(A)\bar\mu(A)$, and $\mu_0^\dagger\bar\mu(A)=1$ for all $A\in N(A_0)$,
\item $\rho$ and $\bar\mu$ are infinitely differentiable on $N(A_0)$.
\end{itemize}
\end{theorem*}

\medskip
If we take $A_0\in \mathbb{A}$, and $\mu_0=\mu(A)$, then it follows that $\bar \mu$ is locally Lipschitz as a function $N(A_0)\cap\mathbb{A}\to \R_+^k$. In this statement $\bar\mu(A)$ differs from our definition $\mu(A)$ by a normalising factor, that varies in $N(A_0)$. However, the choice $\bar\mu$ satisfies $\bar \mu(A_0)^T \bar \mu(A_0)=1$, and so for each $i\in[k]$, $\bar \mu_i(A_0)\le 1$. Therefore, for any $A\in N(A_0)\cap \mathbb{A}$,
\begin{equation}\label{eq:barmuge1}||\bar \mu(A)||_1 \ge \bar \mu(A_0)^T\bar \mu(A)=1.\end{equation}

Now, for $A,A'\in N(A_0)\cap \mathbb{A}$,
\begin{align*}
||\mu(A)-\mu(A')||_1 &= \left|\left| \frac{\bar\mu(A)}{||\bar\mu(A)||_1} - \frac{\bar \mu(A')}{||\bar \mu(A')||_1}\right|\right|.\\
\intertext{Therefore, as in \eqref{eq:l1triangleineq},}
||\mu(A)-\mu(A')||_1&\le \frac{2||\bar\mu(A)-\bar\mu(A')||_1}{||\bar\mu(A)||_1}\stackrel{\eqref{eq:barmuge1}}\le 2||\bar\mu(A)-\bar\mu(A')||_1.
\end{align*}

Since $\bar \mu$ is locally Lipschitz on $N(A_0)\cap \mathbb{A}$, it follows that $\mu$ is also locally Lipschitz on $N(A_0)\cap \mathbb{A}$. Thus $\mu$ is Lipschitz on $\mathbb{A}$ by compactness.
\end{proof}
\end{lemma*}

\subsection{Proof of Proposition \ref{LDPprop}}\label{LDPproof}

We restate Proposition \ref{LDPprop}, concerning uniform exponential tail bounds for the size of the largest component in a near critical IRG.

\begin{prop*}
Fix $\eta,\epsilon\in(0,1/2)$. Then there exist $N_0=N_0(\epsilon,\eta)\in\N$ and constants $M=M(\eta)<\infty$ and $\Gamma=\Gamma(\epsilon,\eta)>0$, such that for any $N\ge N_0$ and
\begin{itemize}
\item any kernel $\kappa\in[\eta,\infty)^{k\times k}$;
\item any vector $p\in\N^k$ such that $\sum p_i=N$ and $p_i/N\ge \eta$ for each $i$;
\item and such that the eigenvalue condition $\rho(\kappa\circ p/N)\le 1+\epsilon$ is satisfied;
\end{itemize}
the following bound on the largest component in $G^N(\rho,\kappa)$ holds:
\begin{equation}
\Prob{L_1\left(G^N(p,\kappa) \right) \ge M\epsilon N}\le \exp(-\Gamma N).
\end{equation}

\begin{proof}
The proof has three stages. First we use the majorisation result of Lemma \ref{majorisation} to reduce the problem (for given $\epsilon>0$) to a finite collection of $(\pi,\kappa)$s. Then, we argue that the survival probabilities $\zeta^{\pi,\kappa}$, as introduced in Section \ref{expbdsetupsection}, are uniformly bounded by $M\epsilon$. Recall that these survival probabilities are related to the \emph{typical} size of the giant component of $G^N(N\pi,\kappa)$ by Proposition \ref{BJRL1enhanced}. Finally, we use estimates of \cite{BJR10} for deviations above its typical size of the largest component of IRGs with a fixed kernel.

The second of these stages is most cumbersome, although the estimates we prove here are far from optimal. This will be treated first, after a straightforward preliminary lemma.

\begin{lemma}\label{alphaelemma}
For every $\epsilon>0$, define $\alpha_\epsilon$ such that
\begin{equation}\label{eq:defnalphae}\alpha_\epsilon = 1-e^{-(1+\epsilon)\alpha_\epsilon}.
\end{equation}
Then a) $\alpha_\epsilon\le 2\epsilon$; b) for all $x\ge \alpha_\epsilon$ we have $x\ge 1-e^{-(1+\epsilon)x}$.
\begin{proof}
Rearranging \eqref{eq:defnalphae} yields $-(1+\epsilon)\alpha_\epsilon = \log(1-\alpha_\epsilon)\le -\alpha_\epsilon-\frac{\alpha_\epsilon^2}{2}$, from which $\alpha_\epsilon\le 2\epsilon$ follows. Set $f(x)=1-e^{-x}$. Then b) is a consequence of $f'\le 1$ on $[0,\infty)$.
\end{proof}
\end{lemma}

\begin{lemma}\label{O1boundzeta}
Fix $\eta,\epsilon\in(0,1)$. Then there exists $M=M(\eta)<\infty$ such that for all subdistributions $\pi\in \Pi_{\le 1}\cap [\eta,1]^k$ and kernels $\kappa\in[\eta,\infty)^{k\times k}$ satisfying $\rho(\pi\circ\kappa)\le 1+\epsilon$,
\begin{equation}\label{eq:zetaisOe}
\pi\cdot\zeta^{\pi,\kappa}\le M\epsilon.\end{equation}
\begin{proof}
Motivated by \eqref{eq:fixedpointzeta}, define the function $F_{\pi,\kappa}:[0,\infty)^k\rightarrow[0,\infty)^k$ by
\begin{equation}\label{eq:defnFpikappa}
[F_{\pi,\kappa}(\mathbf{x})]_i = 1-\exp\left( - [(\kappa\circ \pi)\mathbf{x}]_i\right).
\end{equation}
Then, as shown in \cite{BJRinhomog}, $\zeta^{\pi,\kappa}$ is the maximal fixed point of $F_{\pi,\kappa}$. We also introduce the \emph{right-eigenvector} $\nu=\nu(\kappa\circ\pi)$ of $\kappa\circ\pi$, normalised such that $\nu\cdot\pi=1$. Since $\kappa$ is symmetric, we have $\nu_i=\mu_i/\pi_i$. (Note $\pi$ is positive.) Furthermore, since both eigenvectors are continuous functions of $(\pi,\kappa)$, there exists a constant $c=c(\epsilon,\eta)>0$ such that $\mu_i,\nu_i\ge c$ whenever $(\pi,\kappa)$ are in the range specified.

Define the vector $\theta= \alpha_{\rho(\kappa\circ\pi)-1}\nu/\nu_{\min} $, which, using a) of Lemma \ref{alphaelemma}, satisfies
\begin{equation}\label{eq:thetabds}
\pi\cdot \theta = \frac{\alpha_{\rho(\kappa\circ\pi)-1}}{\nu_{\min}}||\mu||_1\le \frac{\alpha_{\rho(\kappa\circ\pi)-1}}{c}\le \frac{2\epsilon}{c},\qquad \theta_i\ge \alpha_{\rho(\kappa\circ\pi)-1},\; \forall i\in[k].
\end{equation}
Now, set $\zeta^m:= F^{(m)}_{\pi,\kappa}(\theta)$, for $m\ge 0$. Using b) of Lemma \ref{alphaelemma} and the fact that $F_{\pi,\kappa}$ is weakly-decreasing, we have $\theta=\zeta^0\ge \zeta^1\ge \zeta^2\ldots>0$. Thus $\bar\zeta:= \lim_{m\rightarrow\infty} \zeta^m$ exists, and is non-negative, and satisfies $F_{\pi,\kappa}(\bar\zeta)=\bar\zeta$. Since $\kappa>0$, it satisfies the irreducibility condition required for Lemma 5.9 of \cite{BJRinhomog}, which establishes that the \emph{only} fixed points of $F_{\pi,\kappa}$ are $0$ and $\zeta^{\pi,\kappa}$.

Suppose that $\rho>1$ and $\bar\zeta=\lim_{m\rightarrow\infty}\zeta^m=0$. Then, as $m\rightarrow\infty$,
\begin{equation}\label{eq:zetam+1}\zeta^{m+1}=F_{\pi,\kappa}(\zeta^m) = \left( 1-e^{-[(\kappa\circ\pi)\zeta^m]_i},\,i\in[k]\right) = (\kappa\circ \pi)\zeta^m + O\left(||\zeta^m||^2\right).\end{equation}

In particular, for large enough $m$, we have
$$\zeta^{m+1} \ge (\kappa\circ \pi)\zeta^m - \tfrac{\rho-1}{2}\zeta^m,$$
and so, taking a product with $\mu$ (whose entries, recall, are positive),
$$\mu\cdot \zeta^{m+1} \ge \mu(\kappa\circ \pi)\zeta^m - \tfrac{\rho-1}{2}\zeta^m = \left( \rho - \tfrac{\rho-1}{2}\right)\mu\cdot\zeta^m > \mu\cdot\zeta^m.$$
But clearly $(\mu\cdot \zeta^m)$ is also a decreasing sequence, so this is a contradiction when $\rho>1$. So in all cases, $\bar\zeta=\zeta^{\pi,\kappa}$. But $\bar\zeta\le \theta$, and so from \eqref{eq:thetabds}, $\pi\cdot\zeta^{\pi,\kappa}\le \frac{2\epsilon}{c}$, as required.
\end{proof}
\end{lemma} 

We can now continue with the proof of Proposition \ref{LDPprop}. We first use Lemma \ref{majorisation} to exhibit a collection of subdistribution-kernel pairs $(\pi^{(1)},\kappa^{(1)}),\ldots,(\pi^{(L)},\kappa^{(L)})$ with $\rho(\kappa^{(\ell)}\circ \pi^{(\ell)})=1+2\epsilon$, at least one of which dominates any pair $(\pi,\kappa)$ satisfying the conditions of Proposition \ref{LDPprop}, particularly $\rho(\kappa\circ\pi)\le 1+\epsilon$.

As a result, for any suitable sequence $p^N$, for large enough $N$, we have
\begin{equation}\label{eq:stochdomforLDP}L_1\left(G^N(p^N,\kappa)\right)\le_{\mathrm{st}} L_1\left(G^N\left(\fl{N\pi^{(\ell)}},\kappa^{(\ell)}\right)\right),\end{equation}
for some $\ell\in[L]$. However, for each $\ell\in[L]$, Theorem 1.4 of \cite{BJR10} asserts the existence of $\Gamma^{(\ell)}$ such that
\begin{equation}\label{eq:BJRLDP}
\Prob{\left|L_1\left(G^N\left(\fl{N\pi^{(\ell)}},\kappa^{(\ell)}\right)\right) - \pi\cdot \zeta^{\pi,\kappa}N\right| \ge \epsilon N}\le e^{-\Gamma^{(\ell)}N},
\end{equation}
when $N$ is large enough. Taking $M=M(\eta)$ as in Lemma \ref{O1boundzeta}, using \eqref{eq:zetaisOe} gives
$$\Prob{L_1\left(G^N\left(\fl{N\pi^{(\ell)}},\kappa^{(\ell)}\right)\right)\ge (2M+1)\epsilon N}\le e^{-\Gamma^{(\ell)}N}.$$
Using the stochastic domination \eqref{eq:stochdomforLDP} and setting $\Gamma:= \min_{\ell\in[L]}\Gamma^{(\ell)}$, the statement of Proposition \ref{LDPprop} follows after replacing $M$ with $2M+1$.
\end{proof}
\end{prop*}

\bibliographystyle{abbrv}

\end{document}